\documentclass[11pt, a4paper, twoside,leqno]{amsart}
\usepackage[centering, totalwidth = 380pt, totalheight = 600pt]{geometry}
\usepackage{amssymb, amsmath, amsthm, listings, xspace}
\usepackage{enumerate, microtype, stmaryrd, url} 
\usepackage[latin1]{inputenc}
\usepackage{tikz-cd}
\usepackage[arrow, matrix, tips, curve, graph, rotate, 2cell]{xy}
\UseTwocells
\usepackage[retainorgcmds]{IEEEtrantools}
 \usepackage{color}
\definecolor{daRemarkgreen}{rgb}{0,0.45,0}
\usepackage[colorlinks,citecolor=daRemarkgreen,linkcolor=daRemarkgreen]{hyperref}
\usepackage{lscape}
\SelectTips{cm}{10}
\newcommand{\cat}[1]{\mathbf{#1}}

\makeatletter
\def\matrixobject@{%
   \edef \next@{={\DirectionfromtheDirection@ }}%
   \expandafter \toks@ \next@ \plainxy@
   \let\xy@@ix@=\xyq@@toksix@
   \xyFN@ \OBJECT@}
\let\xy@entry@@norm=\entry@@norm
\def\entry@@norm@patched{%
   \let\object@=\matrixobject@
   \xy@entry@@norm }
\AtBeginDocument{\let\entry@@norm\entry@@norm@patched}
\makeatother

\newdir{(}{{}*!<0em,-.14em>-\cir<.14em>{l^r}}
\newdir{ (}{{}*!/-5pt/\dir{(}}
\newdir{ >}{{}*!/-5pt/\dir{>}}

\theoremstyle{definition}

\numberwithin{equation}{section}
\theoremstyle{plain}
\newtheorem{Theorem}{Theorem}[section]

\newtheorem{Cor}[Theorem]{Corollary}
\newtheorem{Prop}[Theorem]{Proposition}
\newtheorem{Lemma}[Theorem]{Lemma}

\theoremstyle{definition}

\newtheorem{Def}[Theorem]{Def}

\newtheorem{Examples}[Theorem]{Examples}
\newtheorem{Example}[Theorem]{Example}
\newtheorem{Remark}[Theorem]{Remark}

\renewcommand{\phi}{\varphi}
\newcommand{\A}{{\mathcal A}}
\newcommand{\B}{{\mathcal B}}
\newcommand{\C}{{\mathcal C}}

\newcommand{\F}{{\mathcal F}}
\newcommand{\G}{{\mathcal G}}

\newcommand{\J}{{\mathcal J}}

\newcommand{\V}{{\mathcal V}}

\newcommand{\X}{{\mathcal X}}

\newcommand{\atwo}{{\mathbf 2}}

\newcommand{\Cat}{{\cat{Cat}}}
\newcommand{\TOb}{{\cat{TOb}}}
\newcommand{\CatGph}{{\cat{Cat}\textnormal{-}\mathbf{Gph}}}

\newcommand{\CAT}{{\cat{CAT}}}
\newcommand{\twocat}{{\textnormal{2-}\cat{Cat}}}
\newcommand{\VCAT}{\V\textnormal{-}{\cat{CAT}}}
\newcommand{\ff}{\Phi{\textnormal{-}\mathbf{Lim}}}

\newcommand{\Twocat}{{\textnormal{2-}\cat{CAT}}}

\newcommand{\Bicat}{{\mathbf{Bicat}}}
\newcommand{\Set}{{\cat{Set}}}
\newcommand{\Gph}{{\cat{Gph}}}
\newcommand{\cSk}{{\cat{C\textnormal{-}Sk}}}
\newcommand{\tSk}{{\cat{T\textnormal{-}Sk}}}
\newcommand{\Lex}{{\cat{Lex}}}
\newcommand{\Acc}{{\textbf{ACC}}}
\newcommand{\Adj}{{\cat{Adj}}}
\newcommand{\Ref}{{\cat{Ref}}}
\newcommand{\Coref}{{\cat{Coref}}}
\newcommand{\AdjEq}{{\cat{AdjEq}}}
\newcommand{\Reg}{{\cat{Reg}}}

\newcommand{\Lali}{{\cat{Lali}}}

\newcommand{\SMonCat}{{\cat{SMonCat}}}
\newcommand{\MonCat}{{\cat{MonCat}}}
\newcommand{\UMonCat}{{\cat{UMonCat}}}
\newcommand{\PsTAlg}{{\cat{Ps}\textnormal{-}\cat{T}\textnormal{-}\cat{Alg}}}
\newcommand{\LaxTAlg}{{\cat{Lax}\textnormal{-}\cat{T}\textnormal{-}\cat{Alg}}}
\newcommand{\ColaxTAlg}{{\cat{Colax}\textnormal{-}\cat{T}\textnormal{-}\cat{Alg}}}
\newcommand{\Lp}{\textbf{LP}}
\newcommand{\Lpm}{\textbf{LP}_{\mathcal M}}
\newcommand{\TAlg}{{\cat{T}\textnormal{-}\cat{Alg}}}

\newcommand{\Alg}{\cat{Alg}}

\newcommand{\ALim}{{A\textbf{-Lim}}}

\newcommand{\cc}{\ensuremath{\mathcal C}\xspace}

\newcommand{\cv}{\ensuremath{\mathcal V}\xspace}

\def\blue{\color{blue}}
\def\black{\color{black}}

\begin{document}
\leftmargini=2em \title{Accessible aspects of 2-category theory}
\author{John Bourke}
\address{Department of Mathematics and Statistics, Masaryk University, Kotl\'a\v rsk\'a 2, Brno 61137, Czech Republic}
\email{bourkej@math.muni.cz}
\thanks{The author acknowledges the support of the Grant Agency of the Czech Republic under the grant 19-00902S}
\keywords{Accessible category, $2$-category, pseudomorphism.}
\date{\today}
\maketitle

\maketitle
\begin{abstract}
Categorical structures and their pseudomaps rarely form locally presentable $2$-categories in the sense of $\Cat$-enriched category theory.  However, we show that if the categorical structure in question is sufficiently weak (such as the structure of monoidal, but not strict monoidal, categories) then the $2$-category in question is accessible. 
Furthermore, we explore the flexible limits that such $2$-categories possess and their interaction with filtered colimits.
\end{abstract}

\section{Introduction}

Accessible categories, first studied in depth by Makkai and Par{\'e} in the late 1980s \cite{Makkai1989Accessible},  describe the categorical properties of structures definable using infinitary first order logic.  Locally presentable categories, those accessible ones which are complete, capture categories of essentially algebraic structures and were introduced earlier by Gabriel and Ulmer \cite{Gabriel1971Lokal}.  In addition to their strong connections with logic and algebra, accessible and locally presentable categories have also come to play important roles in homotopy theory --- see, for instance, \cite{Beke2000Sheafifiable} and \cite{Lurie}.


On the other hand two-dimensional universal algebra, as investigated by the Australian school of category theory in the late 1980's \cite{Bird1989Flexible, Blackwell1989Two-dimensional}, concerns categories, rather than sets, with algebraic structure.  These form $2$-categories rather than mere categories.  A good example is the $2$-category $\MonCat_p$ of monoidal categories, strong monoidal functors and monoidal natural transformations.  The subscript $p$ indicates that, as morphisms, one normally wishes to consider not the strict structure preserving maps, but the pseudomaps, which preserve the algebraic structure only up to coherent isomorphism.  Focusing on these pseudomorphisms makes the subject subtler --- for instance, the $2$-category $\MonCat_p$ does not admit all equalisers, and so is is certainly not locally presentable as a category.  However, it does admit \emph{flexible limits} \cite{Bird1989Flexible}, which include split idempotents, and these generate all bilimits --- that is, weak $2$-dimensional limits.


The purpose of the present paper is to draw a connection between accessibility and two-dimensional universal algebra.  In particular we will show that many $2$-categories of categories with structures and their pseudomaps, such as $\MonCat_p$, are actually \emph{accessible in the classical sense}.  We will see that accessibility is intimately related to the structures involved being \emph{weak} (a.k.a.~\emph{cofibrant} or \emph{flexible}).  For instance, the $2$-category of strict monoidal categories and strong monoidal functors fails to be accessible --- in fact, its idempotent completion is $\MonCat_p$.  Furthermore, we will explore the interaction between the filtered colimits and flexible limits that $2$-categories such as $\MonCat_p$ possess. 

One exciting aspect of this is that much of the story appears to hold also for higher dimensional categories.  In parallel work with Lack and Vok\v{r}\'{i}nek \cite{Bourke2020Adjoint} we move in this direction, showing that a variety of simplicially enriched categories of $(\infty,1)$-categories with structure, and the pseudomaps between them, are accessible with flexible limits.  This fits directly into the program of Riehl and Verity \cite{Riehl2015,Riehl2017Fibrations, Riehl2019Elements} on $\infty$-cosmoi, which uses classical enriched category theory \cite{Kelly1982Basic} as a means to develop the theory of $\infty$-categories.  Furthermore in \cite{Bourke2020Adjoint} we develop weak adjoint functors theorems and results on the existence of weak colimits in the setting of categories enriched in a monoidal model category, with new applications both in the present $2$-categorical and the $\infty$-categorical settings.

Before giving a summary of the results of the paper, let me say a few words about its genesis.  In early 2015 I read Makkai's papers on \emph{generalised sketches} \cite{Makkai1997Sketches1, Makkai1997Sketches2} and was surprised to observe that the results there already imply that several of the key examples of $2$-categories of pseudomaps are accessible.  Lack and Rosicky`s paper \emph{Enriched weakness} \cite{Lack2012Enriched}, another key influence, furthermore implied that bicategories and normal pseudofunctors form an accessible category.  Discussing this with Makkai in Budapest in the summer of 2015 we began to work out some of the general results.   One of the key technical results, Proposition~\ref{prop:acc2}, is due to Makkai in the case $\C = \Cat$. His main aim at that time was to establish a general result for two-dimensional limit sketches.  This is touched on in the final section, and we will give a more detailed treatment of two-dimensional limit theories in a followup paper.

%

Let me now give a more detailed outline of the results of the present paper.  In keeping with how it came about, Section~\ref{sect:sketches} is devoted to Makkai's generalised sketches.  We use these to describe the category $\TOb$ of categories with a terminal object and terminal object preserving functors as the full subcategory of injectives in a presheaf category --- in particular $\TOb$ is then accessible.  This baby example easily extends to cover other doctrines, such as categories with limits and/or colimits possibly satisfying exactness properties, but not so easily to examples such as monoidal categories.  

In order to make the paper as accessible as possible, we devote Section~\ref{sect:background} to background material on $2$-category theory.  Section~\ref{sect:lp} introduces a class of accessible $2$-categories $\Lpm$ that capture the properties that examples like $\TOb$, when considered as a $2$-category, and $\MonCat_p$ satisfy.  The key result of this section concerns the stability of this class of $2$-categories under limits.  In Section~\ref{sect:core} we use these stability properties to prove our central result, Theorem~\ref{thm:cellular}, which asserts that if $\A$ is a small cellular $2$-category and $\C \in \Lpm$ then the 2-category $Ps(\A,\C)$ of 2-functors and \emph{pseudonatural transformations} belongs to $\Lpm$.

Some readers may prefer to begin reading at Section~\ref{sect:examples} and work backwards.  The section is devoted to applying Theorem~\ref{thm:cellular} to show that lots of $2$-categories of weak categorical structures and their pseudomaps belong to $\Lpm$.  This includes the cases $\MonCat_p$ and $\TOb$ mentioned above as well as a $2$-category of bicategories and doctrines such as $\Lex$ and $\Reg$.  We also show that the $2$-category of strict monoidal categories and strong monoidal functors is not accessible.

In Section~\ref{sect:2monad} we again use Theorem~\ref{thm:cellular}, but now with the intention of establishing general results about the algebras for a $2$-monad.  Section~\ref{sect:isofibrations} examines the accessibility of isofibrations and fibrations in a $2$-category whilst Section~\ref{sect:limittheories} takes a first glance at two-dimensional limit theories, which will be the topic of a followup paper.

\subsection*{Acknowledgements}

The present paper is greatly influenced by the work of Michael Makkai.  I am very grateful to Michael for sharing his ideas on this topic and for the enjoyable discussions we had about it when I visited him in Budapest for a week in 2015.  It is also heavily influenced by Australian 2-category theory, and in particular the connections between model categories and $2$-categories pioneered by Steve Lack, and I am very grateful to Steve for sharing his many ideas about this with me.

\section{Makkai's generalised sketches and structures defined by universal properties}\label{sect:sketches}

In \cite{Makkai1997Sketches1, Makkai1997Sketches2} Makkai showed how to realise categorical structures defined by universal properties as injectives in a presheaf category.  Moreover, he did this in such a way that the natural transformations between the underlying presheaves capture precisely the \emph{pseudomorphisms} of the appropriate kind.  It follows that the categories of such categorical structures and their pseudomaps are, in particular, accessible.  In this short section we review Makkai's illuminating approach, via his \emph{generalised sketches}, and give a detailed treatment of the simplest non-trivial example --- the case of categories admitting a terminal object.


Let $\C$ be a category and $C \in \C$.  Then $C//\C$ is the category whose objects are triples $$(A, A_C, a_c:A_C \to \C(C,A))$$ where $A \in \C$, $A_C \in \Set$ and $a_c:A_C \to \C(C,A)$ is a function.  A morphism $(f,f_C):(A, A_C, a_c) \to (B,B_C,b_c)$ consists of a morphism $f:A \to B \in \C$ and function $f_C:A_C \to B_C$ making the square
\begin{equation*}
\xymatrix{
A_C \ar[d]_{f_{c}} \ar[r]^-{a_{c}} & \C(C,A) \ar[d]^{\C(C,f)} \\
B_C  \ar[r]_-{b_{c}} & \C(C,B)}
\end{equation*}

A \emph{finite sketch category} is one obtained by applying the above construction finitely many times beginning with $\C = \Set$, and where at each stage the object $C \in \C$ is \emph{finitely presentable.}  

A pair $(\C,J)$ where $\C$ is a finite sketch category and $J$ a set of maps between finitely presentable objects in $\C$ is called a \emph{finite doctrine specification}.  The \emph{doctrine} is the full subcategory $Inj(J) \hookrightarrow \C$ of $J$-injectives.

The following proposition describes a few categorical properties shared by each finitary doctrine.

\begin{Prop}\label{prop:obvious}
Each finitary doctrine $Inj(J)$ is accessible with products and filtered colimits, and finite products commute with filtered colimits.
\end{Prop}

\begin{proof}
In fact, by \cite{Makkai1997Sketches1} each finite sketch category $\C$ is a presheaf category and so locally finitely presentable.  By Theorem 4.8 of \cite{Adamek1994Locally}, the full subcategory $Inj(J) \hookrightarrow \C$ is accessible and accessibly embedded.  It is also closed under products and, since the morphisms in $J$ have finitely presentable domains and codomains, filtered colimits.  Since finite products commute with filtered colimits in $\C$ it follows that they do so in $Inj(J)$ too.
\end{proof}

\begin{Example}
Let $2 \in \Set$.  Then $2//\Set = \Gph$, the category of directed graphs.
\end{Example}

If $\X$ is a finite set of finitely presentable objects in $\C$ it is convenient to also write $\X//\C$ for the category whose objects $$(A, A_X, a_X:A_X \to \C(C,X))$$ have $A \in \C$ as before,  plus a set $A_X$ and function $a_X:A_X \to \C(C,X)$ for each $X \in \X$ and with morphisms defined in the obvious way.  This is notationally convenient, but offers no extra generality since $\X//\C$ can be obtained by applying the previous construction once for each element of $\X$.

\begin{Example}
We begin by describing the doctrine for categories.  
To this end, consider the directed graphs $T$ and $I$ below.
\begin{equation*}
\begin{xy}
(-13,0)*+{T=}="as";(0,-5)*+{0}="00";(10,5)*+{1}="11";(20,-5)*+{2}="20";
{\ar^{} "00"; "11"};{\ar^{} "11"; "20"};{\ar^{} "00"; "20"};
\ar@{.}(-5,-10);(-5,10);
\ar@{.}(25,-10);(25,10);
\ar@{.}(-5,-10);(25,-10);
\ar@{.}(-5,10);(25,10);
\end{xy}
\hspace{1cm}
\begin{xy}
(-8,0)*+{I=}="as";(0,0)*+{0}="00";(10,0)*+{0}="11";
{\ar^{} "00"; "11"};
\ar@{.}(-3,-5);(13,-5);
\ar@{.}(-3,5);(13,5);
\ar@{.}(-3,5);(-3,-5);
\ar@{.}(13,5);(13,-5);
\end{xy}
\end{equation*}
The objects of $\cSk = \{T,I\} // \Gph$ are what Makkai calls \emph{category sketches} --- such consists of a graph $A$ together with functions $a_T:A(T) \to \Gph(T,A)$ and $a_I:A(I) \to \Gph(I,A)$, which we may think of as indexed sets of triangles and endomorphisms in $A$.  

There is an injective on objects fully faithful inclusion $\Cat \hookrightarrow \cSk$ which sends a category $A$ to its underlying graph, equipped with the subsets $A(T)$ and $I(T)$ of commutative triangles and identity morphisms in $A$ --- it is full since a functor is a graph map taking commutative triangles to commutative triangles and identities to identities.  We would like to identify its essential image as $Inj(J)$ for a set of maps so that $\Cat \simeq Inj(J)$ is the associated doctrine.  This is straightforward.  Firstly, we specify that $a_T:A(T) \to \Gph(T,A)$ and $a_I:A(I) \to \Gph(I,A)$ are injective.  This is done by considering the two morphisms in $\cSk$ depicted below, in which bold faced letters denote the elements of $A(T)$ and $A(I)$.

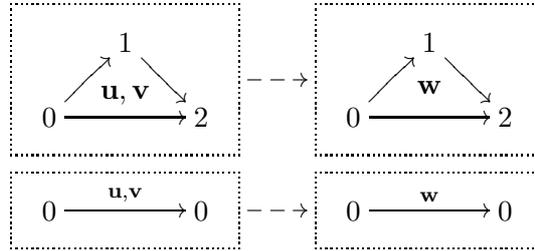
\begin{equation*}
\begin{xy}
(0,-5)*+{0}="00";(10,5)*+{1}="11";(20,-5)*+{2}="20";
{\ar^{} "00"; "11"};{\ar^{} "11"; "20"};{\ar^{} "00"; "20"};
(10,-2)*+{\mathbf{u,v}};
(40,-5)*+{0}="00";(50,5)*+{1}="11";(60,-5)*+{2}="20";
{\ar^{} "00"; "11"};{\ar^{} "11"; "20"};{\ar^{} "00"; "20"};
(50,-1)*+{\mathbf{w}};
(25,0)*+{}="h";(35,0)*+{}="t";{\ar@{-->}^{} "h"; "t"};
\ar@{.}(-5,-10);(-5,10);
\ar@{.}(25,-10);(25,10);
\ar@{.}(-5,-10);(25,-10);
\ar@{.}(-5,10);(25,10);
\ar@{.}(35,-10);(35,10);
\ar@{.}(65,-10);(65,10);
\ar@{.}(35,-10);(65,-10);
\ar@{.}(35,10);(65,10);
\end{xy}
\end{equation*}

\begin{equation*}
\begin{xy}
(0,0)*+{0}="00";(20,0)*+{0}="11";
{\ar^{\mathbf{u,v}} "00"; "11"};
(40,0)*+{0}="00";(60,0)*+{0}="11";
{\ar^{\mathbf{w}} "00"; "11"};
(25,0)*+{}="h";(35,0)*+{}="t";{\ar@{-->}^{} "h"; "t"};
\ar@{.}(-5,-5);(-5,5);
\ar@{.}(25,-5);(25,5);
\ar@{.}(-5,-5);(25,-5);
\ar@{.}(-5,5);(25,5);
\ar@{.}(35,-5);(35,5);
\ar@{.}(65,-5);(65,5);
\ar@{.}(35,-5);(65,-5);
\ar@{.}(35,5);(65,5);
\end{xy}
\end{equation*}

Henceforth we abuse notation and use $=$ to denote an element of $A(T)$ and $\mathbf{i}$ to denote an element of $A(I)$.  For existence and uniqueness of composites we add

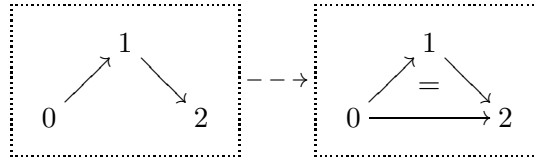
\begin{equation*}
\begin{xy}
(0,-5)*+{0}="00";(10,5)*+{1}="11";(20,-5)*+{2}="20";
{\ar^{} "00"; "11"};{\ar^{} "11"; "20"};
(10,-2)*+{};
(40,-5)*+{0}="00";(50,5)*+{1}="11";(60,-5)*+{2}="20";
{\ar^{} "00"; "11"};{\ar^{} "11"; "20"};{\ar^{} "00"; "20"};
(50,-1)*+{=};
(25,0)*+{}="h";(35,0)*+{}="t";{\ar@{-->}^{} "h"; "t"};
\ar@{.}(-5,-10);(-5,10);
\ar@{.}(25,-10);(25,10);
\ar@{.}(-5,-10);(25,-10);
\ar@{.}(-5,10);(25,10);
\ar@{.}(35,-10);(35,10);
\ar@{.}(65,-10);(65,10);
\ar@{.}(35,-10);(65,-10);
\ar@{.}(35,10);(65,10);
\end{xy}
\end{equation*}
and
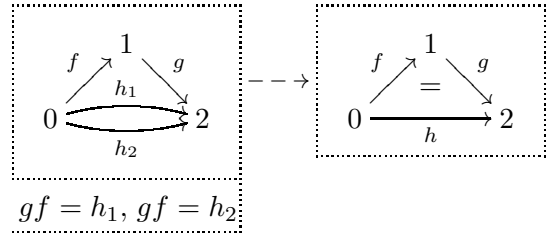
\begin{equation*}
\begin{xy}
(0,-5)*+{0}="00";(10,5)*+{1}="11";(20,-5)*+{2}="20";
{\ar^{f} "00"; "11"};{\ar^{g} "11"; "20"};{\ar@/^1ex/^{h_{1}} "00"; "20"};{\ar@/_1ex/_{h_{2}} "00"; "20"};
(10,-2)*+{};
(40,-5)*+{0}="00";(50,5)*+{1}="11";(60,-5)*+{2}="20";
{\ar^{f} "00"; "11"};{\ar^{g} "11"; "20"};{\ar_{h} "00"; "20"};
(50,-1)*+{=};
(25,0)*+{}="h";(35,0)*+{}="t";{\ar@{-->}^{} "h"; "t"};
\ar@{.}(-5,-13);(-5,10);
\ar@{.}(25,-13);(25,10);
\ar@{.}(-5,-13);(25,-13);
\ar@{.}(-5,10);(25,10);
\ar@{.}(-5,-20);(25,-20);
\ar@{.}(-5,-20);(-5,-25);
(3,-17)*+{gf=h_1,};(18,-17)*+{gf=h_2};
\ar@{.}(25,-13);(25,-20);
\ar@{.}(35,-10);(35,10);
\ar@{.}(65,-10);(65,10);
\ar@{.}(35,-10);(65,-10);
\ar@{.}(35,10);(65,10);
\end{xy}
\end{equation*}
Existence and uniqueness of identities is similar.  The three category axioms are each encoded by a single morphism in $\cSk$.  For instance, the left unit law is encoded by

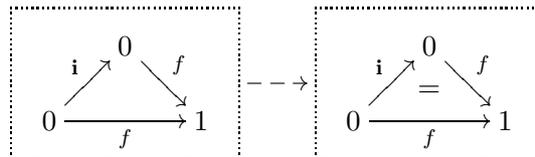
\begin{equation*}
\begin{xy}
(0,-5)*+{0}="00";(10,5)*+{0}="11";(20,-5)*+{1}="20";
{\ar^{\mathbf{i}} "00"; "11"};{\ar^{f} "11"; "20"};{\ar_{f} "00"; "20"};
(40,-5)*+{0}="00";(50,5)*+{0}="11";(60,-5)*+{1}="20";
{\ar^{\mathbf{i}} "00"; "11"};{\ar^{f} "11"; "20"};{\ar_{f} "00"; "20"};
(50,-1)*+{=};
(25,0)*+{}="h";(35,0)*+{}="t";{\ar@{-->}^{} "h"; "t"};
\ar@{.}(-5,-10);(-5,10);
\ar@{.}(25,-10);(25,10);
\ar@{.}(-5,-10);(25,-10);
\ar@{.}(-5,10);(25,10);
\ar@{.}(35,-10);(35,10);
\ar@{.}(65,-10);(65,10);
\ar@{.}(35,-10);(65,-10);
\ar@{.}(35,10);(65,10);
\end{xy}
\end{equation*}
and we leave the other cases to the reader.  In particular, for $J$ the collection of these 11 morphisms, we have $\Cat \simeq Inj(J)$.
\end{Example}

\begin{Example}
As shown in \cite{Makkai1997Sketches2}, the category of small categories with finite (co)limits of a given type and \emph{functors preserving them up to isomorphism} can be expressed as finitary doctrines.  Likewise the categories \textbf{Reg},~\textbf{Ex},~\textbf{Coh},~\textbf{Pretop} and \textbf{CCC} of small regular categories, Barr-exact categories, coherent categories, pretopoi, cartesian closed categories, and many more, can be captured as finitary doctrines.  In particular, all of these form accessible categories, with the properties described in Proposition~\ref{prop:obvious}.  The simplest non-trivial example is the category $\TOb$ of small categories admitting a terminal object and terminal object preserving functors.  This example contains the key idea required in all of the above cases --- for further details on the other examples, we refer the reader to \cite{Makkai1997Sketches2}.

Let $\tSk = T // \cSk$ where $T = \{\bullet\}$ is the graph with no morphisms, and with $T(C) = T(I) = \varnothing$; then an object of $\tSk$ is a category sketch equipped with an indexed set of objects. There is an injective on objects and fully faithful functor $K:\TOb \to \tSk$ sending $A$ to its underlying category sketch equipped with $A(T)$ the subset of \emph{all terminal objects} in $\A$.  Note that $K$ is full because a functor $F:A \to B$ lies in $\TOb$ precisely when it sends elements of $A(T)$ to elements of $B(T)$.

Now the forgetful functor $U:\tSk \to \cSk$ has a left adjoint $D$, which sends a category sketch $A$ to the terminal object sketch with $A(T) = \varnothing$ --- it follows that $X \in \tSk$ lies in $Inj(DJ)$ just when $UX$ is a category --- that is, an object of $Inj(J)$. 

Next we wish to force $a_T:A(T) \to \tSk(T,A)$ to be injective, so that $A(T)$ specifies a subset of objects of $A$.  This is much as in the preceding example: let $T_n \in \tSk$ denote the category sketch $T$ equipped with $T_n(T) = \{1,2,\ldots,n\}$.  Then $A$ is injective to the unique map $!_{2,1}:T_2 \to T_1$ just when $a_T$ is an injective function, as required.

It remains to axiomatise the following: (1) $A(T)$ is non-empty, (2) each element of $A(T)$ is terminal and (3) each terminal object belongs to $A(T)$.  For (1) we consider the inclusion of the empty sketch $\varnothing \to T_1$.  Capturing (2) requires a pair morphisms, depicted below

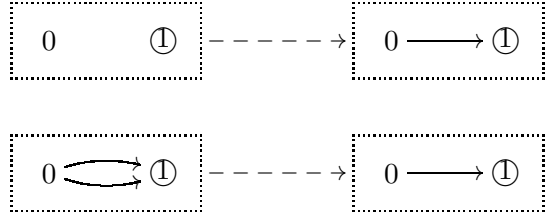
\begin{equation*}
\begin{xy}
(0,0)*+{0}="00";(15,0)*+{\textcircled{1}}="11";
(45,0)*+{0}="00";(60,0)*+{\textcircled{1}}="11";
{\ar^{} "00"; "11"};
(20,0)*+{}="h";(40,0)*+{}="t";{\ar@{-->}^{} "h"; "t"};
\ar@{.}(-5,-5);(20,-5)
\ar@{.}(-5,5);(20,5)
\ar@{.}(-5,5);(-5,-5)
\ar@{.}(20,5);(20,-5)
\ar@{.}(40,-5);(65,-5)
\ar@{.}(40,5);(65,5)
\ar@{.}(40,5);(40,-5)
\ar@{.}(65,5);(65,-5)
\end{xy}
\end{equation*}
\newline
\begin{equation*}
\begin{xy}
(0,0)*+{0}="00";(15,0)*+{\textcircled{1}}="11";
{\ar@/^1ex/^{} "00"; "11"};
{\ar@/_1ex/^{} "00"; "11"};
(45,0)*+{0}="00";(60,0)*+{\textcircled{1}}="11";
(20,0)*+{}="h";(40,0)*+{}="t";{\ar@{-->}^{} "h"; "t"};
{\ar^{} "00"; "11"};
\ar@{.}(-5,-5);(20,-5)
\ar@{.}(-5,5);(20,5)
\ar@{.}(-5,5);(-5,-5)
\ar@{.}(20,5);(20,-5)
\ar@{.}(40,-5);(65,-5)
\ar@{.}(40,5);(65,5)
\ar@{.}(40,5);(40,-5)
\ar@{.}(65,5);(65,-5)
\end{xy}
\end{equation*}
where we use a circle to indicate that the object belongs to the designated set of ``terminal objects".  Given (1) and (2) we know that a terminal object exists --- it follows that (3) amounts to the repleteness condition, that \emph{any object isomorphic to a ``terminal object" is again ``terminal".}

This is expressed using injectivity with respect to the following map

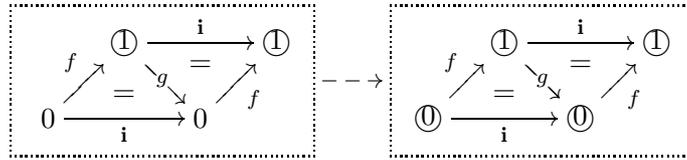
\begin{equation*}
\begin{xy}
(0,-5)*+{0}="00";(10,5)*+{\textcircled{1}}="11";(20,-5)*+{0}="20";(30,5)*+{\textcircled{1}}="t1";
{\ar^{f} "00"; "11"};{\ar|{g} "11"; "20"};{\ar_{\mathbf i} "00"; "20"};{\ar_{f} "20"; "t1"};{\ar^{\mathbf i} "11"; "t1"};
(50,-5)*+{\textcircled{0}}="00";(60,5)*+{\textcircled{1}}="11";(70,-5)*+{\textcircled{0}}="20";(80,5)*+{\textcircled{1}}="t1";
{\ar^{f} "00"; "11"};{\ar|{g} "11"; "20"};{\ar_{\mathbf i} "00"; "20"};{\ar_{f} "20"; "t1"};{\ar^{\mathbf i} "11"; "t1"};
(10,-2)*+{=};(20,2)*+{=};
(35,0)*+{}="h";(45,0)*+{}="t";{\ar@{-->}^{} "h"; "t"};
(60,-2)*+{=};(70,2)*+{=};
\ar@{.}(-5,-10);(-5,10)
\ar@{.}(35,-10);(35,10)
\ar@{.}(-5,-10);(35,-10)
\ar@{.}(-5,10);(35,10)
\ar@{.}(45,-10);(45,10)
\ar@{.}(85,-10);(85,10)
\ar@{.}(45,10);(85,10)
\ar@{.}(45,-10);(85,-10)
\end{xy}
\end{equation*}
where on the left we have the underlying category sketch of the free isomorphism, with the object $1$ marked as terminal, and on the right the same category sketch but with both objects marked as terminal.  All told, taking the union of $DJ$ and the preceding five morphisms, we obtain a set $J_{T}$ of maps in $\tSk$ such that $\TOb \simeq Inj(J_{T})$, as claimed.\end{Example}

\begin{Remark}
In the above, a doctrine is characterised as 
the full subcategory $Inj(J)$ of $J$-injectives in a presheaf category.  
As well as injectives, one can consider \emph{algebraic injectives} --- for instance, see \cite{Bourke2017Equipping, Nikolaus2011Algebraic} and \cite{Garner2011Understanding}.  These are injective objects \emph{equipped} with a solution to each extension problem, and form the objects of a category $AInj(J)$, whose morphisms preserve the chosen liftings.  

In the above example a $J_{T}$-algebraic injective is a category equipped with a choice of terminal object, whilst a morphism of algebraic injectives preserves the chosen terminal object strictly.   More generally, when categories with a class of limits and colimits satisfying some exactness properties are axiomatised using finitary doctrines \cite{Makkai1997Sketches2}, the algebraic injectives capture the same structures but equipped with a choice of the given limits and colimits, and the morphisms thereof preserve these limits and colimits strictly.  It follows that free categories \emph{admitting} such structure can be constructed using \emph{Quillen's small object argument} -- for instance, see \cite{Hovey1999Model} -- whilst free categories \emph{equipped} with such structure can be constructed using the \emph{algebraic small object argument} of \cite{Garner2011Understanding}.
\end{Remark}





\section{Background and terminology}\label{sect:background}
In the present section we fix terminology and draw together the background material, primarily on $2$-category theory, required in the remainder of the paper.
\subsection{Basic terminology}
\begin{itemize}
%
\item Let $\C$ be a $2$-category.  We write $\C_0$ for is its underlying category.
\item Let $\A$ be a small $2$-category and $\C$ a locally small $2$-category.  Then $[\A,\C]$ is the $2$-category of $2$-functors, $2$-natural transformations and modifications whilst $Ps(\A,\C)$ is the $2$-category of $2$-functors, pseudonatural transformations and modifications.
\item $\Cat$ is the $2$-category of small categories.  $\CAT$ is the large $2$-category of locally small categories.
\item $2$-categories can themselves be viewed as the objects of various kinds of higher category --- $2$-categories, $3$-categories and Gray-categories.  In the present paper, we will largely use just $2$-categorical techniques to study them.  
We write $\twocat$ for the $2$-category of small $2$-categories, $2$-functors and $2$-natural transformations, and $\Twocat$ for the $2$-category of locally small $2$-categories, again with $2$-functors and $2$-natural transformations.\end{itemize}
\subsection{Classes of morphism in a $2$-category}
A morphism $f:A \to B$ in a $2$-category $\C$ is:
\begin{itemize}
\item an \emph{equivalence} if there exists a morphism $g:B \to A \in \C$ and invertible $2$-cells $gf \cong 1_A$ and $fg \cong 1_B$; it is furthermore a \emph{retract (or surjective) equivalence} if the equivalence inverse $g$ can be chosen so that $fg = 1_B$;  
\item \emph{faithful/fully faithful/conservative} if the functor $\C(X,f):\C(X,A) \to \C(X,B)$ has the corresponding property for all $X \in \C$;
\item  an \emph{isofibration} if given $g:X \to A$, $h:X \to B$ and $\alpha:f \circ g \cong h$ there exists $\alpha^{\prime}:g \cong h^{\prime}$ such that $f \circ \alpha^{\prime} = \alpha$.  A functor $F:\A \to \B$ in $\Cat$ is then an isofibration if given $A \in \A$ and an isomorphism $\alpha:B \cong FA$ there exists a isomorphism $\alpha^{\prime}:B^{\prime} \cong A$ such that $F\alpha^{\prime} = \alpha$.  
It is a standard and routine fact that a morphism in a $2$-category is a surjective equivalence if and only if it is both an equivalence and isofibration.  
\item an \emph{isocofibration} if it is an isofibration in $\C^{op}$;
\item a \emph{discrete isofibration} if it is an isofibration and the lifted isomorphisms $\alpha^{\prime}:g \cong h^{\prime}$ are unique.  It is a standard fact that a morphism in a $2$-category is an isomorphism if and only if it is both an equivalence and discrete isofibration.  
\end{itemize}
Let us make a couple of further remarks on classes of morphisms.
\begin{itemize}
\item An \emph{isofibration of $2$-categories} -- that is, an isofibration in $\Twocat$ -- just amounts to a $2$-functor whose underlying functor is an isofibration in $\CAT$.
 \item A \emph{discrete isofibration} is an isofibration in which the lifted isomorphisms are required to be unique.  It is a routine fact that a morphism in a $2$-category is invertible if and only if it is both an equivalence and a discrete isofibration.
\item We note that there are various notions of equivalence of $2$-categories --- namely, equivalence in the $2$-category $\Twocat$ and the more general notion of \emph{biequivalence.}  In the present paper, we will only be interested in equivalences in $\Twocat$, which we will refer to as \emph{$2$-equivalences}.  The reason is that we will be dealing with genuine limits, filtered colimits and accessibility and these are concepts that are invariant under $2$-equivalence but not biequivalence.
\end{itemize}

The following is a standard $2$-categorical exercise but since we have been unable to find a reference for it in full generality, we give the construction.

\begin{Lemma}\label{lem:conservative}
For $\A$ a small 2-category the equivalences, surjective equivalences and isomorphisms in $Ps(\A,\C)$ are exactly the pointwise equivalences, pointwise surjective equivalences and pointwise isomorphisms in $\C$.
\end{Lemma}
\begin{proof}
Like any $2$-functor the evaluation $2$-functors $ev_{X}:Ps(\A,\C) \to C$ preserve equivalences, surjective equivalences and isomorphisms --- thus one direction is clear.  For the other, suppose that $f:A \to B \in Ps(\A,\C)$ is a pointwise equivalence.  Then each component $f_X:AX \to BX$ can be made into an adjoint equivalence $(\epsilon_X,f_X \dashv u_X,\eta_X)$.  We extend the components $u_X:BX \to AX$ to a pseudonatural transformation as follows.  At $r:X \to Y$ we have the pseudonaturality component of $f$ as below left.  Taking its \emph{mate} \cite{Kelly1972Review} through the adjoint equivalences $f_X \dashv u_X$ and $f_Y \dashv u_Y$ yields the composite invertible $2$-cell below centre.
\begin{equation*}
\xy
(00,-5)*+{AX}="00";(20,-5)*+{BX}="10";
(00,-25)*+{AY}="01";(20,-25)*+{BY}="11";
{\ar^{f_X} "00"; "10"}; 
{\ar_{Ar} "00"; "01"}; 
{\ar_{Br} "10"; "11"}; 
{\ar^{f_Y} "01"; "11"}; 
{\ar@{=>}^{f_r}(6,-15)*+{};(14,-15)*+{}};
\endxy
\hspace{1cm}
\xy
(00,0)*+{BX}="00";(30,0)*+{AX}="10"; (00,-30)*+{BY}="02";(30,-30)*+{AY}="22";
(00,-15)*+{BX}="01";(30,-15)*+{AX}="11";
{\ar^{Ar} "10"; "11"}; 
{\ar_{u_Y} "02"; "22"}; 
{\ar_{Br} "01"; "02"}; 
{\ar^{u_X} "00"; "10"}; 
{\ar_{1} "00"; "01"}; 
{\ar^{f_X} "10"; "01"}; 
{\ar_{f_Y} "11"; "02"}; 
{\ar^{1} "11"; "22"}; 
{\ar@{=>}_{f_r}(19,-16)*+{};(12,-16)*+{}};
{\ar@{=>}_{\epsilon_X}(14,-6)*+{};(7,-6)*+{}};
{\ar@{=>}_{\eta_Y}(24,-25)*+{};(17,-25)*+{}};
\endxy
\hspace{1cm}
\xy
(00,0)*+{BX}="00";(30,0)*+{AX}="10"; (00,-30)*+{BY}="02";(30,-30)*+{AY}="22";
(00,-15)*+{BX}="01";(30,-15)*+{AX}="11";
{\ar^{Ar} "10"; "11"}; 
{\ar_{u_Y} "02"; "22"}; 
{\ar_{Br} "01"; "02"}; 
{\ar^{u_X} "00"; "10"}; 
{\ar_{1} "00"; "01"}; 
{\ar^{f_X} "10"; "01"}; 
{\ar_{f_Y} "11"; "02"}; 
{\ar^{1} "11"; "22"}; 
{\ar@{=>}^{(f_r)^{-1}}(12,-17)*+{};(19,-17)*+{}};
{\ar@{=>}^{(\epsilon_X)^{-1}}(7,-6)*+{};(14,-6)*+{}};
{\ar@{=>}_{(\eta_Y)^{-1}}(17,-24)*+{};(24,-24)*+{}};
\endxy
\end{equation*}
It is the inverse of this composite $2$-cell, depicted above right, that we actually take as the value of $u_r$.  At $s:Y \to Z$ the compatibility $(As \circ u_r).(u_s \circ Br) = u_{s \circ r}$ for a pseudonatural transformation follows from one of the triangle equations for the adjunction $f_Y \dashv u_Y$ together with the corresponding pseudonaturality condition for $f$.  The remaining conditions for a pseudonatural transformation are straightforward to verify.  The families of invertible $2$-cells $(\epsilon_X)_{X \in A}$ and $(\eta_X)_{X \in A}$ then lift to invertible modifications $\epsilon$ and $\eta$.  Since the evaluation $2$-functors $ev_X:Ps(\A,\C) \to \C$ for $X \in \C$ are jointly faithful on $2$-cells the triangle equations in $Ps(\A,\C)$ follow directly from those in $\C$.  Thus each pointwise equivalence is an equivalence in $Ps(\A,\C)$.  

In the case that $f$ is a pointwise surjective equivalence or isomorphism then, since the $ev_X$ jointly reflect identity $2$-cells, the same construction shows it to be a surjective equivalence or isomorphism in $Ps(\A,\C)$.
\end{proof}

\subsection{Cellularity, cell inclusions and cellular $2$-categories}
Let $J$ be a set of morphisms in a cocomplete category $\C$.  A morphism $f:X \to Y \in \C$ is said to be $J$-cellular if 
\begin{itemize}
\item it is a transfinite composite of pushouts of coproducts of morphisms in $J$.
\end{itemize}
If the domains of the morphisms in $J$ are finitely presentable, then by Proposition A.6 of \cite{Malts2010} it suffices to consider \emph{countable} composition.  This is the approach we will take in the present paper.  

Let us write $Cell(J) \hookrightarrow \C^{\atwo}$ for the subcollection of $J$-cellular morphisms.  This is closed under coproducts, pushout and transfinite cocomposition.  An object $X \in \C$ is said to be \emph{$J$-cellular} if the unique map $\varnothing \to X$ is a $J$-cellular morphism.  A $J$-cofibration is a morphism belonging to the retract closure of the $J$-cellular morphisms in $\C^{\atwo}$, whilst an object $X \in \C$ is $J$-cofibrant if $\varnothing \to X$ is a $J$-cofibration.  Under mild assumptions on $\C$ --- for instance if $\C$ is locally presentable --- the set $J$ generates a weak factorisation system $(Cof(J),RLP(J))$ on $\C$ whose left class consists of the $J$-cofibrations and whose right class consists of the morphisms with the right lifting property with respect to $J$.  See Proposition 2.1.14 of \cite{Hovey1999Model} for a proof.  The right class $RLP(J)$ is closed under products, pullbacks and transfinite cocomposition.

\subsubsection{Cell inclusions for categories}\label{sect:catCof}
There is a natural model structure on $\Cat_0$ (or indeed any sufficiently bicomplete 2-category \cite{Lack2007Homotopy-theoretic}) which captures the equivalences, isofibrations and surjective equivalences described above.  In $\Cat_0$, the set $I$ of generating cofibrations is given by the following maps: firstly, we have the inclusion
\begin{equation}
J_0:P_{-1} \to D_0:
\begin{xy}
(25,0)*+{}="30";(35,0)*+{}="40";(50,0)*+{\bullet}="50";
{\ar@{|->}^{} "30"; "40"};
\ar@{.}(-3,-5);(23,-5)
\ar@{.}(-3,5);(23,5)
\ar@{.}(-3,5);(-3,-5)
\ar@{.}(23,5);(23,-5)
\ar@{.}(37,-5);(63,-5)
\ar@{.}(37,5);(63,5)
\ar@{.}(37,5);(37,-5)
\ar@{.}(63,5);(63,-5)
\end{xy}
\end{equation}
 of the empty category to the terminal such; secondly the inclusion
\begin{equation}
J_1:P_0 \to D_1:
\begin{xy}
(0,0)*+{\bullet}="00";(20,0)*+{\bullet}="10";
(25,0)*+{}="30";(35,0)*+{}="40";(40,0)*+{\bullet}="50";(60,0)*+{\bullet}="60";
{\ar@{|->}^{} "30"; "40"};
{\ar^{} "50"; "60"};
\ar@{.}(-3,-5);(23,-5)
\ar@{.}(-3,5);(23,5)
\ar@{.}(-3,5);(-3,-5)
\ar@{.}(23,5);(23,-5)
\ar@{.}(37,-5);(63,-5)
\ar@{.}(37,5);(63,5)
\ar@{.}(37,5);(37,-5)
\ar@{.}(63,5);(63,-5)
\end{xy}
\end{equation}
of the boundary of the generic arrow $D_1$.  (We often denote $D_1$, which will appear regularly, by $\atwo$.)

Thirdly, there is the functor
\begin{equation}
\nabla:P_1 \to D_1:
\begin{xy}
(0,0)*+{\bullet}="00";(20,0)*+{\bullet}="10";
{\ar@/^1pc/^{} "00"; "10"};{\ar@/_1pc/^{} "00"; "10"};
(25,0)*+{}="30";(35,0)*+{}="40";(40,0)*+{\bullet}="50";(60,0)*+{\bullet}="60";
{\ar@{|->}^{} "30"; "40"};
{\ar^{} "50"; "60"};
\ar@{.}(-3,-5);(23,-5)
\ar@{.}(-3,5);(23,5)
\ar@{.}(-3,5);(-3,-5)
\ar@{.}(23,5);(23,-5)
\ar@{.}(37,-5);(63,-5)
\ar@{.}(37,5);(63,5)
\ar@{.}(37,5);(37,-5)
\ar@{.}(63,5);(63,-5)
\end{xy}
\end{equation}
identifying a parallel pair of morphisms.  

We note that each object is cofibrant whilst the morphisms with the right lifting property are the surjective equivalences.

Let us also mention the single generating trivial cofibration for the model structure on $\Cat_0$.  This is the inclusion
\begin{equation}
I_1:D_0 \to I_1:
\begin{xy}
(10,0)*+{\bullet}="00";
(25,0)*+{}="30";(35,0)*+{}="40";(40,0)*+{\bullet}="50";(60,0)*+{\bullet}="60";
{\ar@{|->}^{} "30"; "40"};
{\ar^{\cong} "50"; "60"};
\ar@{.}(-3,-5);(23,-5)
\ar@{.}(-3,5);(23,5)
\ar@{.}(-3,5);(-3,-5)
\ar@{.}(23,5);(23,-5)
\ar@{.}(37,-5);(63,-5)
\ar@{.}(37,5);(63,5)
\ar@{.}(37,5);(37,-5)
\ar@{.}(63,5);(63,-5)
\end{xy}
\end{equation}
selecting (either) endpoint of the free isomorphism; the functors with the right lifting property to $I_1$ are the isofibrations. 

\subsubsection{Cell inclusions for $2$-categories and cellular $2$-categories}\label{sect:2catcof}
The generating cofibrations $J$ for the natural model structure on $\twocat_{0}$ \cite{Lack2002A-quillen} will play a central role in what follows.  To begin with, we have the inclusions $J_0$ and $J_1$ as above where source and target are viewed as $2$-categories in which all $2$-cells are identities.  Thirdly, we have the inclusion 
\begin{equation}
J_2:P_1 \to D_2:
\begin{xy}
(0,0)*+{\bullet}="00";(20,0)*+{\bullet}="10";
{\ar@/^1pc/^{} "00"; "10"};{\ar@/_1pc/^{} "00"; "10"};
(25,0)*+{}="30";(35,0)*+{}="40";(40,0)*+{\bullet}="50";(60,0)*+{\bullet}="60";
{\ar@{|->}^{} "30"; "40"};
{\ar@/^1pc/^{} "50"; "60"};{\ar@/_1pc/^{} "50"; "60"};
{\ar@{=>}^{}(50,3)*+{};(50,-3)*+{}};
\ar@{.}(-3,-5);(23,-5)
\ar@{.}(-3,5);(23,5)
\ar@{.}(-3,5);(-3,-5)
\ar@{.}(23,5);(23,-5)
\ar@{.}(37,-5);(63,-5)
\ar@{.}(37,5);(63,5)
\ar@{.}(37,5);(37,-5)
\ar@{.}(63,5);(63,-5)
\end{xy}
\end{equation}
of the parallel pair of $1$-cells to the free $2$-cell $D_2$ and, finally, the $2$-functor 
\begin{equation}
J_3:P_2 \to D_2:
\begin{xy}
(0,0)*+{\bullet}="00";(20,0)*+{\bullet}="10";
{\ar@/^1pc/^{} "00"; "10"};{\ar@/_1pc/^{} "00"; "10"};
{\ar@{=>}^{}(8,3)*+{};(8,-3)*+{}};
{\ar@{=>}^{}(12,3)*+{};(12,-3)*+{}};
(25,0)*+{}="30";(35,0)*+{}="40";(40,0)*+{\bullet}="50";(60,0)*+{\bullet}="60";
{\ar@{|->}^{} "30"; "40"};
{\ar@/^1pc/^{} "50"; "60"};{\ar@/_1pc/^{} "50"; "60"};
{\ar@{=>}^{}(50,3)*+{};(50,-3)*+{}};
\ar@{.}(-3,-5);(23,-5)
\ar@{.}(-3,5);(23,5)
\ar@{.}(-3,5);(-3,-5)
\ar@{.}(23,5);(23,-5)
\ar@{.}(37,-5);(63,-5)
\ar@{.}(37,5);(63,5)
\ar@{.}(37,5);(37,-5)
\ar@{.}(63,5);(63,-5)
\end{xy}
\end{equation}
identifying a pair of parallel $2$-cells.  Making the identification $D_3=D_2$ then $J=\{J_i:P_{i-1} \to D_{i}:i = 0,1,2,3\}$ is the set of generating cofibrations.

\subsubsection{Cellular $2$-categories}\label{sect:cellular2}

The $J$-cellular and $J$-cofibrant $2$-categories coincide and we will refer to them as \emph{cellular 2-categories}.  By \cite{Lack2002A-quillen} they can be characterised as those $2$-categories whose underlying category is \emph{free on a graph}.  Using this result, or otherwise, it is not hard to see that each cellular $2$-category can be obtained as a composite
$$\varnothing = \A_0 \to \A_1 \to \A_2  \to \A_3  \to \A_{4}=\A$$
in which each $\A_{i} \to \A_{i+1}$ is a pushout of a copower of  $J_i$'s --- concretely, this says that a 2-category is cellular just when it can be obtained by starting with a set of objects, freely adding some 1-cells between these, then freely adding some 2-cells between the generated $1$-cells and finally adding equations between the generated $2$-cells.


\begin{Examples}
Below are a few key examples of cellular $2$-categories.
\begin{enumerate}
\item The maps $J_i:P_{i-1} \to D_{i}$ for $i \in \{0,1,2,3\}$ have cellular source and target.
\item The free adjunction $\Adj$ of \cite{Schanuel1986Thefree} is the $2$-category generated by objects $0$ and $1$, morphisms $u:0 \to 1$ and $f:1 \to 0$ and a pair of $2$-cells $\eta:id_{1} \Rightarrow uf$ and $\epsilon:fu \Rightarrow id_{0}$ subject to the triangle equations
$$\xy
(0,0)*+{0}="00"; (0,-15)*+{1}="01";
(20,0)*+{0}="10";(20,-15)*+{1}="11";
{\ar_{f} "00"; "01"}; 
{\ar^{1} "00"; "10"};
{\ar|{u} "01"; "10"}; 
{\ar_{1} "01"; "11"};
{\ar^{f} "10"; "11"};   
{\ar@{=>}^{\eta}(5,-2)*+{};(5,-8)*+{}};
{\ar@{=>}^{\epsilon}(15,-7)*+{};(15,-13)*+{}};
(30,-8)*+{= id_{f}}="11";
(45,-8)*+{\textnormal{and}}="11";
\endxy
\hspace{0.5cm}
\xy
(0,0)*+{1}="00"; (0,-15)*+{0}="01";
(20,0)*+{1}="10";(20,-15)*+{0}="11";
{\ar_{u} "00"; "01"}; 
{\ar^{1} "00"; "10"};
{\ar|{f} "01"; "10"}; 
{\ar_{1} "01"; "11"};
{\ar^{u} "10"; "11"};   
{\ar@{=>}^{\epsilon}(5,-8)*+{};(5,-2)*+{}};
{\ar@{=>}^{\eta}(15,-13)*+{};(15,-7)*+{}};
(30,-8)*+{= id_{u}.}="11";
\endxy$$
It is straightforward to translate the above presentation into a presentation in the colimit sense --- four stages are required.  We emphasise here that the maps $f,u:\atwo \to \Adj$ selecting the left and right adjoint in $\Adj$ respectively are cellular.
\item Let $I_2$ be the \emph{free invertible $2$-cell} --- then the induced map $D_2 \to I_2$ is cellular.  In particular the composite $K:P_1 \to D_2 \to I_2$ which adds an invertible $2$-cell between the parallel pair is a cellular $2$-functor.
\item  An adjunction $f \dashv u$ is said to be a reflection if the counit $\epsilon:fu \Rightarrow 1$ is invertible and a coreflection if the unit $\eta:1 \Rightarrow uf$ is invertible.  By forming the pushout of $\epsilon:D_2 \to \Adj$ along $D_2 \to I_2$ we obtain the free reflection $\Ref$, which accordingly is cellular --- in particular the maps $f,u:\atwo \to \Ref$ selecting its left and right adjoints are cellular.  Again, we have the cellular maps $f,u:\atwo \to \Coref$ to the free coreflection, and to the free adjoint equivalence $f,u:\atwo \to \AdjEq$.
\end{enumerate}

\end{Examples}



\subsection{A primer on $2$-categorical limits}

Whilst in category theory one often wishes to consider objects like the pullback of an opspan, in $2$-category theory one often has cause to consider the comma object $f/g$ of an opspan 
\begin{equation*}
\begin{xy}
(0,0)*+{f/g}="00";(15,0)*+{B}="10";(0,-15)*+{A}="01";(15,-15)*+{C}="11";
{\ar^{} "00"; "10"};{\ar_{} "00"; "01"};{\ar^{f} "10"; "11"};{\ar_{g} "01"; "11"};
{\ar@{=>}^{}(10,-4)*+{};(4,-10)*+{}};
\end{xy}
\end{equation*}
or the Eilenberg-Moore object of a monad.  These are not limits in the usual sense of category theory but \emph{weighted limits} in the sense of $\Cat$-enriched category theory \cite{Kelly1982Basic}.  

\subsubsection{Weighted limits}

Recall that for $\A$ a small $2$-category an ($\A$-indexed) weight is a $2$-functor $W:\A \to \Cat$; given a \emph{diagram} $D:\A \to \C$, its limit $L$ weighted by $W$ is defined by a $2$-natural isomorphism $$\C(X,L) \cong [\A,\Cat](W,\C(X,D-)).$$  
The $W$-weighted colimit of a diagram $D:\A^{op} \to \C$ can be defined as the weighted limit of $D^{op}:\A \to \C^{op}$.

\begin{Example}
For instance, the weight for comma objects is the diagram 
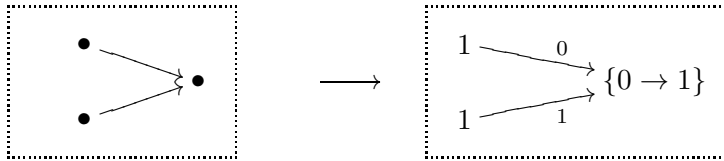
\begin{equation*}
\begin{xy}
(5,5)*+{\bullet}="01";(5,-5)*+{\bullet}="0-1";(20,0)*+{\bullet}="10";
(35,0)*+{}="h";(45,0)*+{}="t";{\ar^{} "h"; "t"};
{\ar^{} "01";"10"};{\ar^{} "0-1";"10"};
(55,5)*+{1}="21";(55,-5)*+{1}="2-1";(80,0)*+{\{0 \to 1\}}="20";
{\ar^{0} "21";"20"};{\ar_{1} "2-1";"20"};
\ar@{.}(-5,-10);(-5,10)
\ar@{.}(25,-10);(25,10)
\ar@{.}(-5,-10);(25,-10)
\ar@{.}(-5,10);(25,10)
\ar@{.}(50,-10);(50,10)
\ar@{.}(90,-10);(90,10)
\ar@{.}(50,-10);(90,-10)
\ar@{.}(50,10);(90,10)
\end{xy}
\end{equation*}
where the pair of functors in $\Cat$ are named by the objects they select.
\end{Example}
\begin{Example}
The conical limits of ordinary category theory are those weighted by the terminal weight $\Delta_1:\A \to \Cat$ where $\A$ is merely a category (viewed as a locally discrete $2$-category).  
\end{Example}

\begin{Remark}
The weighted limit $L$ comes equipped with a universal morphism $\eta:W \to \C(L,D-)$, and so for each $x \in Wj$ a morphism $$\eta_{j}(x):L \to Dj$$ in $\C$.  We will refer to such morphisms as \emph{limit projections}.  It follows from the universal property of $L$ that the limit projections have a number of properties: here, we make note only of the fact that they are \emph{jointly conservative.}
\end{Remark}

\subsubsection{Flexible limits}

Our primary focus will be on flexible limits \cite{Bird1989Flexible} --- by which we mean limits weighted by \emph{flexible weights}.  We write $Flex$ for the class of flexible weights, so that $Flex(\A) \subseteq [\A,\Cat]$ consists of the $\A$-indexed flexible weights.  These can be defined in several ways.  

\begin{enumerate}
\item The most elementary approach is to define flexible limits are those \emph{constructible} from products, inserters, equifiers and splittings of idempotents.  The reader will certainly be familiar with products and splittings of idempotents.  Given a parallel pair $f,g:A \rightrightarrows B$ the \emph{inserter} $I$ of $f$ and $g$ comes equipped with a morphism $i:I \to A$ and $2$-cell $\eta: f \circ i \Rightarrow g \circ i$ as depicted below left

\begin{equation*}
\begin{xy}
(0,0)*+{I}="00";(20,8)*+{A}="10";(20,-8)*+{A}="01";(40,0)*+{B}="11";
{\ar^{i} "00"; "10"};{\ar_{i} "00"; "01"};{\ar^{f} "10"; "11"};{\ar_{g} "01"; "11"};
{\ar@{=>}^{\eta}(20,4)*+{};(20,-4)*+{}};
\end{xy}
\hspace{1.5cm}
\begin{xy}
(-20,0)*+{E}="-10";
(0,0)*+{A}="00";(25,0)*+{B}="10";
{\ar^{e} "-10"; "00"};{\ar@/^1pc/^{f} "00"; "10"};{\ar@/_1pc/_{g} "00"; "10"};
{\ar@{=>}_{\alpha}(10,3)*+{};(10,-3)*+{}};
{\ar@{=>}^{\beta}(15,3)*+{};(15,-3)*+{}};
\end{xy}
\end{equation*}
and is universal amongst such triples $(I,i,\eta)$ of this shape.  

Again given $f,g:A \rightrightarrows B$, but now further equipped with a parallel pair of $2$-cells $\alpha$ and $\beta$ as above right, the \emph{equifier} of $\alpha$ and $\beta$ is given by a pair $(E,e)$, as depicted, satisfying $\alpha \circ e = \beta \circ e$.  It is universal amongst such pairs with this property.

The class of limits constructible from products, inserters and equifiers are called pie limits \cite{Blackwell1989Two-dimensional, Power1991A-characterization}.  This class includes comma objects, pseudo, lax and oplax limits, Eilenberg-Moore objects of monads and most other important $2$-categorical limits.  The correct intuition about this class, an intuition which can be made precise, is that pie limits are those whose ``limiting cone" does not force any new equations between morphisms.  In particular, pullbacks and equalisers are not pie limits.  

By ``constructible" in the above, we mean it in the sense the saturation of a class of weights \cite{Kelly2005Notes}.  That is, given a class $\Phi$ of weights, the \emph{saturation} $\Phi^{*}$ of $\Phi$ is the class of weights with $\Phi^{*}(\A) \subseteq [\A,\Cat]$ defined to be the closure of the representables under $\Phi$-colimits.  Then letting $PIE$ be the class of weights for products, inserters and equifiers, a weight $W:\A \to \Cat$ is, by definition, pie just when when it belongs to $PIE^{*}(\A) \subseteq [\A,\Cat]$ whilst $Flex(\A) = PIES^{*}(\A)$ where $S$ stands for the weight for split idempotents.  

In the above cases however, the generating weighted limits are by no means unique, and somewhat arbitrary --- for instance, flexible limits are also those constructible from weighted pseudolimits and splittings of idempotents.  Two more conceptually satisfying approaches are described below.

\item The second uses the natural model structure on $\Cat_0$.  This induces the \emph{projective} model structure on $[\A,\Cat]_0$ in which the weak equivalences, fibrations and trivial fibrations are defined pointwise as in $\Cat_0$.  This has generating cofibrations the set $I^{*}=\{i \times \A(a,-):i \in I\}$ where $I$ the set of generating cofibrations in $\Cat_0$ described in Section~\ref{sect:catCof}.  Now a weight is pie just when it is $I^{*}$-cellular and flexible just when it is an $I^{*}$-cofibrant object \cite{Lack2007Homotopy-theoretic}.

\item Finally we describe the first approach, which is the $2$-categorical one of \cite{Bird1989Flexible}.  To this end, recall that the identity on objects inclusion $[\A,\Cat] \hookrightarrow Ps(\A,\Cat)$ has a left adjoint $Q$ called the \emph{pseudomorphism classifier}; thus one has the counit, a $2$-natural transformation $p_W:QW \to W$, and the unit, a pseudonatural transformation $q_W:W \to QW$, satisfying the triangle equation $p_W \circ q_W = 1$.  In fact $p_W$ is a surjective equivalence in $Ps(\A,\Cat)$ with section $q_W$.  The weight $W$ is said to be flexible if $p_W$ admits a section in $[A,\Cat]$.  \begin{footnote}{This says that the flexible weights are those admitting coalgebra structure for the copointed endofunctor $(Q,p)$.  In fact the pie weights are those admitting coalgebra structure for $Q$ as a comonad \cite{Lack2012Enhanced, Bourke2013On}.}\end{footnote}

It follows that $p_W:QW \to W$ is a surjective equivalence in $[\A,\Cat]$.  Using that $[\A,\Cat](QW,V) \cong Ps(\A,\Cat)(W,V)$ it follows that for flexible $W$ the inclusion $[\A,\Cat](W,V) \hookrightarrow Ps(J,\Cat)(W,V)$ is an equivalence for all $V$;  thus each pseudonatural transformation with source $W$ is isomorphic to a $2$-natural transformation.

\subsubsection{Bilimits}
Given $W$ a weight and $D$ a diagram as before, the $W$-weighted bilimit $B$ is defined by a pseudonatural equivalence $$\C(X,B) \cong Ps(\A,\Cat)(W,\C(X,D-)).$$  Unlike weighted limits, weighted bilimits are of course only determined up to equivalence.  If the genuine weighted limit $L$ exists, to say that it is the bilimit is precisely to say that the composite
$$\C(X,L) \cong [\A,\Cat](W, \C(X,D-)) \hookrightarrow Ps(\A,\Cat)(W, \C(X,D-))$$
is an equivalence for each $A \in \C$; in other words that $$[\A,\Cat](W, C(X,D-)) \hookrightarrow Ps(\A,\Cat)(W, C(X,D-))$$ is an equivalence.  As remarked above, this is the case for flexible weights --- thus \emph{flexible limits are bilimits}.  

However, other kinds of limit, such as \emph{pullbacks of isofibrations} (see \cite{Joyal1993Pullbacks}) and \emph{transfinite cocomposites of isofibrations} (see Proposition A.3(1) of \cite{Bourke2018Iterated}) are also bilimits.  Pullbacks of isofibrations will play a central role in the present paper.

\end{enumerate}

\subsubsection{Finite pie and finite flexible limits}
We define the class of \emph{finite pie limits} $Pie_f$ and \emph{finite flexible limits} $Flex_f$ to be those generated by finite products, inserters, equifiers and, in the latter case, also splittings of idempotents.  Thus we have, by definition, $Pie_f = (P_{f}IE)^{*}$ and $Flex_f = (P_{f}IES)^{*}$ where $P_f$ indicates the collection of weights for \emph{finite products}.\begin{footnote}{The class of finite flexible weights defined above does not coincide with the class described in \cite{Bird1989Flexible} but rather, by Theorem 5.3 of \emph{loc.cit.}, its saturation.  We use the saturated form as our definition here since the properties of finite flexible limits that we are interested in --- namely, commutativity with a class of colimits --- are invariant under saturation.}\end{footnote}
  By Proposition 5.2 of \cite{Bird1989Flexible} powers by finitely presentable categories are also finite flexible.
\subsection{Accessible $2$-categories and their underlying categories}

Let $\lambda$ be a regular cardinal, $\A$ a small $\lambda$-filtered category and $\C$ a $2$-category.  Given a functor $D:\A \to \C$ (where $\A$ is viewed as a locally discrete 2-category) its $\lambda$-filtered colimit is an object $X$ equipped with an \emph{isomorphism of categories} $$\C(X,Y) \cong [\A,\C](D,\Delta(Y))$$ $2$-natural in $Y$.   Since we require an isomorphism of categories, rather than a mere bijection of underlying sets, this is a slightly stronger condition than being the colimit in the underlying category $\C_0$, but see Proposition~\ref{prop:equivalent}(1) below.  

An object $X$ of $\C$ is said to be $\lambda$-presentable if the $2$-functor $\C(X,-):\C \to \Cat$ preserves $\lambda$-filtered colimits.  Then $\C$ is said to be $\lambda$-accessible if it has $\lambda$-filtered colimits and a set $\G$ of $\lambda$-presentable objects such that each object of $\C$ is a $\lambda$-filtered colimit of objects in $\G$.  A $2$-category $\C$ is accessible if it is $\lambda$-accessible for some $\lambda$.  We note that another definition of enriched accessibility is given in \cite{BorceuxQuinteiro}.

The following proposition shows that if $\C$ has powers with the generic arrow $\atwo$ then accessibility of $\C$ as a $2$-category can be understood in terms of accessibility of its underlying category.  The following result can be easily extended to general $\cv$. A slight variant is given in Section 7.3 of \cite{Bourke2020Adjoint}.

\begin{Prop}\label{prop:equivalent}
Let $\C$ have powers with $\atwo$.  Then
\begin{enumerate}
\item $\C$ has $\lambda$-filtered colimits if and only if $\C_0$ does.
\item Moreover $\C$ is $\lambda$-accessible if and only if $\C_0$ is and $\lambda$-filtered colimits in $\cc$ commute with powers by $\atwo$.
\end{enumerate}
\end{Prop}
\begin{proof}
The first part is standard, but we give it for completeness.  Let $D:\A \to \C_0$ be a filtered diagram, with colimiting cocone $\eta:D \to \Delta(X)$ in $\C_0$.  We must show that the induced functor $\C(X,A) \to [\A,\C](D,\Delta(A))$ is invertible.  Since $\atwo$ is a strong generator in $\Cat_0$ this is equally to show that the top row below is a bijection for all $A$.
\begin{equation*}
\xymatrix{
\Cat_{0}(\atwo,\C(X,A)) \ar[d]^{\cong} \ar[rr]^{} && \Cat_{0}(\atwo,[\A,\C](D,\Delta(A))) \ar[d]^{\cong} \\
\C_0(X,A^{\atwo}) \ar[rr]_{} && [\A,\C_0](D,\Delta(A^{\atwo}))
}
\end{equation*}
Since $\C$ has powers by $\atwo$, we have an isomorphism of rows as above. Thus it suffices to show that the bottom row is a bijection and this is simply the universal property of the filtered colimit in $\C_0$.

For the second part suppose that $\C$ is $\lambda$-accessible.  We must show that each $\lambda$-presentable object $A$ in $\cc$ is $\lambda$-presentable in $\cc_0$.  We have a commutative diagram as below
\begin{equation*}
\xymatrix{
\cc_0 \ar@/{_1pc}/[rrrr]_{\cc_0(A,-)} \ar[rr]^{\cc(A,-)_0} && \Cat_0 \ar[rr]^{\Cat_0(1,-)} && \Set
}
\end{equation*}
Now $\cc(A,-)$, and hence $\cc(A,-)_0$, preserves $\lambda$-filtered colimits.  Since $1$ is finitely presentable in $\Cat$, so does the right leg above whence the composite does too.  Therefore $A$ is $\lambda$-presentable in $\cc_0$, so that by the first part $\cc_0$ is $\lambda$-accessible.  Now since $\cc$ is $\lambda$-accessible the inclusion $J:\cc_{\lambda} \to \cc$ induces a fully faithful $2$-functor $\cc \to [\cc_{\lambda}^{op},\Cat]$ preserving limits and $\lambda$-filtered colimits.  Thus to prove the stated exactness property holds in $\cc$ it suffices to establish it in $\Cat$.  Now the powering $2$-functor is the representable $\Cat(\atwo,-):\Cat \to \Cat$ whose underlying functor is $[\atwo,-]:\Cat_0 \to \Cat_0$.  To show that this preserves $\lambda$-filtered colimits, it is enough to show that $\Cat_0(\atwo,-) \circ [\atwo,-] \cong \Cat_0(\atwo \times \atwo,-)$ does so, and this is the case since $\atwo \times \atwo$ is finitely presentable in $\Cat_0$.  

Now suppose $\C_0$ is $\lambda$-accessible and that $\lambda$-filtered colimits commute with powers by $\atwo$.  Let $A$ be a $\lambda$-presentable object of $\cc_0$ --- we must prove that $\cc(A,-):\cc \to \Cat$ preserves $\lambda$-filtered colimits or, equivalently, that its underlying functor $\cc(A,-)_0:\cc_0 \to \Cat_0$ does so.  Since $\Cat_0$ is lfp, much as before, it suffices to show that $\Cat_0(\atwo,\cc(A,-)_0):\cc_0 \to \Set$ preserves $\lambda$-filtered colimits.  Now $[\atwo,\cc(A,-)] \cong \cc(A,(-)^{\atwo})$ so taking underlying functors we have $\Cat_0(\atwo,\cc(A,-)_0) \cong \cc_0(A,(-)^{\atwo})$.  Now the right hand side preserves $\lambda$-filtered colimits since $\cc_0(A,-)$ does and $(-)^{\atwo}$ does by assumption.
\end{proof}

\section{The accessible $2$-categories of interest and their stability under limits}\label{sect:lp}

\begin{Def}\label{defn:lp}
A locally small $2$-category $\C$ is said to belong to $\Lp$ if
\begin{itemize}
\item $\C$ is accessible with filtered colimits;
\item $\C$ has flexible limits;
\item finite flexible limits commute with filtered colimits in $\C$.
\end{itemize}
\end{Def}

A morphism of the large $2$-category $\Lp$ is a $2$-functor preserves the limits and colimits in question, whilst a $2$-cell of $\Lp$ is a 2-natural transformation.  The name $\Lp$ reflects the fact that its objects are certain \emph{weakly locally presentable} $2$-categories --- in \cite{Lack2012Enriched} a $2$-category is said to be weakly locally presentable if it is accessible with flexible limits.

\begin{Remark}
There is some freedom in which colimits we take in the above ---we could, for instance, replace filtered colimits by those colimits which commute with finite flexible limits in $\Cat$.  Or more generally, we could use any subclass of these which contain $\lambda$-filtered colimits for some fixed regular cardinal $\lambda$.
\end{Remark}


\begin{Remark}
Let us mention that it follows immediately from Section 8.2 of ~\cite{Bourke2020Adjoint} that each $2$-category in $\Lp$ admits a strong form of bicolimits and that each morphism of $\Lp$ admits a strong form of left biadjoint.
\end{Remark}

 We will occasionally make of the larger $2$-category $\Acc_{\omega}$, an object of which is a $2$-category whose underlying category is accessible with filtered colimits.  Morphisms of $\Acc_{\omega}$ are $2$-functors whose underlying functor preserves filtered colimits.  Again morphisms are $2$-natural transformations.

\begin{Prop}\label{prop:limits}
Both $\Acc_{\omega}$ and $\Lp$ are closed in $\Twocat$ under 
bilimits --- in particular, under products, powers by small categories, pullbacks of isofibrations, and transfinite cocomposites of isofibrations.
\end{Prop}
\begin{proof}
To begin with, let $\V$ be a complete and cocomplete symmetric monoidal category and $\VCAT$ the large $2$-category of locally small $\V$-categories.  Let $\Phi$ be a class of weights, and let $\mathbf{CTS}_{\Phi} \hookrightarrow \VCAT$ the sub-$2$-category of $\Phi$-cocomplete categories and $\Phi$-continuous functors; similarly let $\mathbf{COCTS}_{\Phi} \hookrightarrow \VCAT$ consist of the $\Phi$-cocomplete categories and $\Phi$-cocontinuous functors.  By Proposition 6.2 of \cite{Bird1984Limits} $\mathbf{CTS}_{\Phi}$ is closed in $\VCAT$ under flexible limits (there called limits of retract type).  Using that a $\V$-category $\C$ has $\Phi$-limits if and only if $\C^{op}$ has $\Phi$-colimits, it follows by duality that $\mathbf{COCTS}_{\Phi}$ is also closed in $\VCAT$ under flexible limits.


Combining the above cases, we see that the $2$-category $\A$ of locally small $2$-categories with flexible limits and filtered colimits, and $2$-functors preserving such, is closed under $\Twocat$ under flexible limits.  Let $\B \hookrightarrow \A$ be the full sub $2$-category of $\A$ whose objects satisfy the exactness property of Definition~\ref{defn:lp}.  Then the flexible limit $L=\{W,D\}$ of a diagram $D:\J \to \B$, formed in $\A$, comes equipped with a jointly conservative set of limit projections $\{\eta_{i}(x):L \to Di:i \in \J, x \in Wi\}$.  Since exactness properties are phrased in terms of certain morphisms being invertible, they are reflected by jointly conservative morphisms that preserve the limits and colimits in question; it follows therefore that the limit $L$ also lies in $\B$.  Thus $\B \hookrightarrow \Twocat$ is closed under flexible limits.  Now if $\{W,D\}_{bi}$ is a bilimit and the flexible limit $\{QW,D\}$ exists, then $\{W,D\}_{bi} \simeq \{QW,D\}$. Since $\B \hookrightarrow \Twocat$ is closed under flexible limits and equivalence-replete, it follows that $\B$ is closed in $\Twocat$ under bilimits.

By Theorem 5.1.6 of \cite{Makkai1989Accessible} the $2$-category $\Acc$ of accessible categories and accessible functors is closed in $\CAT$ under bilimits.  Taking $\cv = \Set$ and $\Phi = filt$ the class of weights for filtered colimits, it follows as before that $\mathbf{COCTS}_{filt} \hookrightarrow \CAT$ is closed under bilimits.  Since the forgetful $2$-functor $U:\Twocat \to \CAT$ preserves bilimits, therefore $\Acc_{\omega} = U^{-1}(\mathbf{COCTS}_{filt} \cap \Acc) \hookrightarrow \Twocat$ is closed under bilimits, whereby the intersection $\Lp = \B \cap \Acc_{\omega}$ is too.  

%
%
%
\end{proof}

The $2$-categories of interest to us will be shown in Section~\ref{sect:examples} to belong to $\Lp$.  However, they each share a \emph{crucial further property} which we turn to now.  Let $\C^{\atwo}=[\atwo,\C]$ denote the $2$-category whose objects are arrows in $\C$ and whose morphisms are strictly commuting squares, and let $$RE(\C) \hookrightarrow \C^{\atwo}$$ denote the full sub 2-category consisting of the retract equivalences in $\C$.  By the preceding proposition, if $\C \in \Acc_{\omega}/\Lp$ then $\C^{\atwo} \in \Acc_{\omega}/\Lp$ too.
%

\begin{Def}
Let $\C \in \Acc_{\omega}$.  We say that $\C$ satisfies Property $\mathcal M$ if $RE(\C) \hookrightarrow \C^{\atwo} \in \Acc_\omega$ --- that is, the full subcategory $RE(\C)_0 \hookrightarrow \C^{\atwo}_0$ of retract equivalences is accessible and closed under filtered colimits.
\end{Def}

Let $\Acc_{\omega,\mathcal{M}}$ and $\Lpm$ denote the full sub-2-categories of $\Acc_{\omega}$ and $\Lp$ whose objects satisfy Property $\mathcal M$.

\begin{Remark}
In the above the symbol $\mathcal M$ stands for \emph{map}.  From the model categorical perspective, it is natural also to consider the analogue of this property for the other key classes of maps ---namely, the isofibrations and equivalences.  For $\C \in \Lp$ we prove in Proposition~\ref{prop:isofibrations} that these follow from Property $\mathcal M$ itself.
\end{Remark}

We have an alternative description of Property~$\mathcal{M}$ for those $\C$ belonging to $\Lp$.

\begin{Lemma}
$\C \in \Lp$ satisfies Property $\mathcal M$ if and only if the inclusion $RE(\C) \hookrightarrow \C^{\atwo}$ belongs to $\Lp$.
\end{Lemma}
\begin{proof}
It suffices to prove if $\C$ has flexible limits then $RE(\C) \hookrightarrow \C^{\atwo}$ is closed under flexible limits.  Indeed, since the representables $\C(C,-):\C \to \Cat$ preserve flexible limits and both preserve and \emph{reflect} surjective equivalences, it suffices to prove that $RE(\Cat) \hookrightarrow \Cat^{\atwo}$ is closed under flexible limits.  For this case, see Section 9 of \cite{Lack2012Enriched}.
\end{proof}

\begin{Example}\label{ex:1}
If $\C$ is a locally finitely presentable $2$-category and the surjective equivalences in $\C$ are of the form $RLP(I)$ for a set $I$ of morphisms between finitely presentable objects in $\C$ then $\C \in \Lpm$.  Indeed, $RLP(I) \to (\C_0)^{\atwo}$ is accessible and accessibly embedded by Proposition 3.3 of \cite{Rosicky2007On} and closed under filtered colimits since the objects involved in $I$ are finitely presentable.  In particular, this is the case when $\C =\Cat$ by Example~\ref{sect:catCof}.  Therefore $\Cat$ belongs to $\Lpm$.
\end{Example}
\begin{Example}\label{ex:2}
Let us show that $\Cat^{\atwo}$ does \emph{not} belong to $\Lpm$.  By Remark 3.12(3) of \cite{Borceux2007Purity} the split epis in $\Set^{\atwo}$ are not closed under filtered colimits.  Let $I:\Set \to \Cat_0$ sending a set to the indiscrete category with the same set of objects.  Then $[\atwo,I]:[\atwo,\Set] \to [\atwo,\Cat]_0$ preserves filtered colimits and the split epis in $[\atwo,\Set]$ are the preimage of the retract equivalences in $[\atwo,\Cat]$.  Applying $[\atwo,I]$ to the example from \cite{Borceux2007Purity} of a filtered colimit of split epis in $\Set^{\atwo}$ which is not a split epi hence produces a filtered colimit of surjective equivalences in $\Cat^{\atwo}$ which is not a surjective equivalence, as required.

The morphisms of $\Cat^{\atwo}$ are strict maps, and therein lies the problem.  On the other hand, we will see in Proposition~\ref{prop:acc2} that $Ps(\atwo,\Cat)$ does belong to $\Lpm$.

\end{Example}

Note that $\Lpm$, unlike $\Lp$, is \emph{not} closed in $\Twocat$ under powers by small categories --- this is the content of Example~\ref{ex:2}.  However, we will show shortly that it does have most of the limit closure properties of Proposition~\ref{prop:limits}.  

To prove this, we will find it convenient to work with an equivalent form of Property~$\mathcal M$.  To this end, recall that a \emph{cloven} surjective equivalence $f:A \to B \in \C$ is one equipped with a choice of section $r:B \to A$.  Let $RE_{c}(\C)$ denote the $2$-category of cloven retract equivalences in $\C$ --- these are retract equivalences equipped with a choice of section, but with $1$-cells and $2$-cells as in $RE(\C)$ (that is, not required to commute with the sections.)  Evidently, the forgetful 2-functor $RE_{c}(\C) \to RE(\C)$ is a $2$-equivalence and furthermore we have a commutative diagram of forgetful $2$-functors as below.
\begin{equation*}
\xymatrix{
RE_{c}(\C) \ar[dr] \ar[r]^{\simeq} & RE(\C) \ar[d] \\
& \C^{\atwo}
}
\end{equation*}
Thus $\C \in \Acc_{\omega}$ satisfies Property $\mathcal M$ just when the forgetful $2$-functor $RE_{c}(\C) \to \C^{\atwo}$ belongs to $\Acc_{\omega}$.

Now consider the free surjective equivalence $RE$ depicted in its entirety below
$$\xy
(0,0)*+{0}="00"; (0,-14)*+{0}="01";
(20,0)*+{1}="10";(20,-14)*+{1}="11";
{\ar_{1} "00"; "01"}; 
{\ar^{p} "00"; "10"};
{\ar^{q} "10"; "01"}; 
{\ar_{p} "01"; "11"};
{\ar^{1} "10"; "11"};   
(6,-3)*{{\eta}};
(6,-6)*{{\cong}};
\ar@{.}(-5,-20);(-5,5);
\ar@{.}(25,-20);(25,5);
\ar@{.}(-5,-20);(25,-20);
\ar@{.}(-5,5);(25,5);
\endxy$$
 wherein $p \eta = id$, and let $p:\atwo \to RE$ denote the map selecting the surjective equivalence $p$.  The following lemma is straightforward, but we will give the details since we will use the construction twice.

\begin{Lemma}\label{lem:pullback}
We have a pullback
$$\xymatrix{
RE_{c}(\C) \ar[d] \ar[r] & Ps(RE,\C) \ar[d]^{p^{*}} \\
\C^{\atwo} \ar@{{(}->}[r] & Ps(\atwo,\C)
}$$
of $2$-categories, in which the bottom leg is the inclusion.
\end{Lemma}
\begin{proof}

Let $f:A \to B$ be a cloven surjective equivalence, with cleavage $s_f$.  Since the surjective equivalence $\C(A,f):\C(A,A) \to \C(A,B)$ is fully faithful there exists a unique invertible $2$-cell $\rho_f:s_f \circ f \cong 1_A$ which postcomposed by $f$ yields the identity $f \circ s_f = 1_B$.  This specifies a unique $2$-functor $K(f,s_f):RE \to \C$ with $K(p)=f$ and $K(q)=s_f$.

Now consider a map $(u,v):(f,s_f) \to (g,s_g) \in RE_{c}(\C)$ as in the right hand square below.

\begin{equation}\label{eq:se}
\xy
(0,0)*+{B}="A";(20,0)*+{A}="R";(40,0)*+{B}="B";
(0,-15)*+{D}="C";(20,-15)*+{C}="S";(40,-15)*+{D}="D";
{\ar^{s_{f}} "A"; "R"}; 
{\ar^{f} "R"; "B"}; 
{\ar_{s_{g}} "C"; "S"}; 
{\ar_{g} "S"; "D"}; 
{\ar_{v} "A"; "C"}; 
{\ar^{u} "R"; "S"}; 
{\ar^{v} "B"; "D"}; 
(10,-9)*{{\cong}};
(10,-5)*{s^{u}_{v}};
\endxy
\end{equation}
Then since $g$ is fully faithful, there exists a unique invertible $2$-cell $s^{u}_{v}$ with the displayed source and target such that $g \circ s^{u}_{v} = id_{u}$.  This gives the component of a pseudonatural transformation $K(u,v):K(f,s_f) \to K(g,s_g)$ at $q$, the unique such with identity component at $p$.  The remaining pseudonaturality condition concerns naturality at $\eta$, and this is again straightforward using that $g$ is faithful.  Given a second map $(u',v'):(g,s_g) \to (h,s_h) \in RE_{c}(\C)$ we must show that $K(u',v') \circ K(u,v) = K(u'u,v'v)$, and this amounts to checking that we have the following equality of $2$-cells.

\begin{equation*}
\xy
(0,0)*+{B}="A";(20,0)*+{A}="R";
(0,-15)*+{D}="C";(20,-15)*+{D}="S";
(0,-30)*+{F}="E";(20,-30)*+{E}="T";
{\ar^{s_{f}} "A"; "R"}; 
{\ar_{s_{g}} "C"; "S"}; 
{\ar_{s_{h}} "E"; "T"}; 
{\ar_{v} "A"; "C"}; 
{\ar^{u} "R"; "S"}; 
{\ar_{v'} "C"; "E"}; 
{\ar^{u'} "S"; "T"}; 
(10,-9)*{{\cong}};
(10,-5)*{s^{u}_{v}};
(10,-26)*{{\cong}};
(10,-22)*{s^{u'}_{v'}};
(30,-15)*{=};
\endxy
\hspace{0.5cm}
\xy
(0,-8)*+{B}="A";(20,-8)*+{A}="R";
(0,-23)*+{F}="C";(20,-23)*+{E}="S";
{\ar^{s_{f}} "A"; "R"}; 
{\ar_{s_{h}} "C"; "S"}; 
{\ar_{v' \circ v} "A"; "C"}; 
{\ar^{u' \circ u} "R"; "S"}; 
(10,-18)*{{\cong}};
(10,-13)*{s^{u'u}_{v'v}};
\endxy
\end{equation*}
To see this, observe that since the surjective equivalence $\C(B,h):\C(B,E) \to \C(B,F)$ is faithful, it suffices to show that both are sent to the identity $2$-cell on $u'u$, and this follows easily from the definitions of $s^{u}_{v}$ and $s^{u'}_{v'}$.  That $K$ preserves identity $1$-cells is straightforward.

Given a $2$-cell in $RE_{c}(\C)$ as depicted below

\begin{equation*}
\xy
(0,0)*+{A}="A";(25,0)*+{B}="R";
(0,-15)*+{C}="C";(25,-15)*+{D}="S";
{\ar^{f} "A"; "R"}; 
{\ar_{g} "C"; "S"}; 
{\ar@/_1.2pc/_{u} "A"; "C"}; 
{\ar@/^1.2pc/^{u'} "A"; "C"}; 
{\ar@{=>}^{\alpha}(-3,-8)*+{};(3,-8)*+{}};
{\ar@/_1.2pc/_{v} "R"; "S"}; 
{\ar@/^1.2pc/^{v'} "R"; "S"}; 
{\ar@{=>}^{\beta}(22,-8)*+{};(28,-8)*+{}};
\endxy
\end{equation*}
the additional condition for a modification $K(\alpha,\beta)$ between the associated pseudonatural transformations is the equation below.
\begin{equation*}
\xy
(0,0)*+{B}="A";(25,0)*+{A}="R";
(0,-15)*+{D}="C";(25,-15)*+{C}="S";
{\ar^{s_f} "A"; "R"}; 
{\ar_{s_g} "C"; "S"}; 
{\ar@/_1.2pc/_{v} "A"; "C"}; 
{\ar@/^1.2pc/^{v'} "A"; "C"}; 
{\ar@{=>}^{\beta}(-3,-8)*+{};(3,-8)*+{}};
{\ar@/^1.2pc/^{u'} "R"; "S"}; 
(15,-5)*{s^{u'}_{v'}};
(15,-9)*{{\cong}};
(40,-7)*{=};
\endxy
\hspace{0.5cm}
\xy
(0,0)*+{B}="A";(25,0)*+{A}="R";
(0,-15)*+{D}="C";(25,-15)*+{C}="S";
{\ar^{s_f} "A"; "R"}; 
{\ar_{s_g} "C"; "S"}; 
{\ar@/_1.2pc/_{v} "A"; "C"}; 
{\ar@/_1.2pc/_{u} "R"; "S"}; 
{\ar@{=>}^{\alpha}(22,-8)*+{};(28,-8)*+{}};
{\ar@/^1.2pc/^{u'} "R"; "S"}; 
(10,-5)*{s^{u}_{v}};
(10,-9)*{{\cong}};
\endxy
\end{equation*}
As before, it suffices to show that both composite $2$-cells coincide on postcomposition with $g$ --- under which, they each yield the $2$-cell $\alpha$.  Functoriality of $K$ on $2$-cells is immediate.  By uniqueness of the constructions subject to the given constraints, we have a pullback as desired.
\end{proof}

\begin{Theorem}\label{thm:limits2}
Both $\Acc_{\omega,\mathcal{M}}$ and $\Lpm$ are closed in $\Twocat$ under products, pullbacks of isofibrations, and transfinite cocomposites of isofibrations.
\end{Theorem}
\begin{proof}
Since $\Lpm = \Acc_{\omega,\mathcal{M}} \cap \Lp$ and $\Lp$ satisfies, by Proposition~\ref{prop:limits}, the above stability properties, it suffices to prove the claim for $\Acc_{\omega,\mathcal{M}}$.

Let $D:J \to  \Acc_{\omega,\mathcal{M}}$ and $p_i:L \to Di \in \Twocat$ be the limiting cone of a diagram $D$ of one of the three given types.  Now consider the commutative diagram below in which the vertical arrows are the forgetful $2$-functors.
\begin{equation*}
\xymatrix{
RE_c(L)\ar[d]_{U_{L}} \ar[rr]^{RE_{c}(p_i)} && RE_{c}(Di) \ar[d]_{U_{Di}} \ar[rr]^{RE_{c}(D\alpha)} && RE_{c}(Dj) \ar[d]_{U_{Dj}} \\
[\atwo,L] \ar[rr]^{[\atwo,p_i]} && [\atwo,Di] \ar[rr]^{[\atwo,D\alpha]} && [\atwo,Dj]
}
\end{equation*}
We must show that the whole diagram belongs to $ \Acc_{\omega}$.  Certainly the right adjoint $[\atwo,-]:\Twocat_0 \to \Twocat_0$ preserves limits, and it easy to see that it also preserves isofibrations.  It follows that the bottom row is a limit of a diagram of the same type in $ \Acc_{\omega}$ and so belongs to $ \Acc_{\omega}$ by Proposition~\ref{prop:limits}.  By assumption each vertical comparison $U_{D_i}:RE_c(D_i) \to [\atwo,D_i]$ belongs to $ \Acc_{\omega}$.  Since these maps are fully faithful they also reflect filtered colimits.  It follows that each connecting map $RE_{c}(D\alpha):RE_{c}(Di) \to RE_{c}(Dj)$ belongs to $ \Acc_{\omega}$.  Thus the diagram $RE_c(D-):J \to \Twocat$ takes values in $ \Acc_{\omega}$.  Therefore, if we can show that $RE_c(-):\Twocat_0 \to \Twocat_0$ preserves limits and isofibrations then, as before, the top row, as well as the canonical comparison $U_L$ between limits, will belong to $ \Acc_{\omega}$.  

Again isofibrations are clearly preserved.  As for limit preservation, consider again the pullback square of Lemma~\ref{lem:pullback}.  Since the pullback square is natural in $\C$ and since pullbacks commute with all limits we see that $RE_{c}(-)$ will preserve limits if the components $(-)^{\atwo}, Ps(RE,-),Ps(\atwo,-):\Twocat_0 \to \Twocat_0$ do so.  The first of these we have already dealt with, whilst the latter two preserve limits by Lemma~\ref{lemma:size} below.

%
\end{proof}

\begin{Lemma}\label{lemma:size}
The functor $Ps(-,-):\twocat_{0}^{op} \times \Twocat_{0} \to \Twocat_{0}$ preserves limits in each variable.
\end{Lemma}
\begin{proof}
The Gray tensor product $\A \otimes \B$ of small $2$-categories $\A$ and $\B$ \cite{Gray1974Formal} has the universal property that there is a natural bijection $\twocat_{0}(\A \otimes \B,\C) \cong \twocat_{0}(\A,Ps(\B,\C))$.  For $\C$ merely locally small, one sees by factoring through small full sub $2$-categories of $\C$, that this extends to a natural bijection $\Twocat_{0}(\A \otimes \B,\C) \cong \Twocat_{0}(\A,Ps(\B,\C))$.   Since the objects $\{D_i:i=0,1,2\}$ of $\Twocat_{0}$ form a strong generator, in order to prove that $Ps(-,-):\Twocat_{0}^{op} \times \Twocat_{0} \to \Twocat_{0}$ preserves limits in each variable, it suffices to prove that each composite $\Twocat_{0}(D_i,Ps(-,-))$ does so.  By the preceding, we have a natural bijection $\Twocat_{0}(D_i,Ps(\B,\C)) \cong \Twocat_{0}(D_i \otimes \B,\C)$, so that for fixed $B$ the functor is representable (and so preserves limits), whilst for fixed $\C$ it is the composite $\Twocat_{0}(-,\C) \circ J^{op} \circ (D_i \otimes -)^{op}$.  This preserves limits as the composite of a representable, the inclusion $J^{op}:\twocat_{0}^{op} \to \Twocat_{0}^{op}$ and the right adjoint $(D_i \otimes -)^{op}$.  Note that $J:\twocat_0 \to \Twocat_0$, is easily seen to preserve small colimits by factoring through small full sub-$2$-categories of large ones.
\end{proof}

\section{Cellularity and the accessibility of pseudonatural transformations}\label{sect:core}

In the present section, which is the technical heart of the paper, we prove in Theorem~\ref{thm:cellular} that:
\begin{itemize}
\item if $\A$ is a cellular $2$-category and $\C \in \Lpm$ then $Ps(\A,\C) \in \Lpm$ with flexible limits and filtered colimits pointwise as in $\C$.
\item furthermore, if $F:\A \to \B$ is a cellular 2-functor between cellular $2$-categories then $Ps(F,\C):Ps(\B,\C) \to Ps(\A,\C)$ is an isofibration in $\Lpm$.
\end{itemize} 
Note that by \emph{pointwise} in the above, we mean that for each $x \in \A$ the $2$-functor $ev_x:Ps(\A,\C) \to \C$ preserves flexible limits and filtered colimits.  Most of the work lies in the following special case. 
\subsection{The case $\A =\atwo$}\label{sect:squares}
In this case the $2$-category of interest is $Ps(\atwo,\C)$, whose objects are morphisms and whose morphisms are pseudo-commutative squares.  The key idea is to show that $Ps(\atwo,\C)$ can be embedded into 
a functor $2$-category $[S,\C]$ of spans, from which it inherits many good properties.  The embedding is defined using \emph{pseudolimits of arrows}, which are finite pie limits.  

Given a morphism $f:A \to B$ in $\C$ its pseudolimit $Pf$ comes equipped with a ``cone" as on the left below.

$$\xy
(0,0)*+{Pf}="a0"; (-15,-15)*+{A}="b0";(15,-15)*+{B}="c0";
{\ar_{p_{f}} "a0"; "b0"}; 
{\ar^{q_{f}} "a0"; "c0"}; 
{\ar_{f} "b0"; "c0"}; 
(1,-9)*{{\cong}_{\lambda_{f}}};
\endxy
\hspace{2cm}
\xy
(0,0)*+{X}="a0"; (-15,-15)*+{A}="b0";(15,-15)*+{B}="c0";
{\ar_{r} "a0"; "b0"}; 
{\ar^{s} "a0"; "c0"}; 
{\ar_{f} "b0"; "c0"}; 
(1,-9)*{{\cong}_{\alpha}};
\endxy$$
This has the $1$-dimensional universal property that given any other cone, as on the right above, there exists a unique $t:X \to Pf$ satisfying 
\begin{equation*}
p_{f}t=r\textnormal{, }q_{f}t=s\textnormal{ and }\lambda_{f}t=\alpha \hspace{0.1cm} .
\end{equation*}

The $2$-dimensional universal property is that given a pair of cones $(r,\alpha,s)$ and $(r^{\prime},\alpha^{\prime},s^{\prime})$ with common base $X$ together with 2-cells $\theta_{r}:r \Rightarrow r^{\prime} \in \C(X,A)$ and $\theta_{s}:s \Rightarrow s^{\prime} \in \C(X,B)$ satisfying
$$\xy
(0,0)*+{fr}="00"; 
(15,0)*+{s}="10"; (0,-10)*+{fr^\prime}="01";
(15,-10)*+{s^{\prime}}="11";
{\ar@{=>}^{\alpha} "00"; "10"}; 
{\ar@{=>}^{\theta_s} "10"; "11"}; 
{\ar@{=>}_{f\theta_{r}} "00"; "01"}; 
{\ar@{=>}_{\alpha^{\prime}} "01"; "11"}; 
\endxy$$
there exists a unique 2-cell $\phi:t \Rightarrow t^{\prime} \in \C(X,Pf)$ between the induced factorisations such that 
\begin{equation*}
p_{f}\phi=\theta_{r}\textnormal{ and }q_{f}\phi=\theta_{s}\hspace{0.1cm} .
\end{equation*}

\begin{Example}\label{ex:pseudolimit}
In $\Cat$, the objects of the pseudolimit $Pf$ are triples $(a \in A,\alpha:fa \cong b,b\in B)$, whilst a morphism $(u,v):(a,\alpha,b) \to (c,\beta,d)$ consists of arrows $u:a \to c \in A$ and $v:b \to d \in B$ making the square
\begin{equation*}
\xymatrix{
fa \ar[d]_{fu} \ar[r]^{\alpha} & b \ar[d]^{v} \\
fc \ar[r]^{\beta} & d}
\end{equation*}
commute.  The projections $p_f:P_f \to A$ and $q_f:P_f \to B$ send $(a,\alpha,b)$ to $a$ and $b$ respectively, whilst the component $\lambda_f(a,\alpha,b)$ is $\alpha:fa \cong b$ itself.
\end{Example}

Observe that the commuting triangle

$$\xy
(0,0)*+{A}="a0"; (-15,-15)*+{A}="b0";(15,-15)*+{B}="c0";
{\ar_{1} "a0"; "b0"}; 
{\ar^{f} "a0"; "c0"}; 
{\ar_{f} "b0"; "c0"};
\endxy
$$
induces a unique morphism $s_f:A \to P_f$ such that 
\begin{equation*}
p_f \circ s_f = 1_A, q_f \circ s_f = f \textnormal{ and } \lambda_f \circ s_f = id_f \hspace{0.1cm}.
\end{equation*}
In fact, as is well known --- see, for instance, \cite{Blackwell1989Two-dimensional} --- $p_f$ is a surjective equivalence.

\begin{Prop}\label{prop:properties}
Consider $f:A \to B \in \C$ where $\C$ is a $2$-category with pseudolimits of arrows and products.  Then 
\begin{enumerate}
\item $p_f$ is a surjective equivalence with section $s_f$;
\item $(p_f,q_f):P_f \to A \times B$ is a discrete isofibration.
\end{enumerate}
\end{Prop}
\begin{proof}
Since the representables $\C(X,-):\C \to \Cat$ preserve pseudolimits of arrows and products, and both preserve and reflect the property of being a surjective equivalence or discrete isofibration, it suffices to verify that Properties (1) and (2) above hold when $\C = \Cat$, in which case they are evident from the description in Example~\ref{ex:pseudolimit}.
\end{proof}    



Consider the 2-category $Ps(\atwo,\C)$ of arrows and pseudocommutative squares in $\C$.  We now describe a 2-functor $P:Ps(\atwo,\C) \to [S,\C]$ where $S=\{0 \leftarrow 1 \rightarrow 2\}$ is the free span.  On objects it sends
\begin{equation*}
\xymatrix{
 & A \ar[r]^{f} & B &  \textnormal{to the span} & A & Pf\ar[l]_{p_{f}} \ar[r]^{q_{f}} & B.}
\end{equation*}  
At a morphism $(r,s,\alpha):f \to g \in Ps(\atwo,\C)$ the span map $(r,P_\alpha,s):(p_f,q_f) \to (p_g,q_g)$ is specified by the commutativity depicted below.

$$\xy
(15,15)*+{P_{f}}="P";
(0,0)*+{A}="A";(30,0)*+{B}="B";(0,-15)*+{C}="C";(30,-15)*+{D}="D";
{\ar_{f} "A"; "B"}; 
{\ar_{g} "C"; "D"}; 
{\ar_{r} "A"; "C"}; 
{\ar^{s} "B"; "D"}; 
(15,-8)*{{\cong}_{\alpha}};
{\ar_{p_f} "P"; "A"}; 
{\ar^{q_f} "P"; "B"}; 
(15,6)*{{\cong}_{\lambda_f}};
\endxy
\hspace{1cm}
=
\hspace{1cm}
\xy
(15,15)*+{P_{f}}="P";
(15,0)*+{P_{g}}="Pg";
(0,0)*+{A}="A";(30,0)*+{B}="B";(0,-15)*+{C}="C";(30,-15)*+{D}="D";
{\ar_{g} "C"; "D"}; 
{\ar_{r} "A"; "C"}; 
{\ar^{s} "B"; "D"}; 
{\ar_{p_f} "P"; "A"}; 
{\ar^{q_f} "P"; "B"}; 
{\ar|{P\alpha} "P"; "Pg"}; 
{\ar_{p_{g}} "Pg";"C"};
{\ar^{q_{g}} "Pg";"D"};
(15,-9)*{{\cong}_{\lambda_g}};
\endxy
$$
Let $(\theta_r,\theta_s):(r,s,\alpha) \Rightarrow (r^\prime,s^\prime,\alpha^\prime) \in Ps(\atwo,\C)(f,g)$ be a 2-cell.  Using the $2$-dimensional universal property of $Pg$ there exists a unique $2$-cell $P\theta:P\alpha \Rightarrow P\alpha^{\prime}$ such that $(\theta_r,P\theta,\theta_s):(r,P\alpha,s) \Rightarrow (r^{\prime},P\alpha^{\prime},s^{\prime})$ is a 2-cell in $[S,\C]$.

\begin{Theorem}\label{thm:span}
The above assignments determines a $2$-functor $P:Ps(\atwo,\C) \to [S,\C]$ which is fully faithful on $1$-cells and on $2$-cells.  Furthermore:
\begin{itemize}
\item a span $(a,b):A  \leftarrow R  \to  B$ lies in its essential image if and only if $a:R \to A$ is a surjective equivalence and $(a,b):R \to A \times B$ is a discrete isofibration.  
\end{itemize}
\end{Theorem}
\begin{proof}
$2$-functoriality of $P$ follows in a straightforward manner from the universal properties.  To prove the theorem, let us define an interposing $2$-category $[S,\C]_{f}$ as follows.  Objects of $[S,\C]_{f}$ are triples $(\mathbf{a},R,b)$ consisting of a span $(a,R,b)$ satisfying the above two properties together with a section $s_a:A \to R$ of $a$.  The morphisms and $2$-cells of $[S,\C]_{f}$ are just those of the underlying spans in $[S,\C]$, so that there is a forgetful $2$-functor $[S,\C]_{f} \to [S,\C]$ which is fully faithful on $1$-cells and $2$-cells and has in its image exactly those span satisfying Property $\mathcal M$.  By Lemma~\ref{prop:properties} $P$ factors through $[S,\C]_{f}$, as $P:Ps(\atwo,\C) \to [S,\C]_{f}$, and it remains to prove that this factored $2$-functor is a $2$-equivalence.  We will now describe its equivalence inverse $I$, which augments the construction of Lemma~\ref{lem:pullback}.

To this end, consider a span $(\mathbf a,b):R \to A \times B \in [S,\C]_{f}$.  We define $I(\mathbf a,b) = b \circ s_a:A \to B$.  At a span morphism $(u,v,w):(\mathbf a,b) \to (\mathbf c,d)$ there exists, as in \eqref{eq:se}, a unique invertible $2$-cell $s^{u}_{v}$ as in the left square below
\begin{equation*}
\xy
(0,0)*+{A}="A";(20,0)*+{R}="R";(40,0)*+{B}="B";
(0,-15)*+{D}="C";(20,-15)*+{S}="S";(40,-15)*+{D}="D";
{\ar^{s_{a}} "A"; "R"}; 
{\ar^{b} "R"; "B"}; 
{\ar_{s_{c}} "C"; "S"}; 
{\ar_{d} "S"; "D"}; 
{\ar_{u} "A"; "C"}; 
{\ar^{v} "R"; "S"}; 
{\ar^{w} "B"; "D"}; 
(10,-9)*{{\cong}};
(10,-5)*{s^{u}_{v}};
\endxy
\end{equation*}
which yields an identity under postcomposition by $c$.  We define $I(u,v,w) = (u,w,d \circ s^u_v)$ as depicted above.  At a $2$-cell $(\alpha, \beta,\gamma):(u,v,w) \to (u',v',w')$ we have $I(\alpha, \beta,\gamma)=(\alpha,\gamma)$.  
Now the assignment $$\mathbf a \mapsto s_a, (u,v) \mapsto (u,v,s^u_v), (u,v,w) \mapsto (u,w)$$ was shown to be $2$-functorial in the construction of the $2$-functor $K:RE_{c}(\C) \to Ps(RE,\C)$ of Lemma~\ref{lem:pullback} and $2$-functoriality of $I$, which combines this with postcomposition, easily follows.

%
%

It remains to show that $I$ and $P$ are inverse halves of a $2$-equivalence.  Firstly, observe that given $f:A \to B \in \C$ we have that $q_{f} \circ s_{f} = f$.  Thus $I \circ P = Id$ on objects, and it is straightforward to show that the equality holds for $1$-cells and $2$-cells as well.  It remains then to describe a $2$-natural isomorphism $P \circ I \cong Id$.  

To this end, consider $(\overline{a},R,b) \in [S,\C]_f$ and $I(\overline{a},R,b) = b \circ s_a:A \to R \to B$.  Now we have the composite invertible $2$-cell on the left below
$$\xy
(0,0)*+{R}="a0"; (-20,-15)*+{A}="b0";(0,-15)*+{R}="d0";(20,-15)*+{B}="c0";
{\ar_{a} "a0"; "b0"}; 
{\ar^{b} "a0"; "c0"}; 
{\ar^{1} "a0"; "d0"}; 
{\ar_{s_a} "b0"; "d0"}; 
{\ar_{b} "d0"; "c0"}; 
(-5,-10)*{{\cong}_{\rho_{a}}};
(60,0)*+{R}="a0"; (40,-15)*+{A}="b0";(80,-15)*+{B}="c0";(60,-15)*+{P_{b \circ s_a}}="d0";
{\ar_{a} "a0"; "b0"}; 
{\ar^{b} "a0"; "c0"}; 
{\ar^{\eta_{R}} "a0"; "d0"}; 
{\ar^{p_{b \circ s_a}} "d0"; "b0"}; 
{\ar_{q_{b \circ s_a}} "d0"; "c0"}; 
\endxy$$
which induces a unique $1$-cell $\eta_R$ giving a span map as above and whose composite with $\lambda_{b \circ s_a}$ equals $b \circ \rho_a$.  $2$-naturality of $\eta_{R}:(a,R,b) \to PI(a,R,b)$ follows from the universal property of the pseudolimit of an arrow.  Finally, we should show that $\eta_R$ is invertible.  Since $p_{b \circ s_a}$ and $a$ are equivalences with $p_{b \circ s_a} \circ \eta_R = a$ we deduce, by three from two, that $\eta_R$ is an equivalence.  Finally, we have a commutative triangle 
\begin{equation*}
\xymatrix{
& R \ar[dl]_{\eta_R} \ar[dr]^{(a,b)} \\
PI(a,R,b) \ar[rr]_-{(p_{b \circ s_{a}},q_{b \circ s_a})} && A \times B}
\end{equation*}
in which both morphisms to the product are discrete isofibrations.  It follows, by three from two, that $\eta_R$ is itself a discrete isofibration.  As it is also an equivalence, it is therefore invertible, completing the proof.
\end{proof}


\begin{Lemma}\label{lem:2from3}
Let $\C \in \Lpm$.
\begin{enumerate}
\item If $Ps(\A,\C) \in \Lp$ with flexible limits and filtered colimits pointwise in $\C$ then $Ps(\A,\C) \in \Lpm$.
\item Assuming $Ps(\A,\C)$ satisfies the above conditions and $\B \in \Lpm$ too, then $F:\B \to Ps(\A,\C) \in \Lpm$ just when each composite $ev_a \circ F:\B \to \C \in \Lpm$.
\end{enumerate}
\end{Lemma}
\begin{proof}
Let $Res_{\A}:Ps(\A,\C) \to Ps(obA,\C) = \C^{ob\A}$ denote the forgetful $2$-functor obtained by restriction.  Then by Lemma~\ref{lem:conservative}, we have a pullback square
\begin{equation}\label{eq:pwise}
\xymatrix{Re(Ps(\A,\C)) \ar[d]_{} \ar[rr]^-{Re(Res_{\A})} && RE([ob\A,\C]) \cong RE(\C)^{ob\A} \ar[d]_{} \\
[\atwo,Ps(\A,C)] \ar[rr]^-{[\atwo,Res_{\A}]} && [\atwo,Ps(ob\A,\C)] \cong [\atwo,\C]^{obA} }
\end{equation}
whose right vertical leg is, furthermore, an isofibration.  We must prove that the left vertical leg belongs to $\Lp$.  The right vertical leg does so, since $\Lp$ is closed under powers by small categories.  The bottom leg above belongs to $\Lp$ using our assumption on $\C$ and that $\Lp$ has powers by small categories.  Both of these legs are isofibrations.  Therefore, by Proposition~\ref{prop:limits}, the pullback and pullback projections belong to $\Lp$ proving (1).

For (2), since the evaluation $2$-functors $ev_a$ jointly reflect isomorphisms by Lemma~\ref{lem:conservative} and preserve filtered colimits and flexible limits by assumption, they also reflect them.  The claim then follows immediately.
\end{proof}

\begin{Prop}\label{prop:acc2}
If $\C \in \Lpm$ then $Ps(\atwo,\C)\in \Lpm$ with flexible limits and filtered colimits pointwise as in $\C$.
\end{Prop}
\begin{proof}
Let $[S,\C]_f \subseteq [S,\C]$ denote the full sub-2-category consisting of those spans satisfying the conditions of Theorem~\ref{thm:span}.  To capture the discrete isofibration condition, we will use that $f:A \to B$ in a 2-category $\C$ is a discrete isofibration if and only if the induced map $Rf:A^{I} \to Pf$ is invertible.  Since powers by the free isomorphism $I$ and pseudolimits of arrows are both finite flexible limits the corresponding $2$-functor $R:\C^{\atwo} \to \C^{\atwo}$ belongs to $\Lp$.  Let $B:\C^{\atwo} \to \C^{\atwo} \times \C^{\atwo}$ be the $2$-functor sending a span $X \leftarrow A \rightarrow Y$ to the pair $(X \leftarrow A, R(A \to X \times Y))$.  Then we have 
a pullback square as on the right below 
\begin{equation*}
\xymatrix{Ps(\atwo,\C) \ar[dr]_{P} \ar[r]^{\simeq} & [S,\C]_f \ar[d]_{} \ar[r]^{} & RE(\C) \times Iso(\C) \ar[d]^{} \\
& [S,\C] \ar[r]^{B} & \C^{\atwo} \times \C^{\atwo}}
\end{equation*}
where the right vertical morphism is the product of the inclusions.  The $2$-category $[S,\C]$ belongs to $\Lp$ since $\Lp$ admits powers by small categories.  Since finite products are finite flexible limits and $R$ belongs to $\Lp$ so does the bottom leg $B$.

Now the $2$-functor $\C \to \C^{\atwo}$ sending an object to the corresponding identity factors through the full sub-$2$-category of isomorphisms $Iso(\C) \to C^{\atwo}$ via the $2$-equivalence $\C \hookrightarrow Iso(\C)$  --- since $\C \to \C^{\atwo} \in \Lp$ therefore $Iso(\C) \to C^{\atwo} \in \Lp$ too.  By assumption $RE(\C) \to \C^{\atwo} \in \Lp$.  Thus the right vertical leg belongs to $\Lp$ and it is also an isofibration.  Therefore, by Proposition~\ref{prop:limits}, the entire pullback lives in $\Lp$.  

Composing with the $2$-equivalence $Ps(\atwo,\C) \simeq [S,\C]_{f}$ of Theorem~\ref{thm:span}, it follows that the diagonal composite $P:Ps(\atwo,\C) \to [S,\C]$ belongs to $\Lp$.  Postcomposing this by the $2$-functors $[S,\C] \to \C \in \Lp$ selecting the source and target of a span establishes the pointwise nature of flexible limits and filtered colimits in $Ps(\atwo,\C)$.  Then by Lemma~\ref{lem:2from3} $Ps(\atwo,\C) \in \Lpm$.
\end{proof}

\subsection{The more general case $\atwo_{\A}$.}

We now generalise the above slightly, in a manner that will be sufficient to capture the sources and targets of the generating cofibrations $\{J_i:P_{i-1} \to D_i:i = 1,2,3\}$.  


\begin{Def}
Let $\A$ be a category.  Let $\atwo_{\A}$ be the $2$-category with objects $0,1$ and homs $\atwo_{\A}(0,0)=\atwo_{\A}(1,1)=1$, $\atwo_{\A}(0,1)=A$ and $\atwo_{\A}(1,0)=\emptyset$.
\end{Def}

Then $\atwo = \atwo_{1}$, whilst $2 = \atwo_{\emptyset}$.  A $2$-functor $(X,Y,\theta):\atwo_{\A} \to \C$ is specified by a pair of objects $X,Y$ of $\C$ and a functor $\theta:\A \to \C(X,Y)$; in these terms a morphism $(X,Y,\theta) \to (X^{'},Y^{'},\theta^{'})$ of $Ps(\atwo_{\A},\C)$ is specified by arrows $\alpha_X:X \to X'$ and $\alpha_Y:Y \to Y'$ together with a natural isomorphism as below.
\begin{equation}\label{eq:local}
\begin{xy}
(0,0)*+{\A}="00";(30,0)*+{\C(X,Y)}="10";(0,-15)*+{\C(X',Y')}="01";(30,-15)*+{\C(X,Y')}="11";
{\ar^{\theta} "00"; "10"};{\ar_{\theta^{'}} "00"; "01"};{\ar^{\C(X,\alpha_Y)} "10"; "11"};{\ar_{\C(\alpha_X,Y)} "01"; "11"};
(15,-8)*+{\cong^{\alpha}};
\end{xy}
\end{equation}

Now let $\C$ be a $2$-category with powers by $\A$.  In this case an object $(X,Y,\theta):\atwo_{\A} \to \C$, as specified by a functor $\theta:\A \to \C(X,Y)$, equally corresponds a map $X\to (Y)^\A$.   In this way, we see that we have a pullback square
\begin{equation}\label{eq:power}
\xymatrix{Ps(\atwo_{\A},\C) \ar[d]_{U} \ar[r]^{P} & Ps(\atwo,\C) \ar[d]^{U} \\
\C^{2} \ar[r]^{(1,(-)^{\A})} & \C^{2}}
\end{equation}
in which the top horizontal leg sends $\theta:\A \to \C(X,Y)$ to the corresponding map $X \to Y^{\A}$.

\begin{Prop}\label{prop:basic}
Let $\C \in \Lp_{\mathcal M}$ and $\A$ be a finitely presentable category.
\begin{enumerate}
\item Then $Ps(\atwo_{\A},\C) \in \Lpm$ with flexible limits and filtered colimits pointwise as in $\C$.
\item Given an injective on objects functor between finitely presentable categories $F:\A \to \B$ the 2-functor $Ps(\atwo_{F},\C):Ps(\atwo_{\B},\C) \to Ps(\atwo_{\A},\C)$ is an isofibration in $\Lpm$.
\end{enumerate}
\end{Prop}
\begin{proof}
In the pullback square \eqref{eq:power} the right vertical map belongs to $\Lpm$ by Proposition~\ref{prop:acc2} and is an isofibration.  The bottom leg belongs to $\Lpm$ since powers by a finitely presentable category are finite flexible limits.  Therefore the whose diagram belongs to $\Lpm$ by Proposition~\ref{prop:limits}.  

For the second part, we use the first part and  Lemma~\ref{lem:2from3} to conclude that $Ps(F,\C)$ belongs to $\Lpm$.  To see that it is an isofibration consider a morphism $(X,Y,\theta) \to (X^{'},Y^{'},\theta^{'})$ of $Ps(\atwo_{\A},\C)$ as in \eqref{eq:local}.  By Lemma~\ref{lem:conservative}, this is invertible in $Ps(\atwo_{\A},\C)$ just when when $\alpha_X$ and $\alpha_Y$ are invertible. 
Thus an isomorphism $F^*(X,Y,\theta) \cong (X^{'},Y^{'},\theta^{'})$ is given by a diagram as on the left below. 
\begin{equation*}
\begin{xy}
(-15,0)*+{\A}="-10";(0,0)*+{\B}="00";(30,0)*+{\C(X,Y)}="10";(0,-15)*+{\C(X',Y')}="01";(30,-15)*+{\C(X,Y')}="11";
{\ar^{\theta} "00"; "10"};{\ar^{\C(X,\alpha_Y)} "10"; "11"};{\ar_{\C(\alpha_X,Y)} "01"; "11"};
(10,-8)*+{\cong^{\alpha}};
{\ar^{F} "-10"; "00"};{\ar_{\theta'} "-10"; "01"};
\end{xy}
\hspace{1cm}
\begin{xy}
(-15,0)*+{\A}="-10";(0,0)*+{\B}="00";(30,0)*+{\C(X,Y)}="10";(0,-15)*+{\C(X',Y')}="01";(30,-15)*+{\C(X,Y')}="11";
{\ar^{\theta} "00"; "10"};{\ar^{\C(X,\alpha_Y)} "10"; "11"};{\ar_{\C(\alpha_X,Y)} "01"; "11"};
(15,-8)*+{\cong^{\alpha'}};
{\ar^{F} "-10"; "00"};{\ar_{\theta'} "-10"; "01"};{\ar^{\theta''} "00"; "01"};
\end{xy}
\end{equation*}
with $\alpha_X$ and $\alpha_Y$ invertible.  Since $f$ is injective on objects $\Cat(F,\C(X,Y')):\Cat(\B,\C(X,Y')) \to \Cat(\A,\C(X,Y'))$ is an isofibration by Theorem 2 of \cite{Joyal1993Pullbacks}; thus we can lift $\alpha$ obtain a diagram as on the right above, in which $\alpha^{'} \circ F = \alpha$, proving the claim.
\end{proof}



\subsection{The general case of a cellular 2-category $\A$.}
Finally, we are in a position to prove the main result of this section.

\begin{Theorem}\label{thm:cellular}
Let $\C \in \Lpm$.  
\begin{enumerate}
\item If $\A$ is cellular then $Ps(\A,\C) \in \Lpm$ with flexible limits and filtered colimits pointwise in $\C$.
\item Furthermore, if $F:\A \to \B$ is a cellular 2-functor between cellular $2$-categories then $Ps(F,\C):Ps(\B,\C) \to Ps(\A,\C)$ is an isofibration in $\Lpm$.
\end{enumerate}
\end{Theorem}
\begin{proof}
Firstly, we show that 
\begin{itemize}
\item
For each generating cofibration $J_i:P_{i-1} \to D_i$ the $2$-functor $Ps(J_i,\C):Ps(D_i,\C) \to Ps(P_{i-1},\C)$ is an isofibration in $\Lpm$, whose source and target have flexible limits and filtered colimits pointwise as in $\C$.
\end{itemize}
To this end, observe that the case of $J_0$ is trivial since the resulting restriction $2$-functor is $\C \to 1$, whilst the other cases are special cases of Proposition~\ref{prop:basic}.  
 
Now the cellular 2-category $\A$ fits into a diagram $$\varnothing = \A_{0} \to \A_1 \to \A_2 \to \A_3 \to \A_ 4 = \A$$ in which each $\A_i \to \A_{i+1}$ is a pushout of a copower of $J_i$.  Certainly $Ps(\A_{0},\C) = 1 \in \Lpm$ and trivially satisfies the pointwise condition.  Therefore, it suffices to prove that
\begin{itemize}
\item if Theorem~\ref{thm:cellular} holds for $\A_i$ then it holds also for $\A_{i+1}$.
\end{itemize}

To this end, consider the defining pushout below
\begin{equation}\label{eq:pushout}
\xymatrix{X.P_{i-1} \ar[d] \ar[r]^{X.J_i} & X.D_i \ar[d] \\
\A_i \ar[r] & \A_{i+1}}
\hspace{1cm}
\xymatrix{Ps(\A_{i+1},\C) \ar[d] \ar[rr] && Ps(\A_{n},\C) \ar[d] \\
Ps(D_i,\C)^X \ar[rr]_{Ps(J_{i},\C)^X} && Ps(P_{i-1},\C)^X}
\end{equation}
The bottom leg is a product of isofibrations in $\Lpm$ and, since $\Lpm$ is closed under products, an isofibration in $\Lpm$.  To show that the right leg belongs to $\Lpm$ is equally to show that its composite $Ps(\A_i,\C) \to Ps(P_{i-1},\C)$ with each product projection belongs to $\Lp_{\mathcal M}$, and this latter claim holds by Lemma~\ref{lem:2from3} applied to $P_{i-1} \to X.P_{i-1} \to \A_i$ where $P_{i-1} \to X.P_{i-1}$ is the appropriate injection to the copower.  Therefore by Theorem~\ref{thm:limits2} the pullback square belongs to $\Lpm$.  Furthermore, as the pullback of an isofibration, the $2$-functor $Ps(\A_{i+1},\C) \to Ps(\A_i,\C)$ is an isofibration in $\Lpm$.

For the pointwise condition, we must show that $ev_{x}:Ps(\A_{i+1},\C) \to \C \in \Lpm$ for each $x \in \A_{i+1}$.  Since the pushout projections in \eqref{eq:pushout} are jointly surjective on objects there are two possibilities; (1) that $x$ is the image of $y \in \A_i \to \A_{i+1}$, in which case $ev_x$ factors as the composite of $Ps(\A_{i+1},\C) \to Ps(\A_{i},\C) \in \Lpm$ and $ev_y:Ps(\A_i,\C) \to \C \in \Lp_{\mathcal M}$.  Or (2) that $x$ is the image of $y \in X.D_{i}\to \A_{i+1}$, in which case $ev_x$ factors as the composite of  $Ps(\A_{i+1},\C) \to Ps(X.D_{i},\C) \in \Lpm$ and $ev_y:Ps(X.D_{i},\C) \to \C$.  The latter map belongs to $\Lpm$ since we have an isomorphisms $Ps(X.D_{i},\C) \cong Ps(D_{i},\C)^{X}$ whilst filtered colimits and flexible limits are pointwise in $Ps(D_{i},\C)$ and so in the product $Ps(D_{i},\C)^{X}$.

This completes the proof of Part 1 of the theorem.  For the second part, Lemma~\ref{lem:2from3} ensures that $Ps(F,\C)$ belongs to $\Lpm$.  Furthermore, it is a countable cocomposite of powers of pullbacks of the maps $Ps(J_i,\C)$.  As established above, each such map is an isofibration and since the isofibrations can be characterised by having a right lifting property, they are stable under the above limit constructions --- hence $Ps(F,\C)$ is an isofibration.

\end{proof}


\section{A variety of examples and non-examples}\label{sect:examples}
In the present section we illustrate the power of Theorem~\ref{thm:cellular} by using it to show that a variety of examples of categories equipped with coherent structure and their pseudomorphisms form accessible $2$-categories --- indeed, that they belong to $\Lpm$.

\subsection{Categories with coherent structure, and pseudomorphisms}\label{sect:monadic1}
For simplicity, let us give the example of semi-monoidal categories.  This example is easily adapted to capture structures such as monoidal categories, symmetric monoidal categories and so on, and we will describe general results of this nature in Section~\ref{sect:2monad}.

A semi-monoidal category $X$ comes equipped with a bifunctor $\otimes:X^2 \to X:(a,b) \mapsto ab$ and a natural transformation $\alpha_{a,b,c}:(ab)c \to a(bc)$ satisfying Maclane's pentagon equation.  A pseudomorphism between semimonoidal categories is a functor $f:X \to Y$ together with a natural isomorphism $f_{a,b}:f(ab) \cong (fa)(fb)$ satisfying the same compatibility with the associator as in the notion of a strong monoidal functor.  There is also the evident notion of 2-cell.  

We form the 2-category of these structures in three steps.  Firstly, we form the pullback below.

\begin{equation*}
\xymatrix{\Alg_{1} \ar[d]^{U_{1}} \ar[rr] && Ps(D_1,\Cat) \ar[d]^{Ps(J_{1},\Cat)} \\
\Cat \ar[rr]_-{R_{1}:X \mapsto (X^2,X)} && Ps(P_0,\Cat)}
\hspace{1.5cm}
\begin{xy}
(0,0)*+{X^2}="00";(20,0)*+{Y^2}="10";(0,-15)*+{X}="01";(20,-15)*+{Y}="11";
{\ar^{f^2} "00"; "10"};{\ar_{m_X} "00"; "01"};{\ar^{m_Y} "10"; "11"};{\ar_{f} "01"; "11"};
(10,-8)*+{\cong^{\overline{f}}};
\end{xy}
\end{equation*}
The objects of the pullback $\Alg_1$ are categories equipped with a bifunctor; the morphisms are functors together with a natural isomorphism as depicted above. 
The right leg is an isofibration of $\Lp_{\mathcal M}$ by Theorem~\ref{thm:cellular}.  Since finite products are finite flexible limits the bottom leg belongs to $\Lp_{\mathcal M}$.  Therefore by Theorem~\ref{thm:limits2} the whole pullback diagram belongs to $\Lp_{\mathcal M}$.  

We now add the associator by forming the pullback below
\begin{equation*}
\xymatrix{\Alg_{2} \ar[d] \ar[rr] && Ps(I_2,\Cat) \ar[d]^{Ps(I_1,\Cat)} \\
\Alg_1 \ar[rr]_{R_2} && Ps(P_1,\Cat)}
\end{equation*}
in which $I_1:P_2 \to I_2$ is the inclusion of the boundary of the free invertible $2$-cell and the bottom leg $R_2$ sends $(X,\otimes)$ to $((\otimes \circ (\otimes,1), \otimes \circ (1,\otimes)):X^3 \rightrightarrows X \in Ps(P_1,\Cat)$.  In the pullback, we have categories equipped with a bifunctor and associator isomorphisms $\alpha_{a,b,c}:(ab)c \cong a(bc)$; the morphisms now satisfy the desired commutativity with the associators.  

Again, by Theorem~\ref{thm:cellular}, the right leg is an isofibration of $\Lp_{\mathcal M}$.  To check that $R_2 \in \Lpm$ it is enough, by Lemma~\ref{lem:2from3}(2), to check that it belongs to $\Lpm$ componentwise.  Its first component  $(-)^{3} \circ U_1:\Alg_1 \to \Cat:(X,\otimes) \mapsto X^3$ is a composite of two morphisms in $\Lpm$ whilst its second is just $U_1$.  Therefore $R_2 \in \Lpm$ so that by Theorem~\ref{thm:limits2} pullback belongs to $\Lp_{\mathcal M}$. 

Finally we consider the pullback square

\begin{equation*}
\xymatrix{\Alg_{3} \ar[d] \ar[rr] && Ps(D_2,\Cat) \ar[d]^{Ps(J_{3},\Cat)} \\
\Alg_2 \ar[rr]_{R_3} && Ps(P_2,\Cat)}
\end{equation*}

in which $R_3:\Alg_2  \to Ps(P_2,\Cat)$ sends $(X,\otimes,\alpha)$ to the pair of parallel natural transformations in $\Cat(X^4,X)$ with components $((ab)c)d \rightrightarrows a(b(cd))$ the two paths of the pentagon.  In the pullback $\Alg_3$ these are forced to be equal, so that we obtain  semi-monoidal categories as desired.  Arguing as before, we see that the pullback square lies in $\Lp_{\mathcal M}$.

\subsection{Inaccessibility of strict monoidal categories and strong monoidal functors}\label{sect:strictmonoidal}
The above construction cannot be adapted to show that the $2$-category $\SMonCat$ of \emph{strict monoidal categories and strong monoidal functors} belongs to $\Lpm$.  For in order to add the strict associativity equation, one needs to form a pullback along $Ps(\nabla,\Cat)$ where  $\nabla:P_1 \to D_1$ is the codiagonal which identifying the two parallel $1$-cells of $P_1$.   However, the $2$-functor $\nabla$ is \emph{not cellular} so Theorem~\ref{thm:cellular} does not apply.  Indeed, we will show now that \emph{idempotents do not split} in $\SMonCat$ so that it is not accessible of any degree.  For a general result of this flavour, see Theorem~\ref{thm:limitations}.

Let $\MonCat_p$ denote the $2$-category of monoidal categories and strong monoidal functors and consider a monoidal category $X$.  It is well known that MacLane's coherence theorem \cite{MacLane1963} enables the construction of a strict monoidal category $QX$ together with an equivalence $QX \to X$ in $\MonCat_p$.  Explicitly, the objects of $QX$ are words $\overline{a} = [a_1,\ldots,a_n]$ of objects in $X$.  Let $l[a_1,\ldots,a_n] \in X$ denote the \emph{left bracketed} tensor product, defined inductively by $l[-] = I$, $l[a] = a$ and $l[a_1,\ldots,a_n,a_{n+1}] = l[[a_1,\ldots,a_n] \otimes a_{n+1}$.  By definition, a morphism $f:\overline{a} \to \overline{b}$ in $QX$ is a morphism $f:l(\overline{a}) \to l(\overline{b}) \in X$.  Now $QX$ admits a strict monoidal structure extending the free monoid structure on its object set.  The tensor product of morphisms $f:l(\overline{a}) \to l(\overline{b})$ and $g:l(\overline{c}) \to l(\overline{d})$ is the composite
\begin{equation*}
\xymatrix{
l(\overline{a}\overline{b}) \ar[r]^{\cong} & l(\overline{a})l(\overline{b}) \ar[r]^{f \otimes g} & l(\overline{c})l(\overline{d}) \ar[r]^{\cong} & l(\overline{c}\overline{d})
}
\end{equation*}
whose first and last components are the unique structural isomorphisms arising from the coherence theorem.  

By construction $l:QX \to X$ is a surjective equivalence of categories.  It is also a normal strong monoidal functor with coherence constraints $l(\overline{a}\overline{b}) \cong l(\overline{a})l(\overline{b})$ the structural isomorphisms.  Furthermore $l$ has a strong monoidal section $s:X \to QX$ sending $a$ to $[a]$.  Therefore $s \circ l:QX \to QX$ is an idempotent in $\SMonCat_p$ with splitting $X$ in $\MonCat_p$.  

Now if the idempotent split in $\SMonCat_p$ then, since split idempotents are preserved by any $2$-functor and in particular the inclusion $\SMonCat_{p} \hookrightarrow \MonCat_{p}$, it would follows that $X$ was \emph{isomorphic} in $\MonCat_p$ to a strict monoidal category.  Therefore, if we can exhibit a small monoidal category which is not isomorphic to a strict one, the conclusion will be that idempotents do not split in $\SMonCat_p$.



To this end, consider the skeleton $S_{\leq \omega}$ of the category of at most countable sets.  It admits a choice of finite products $a\times b$ and so an associated monoidal structure, where the associator and unit isomorphisms are determined by the universal properties of the projections.  A monoidal category, arising from a cartesian category in this way, is called \emph{cartesian monoidal}.  Suppose that $(F,f,f_{0}):(S_{\leq \omega},\times,1) \to (X,\otimes,i) \in \MonCat_p$ is an isomorphism of monoidal categories with $X$ strict monoidal.  Then by transport of structure along the invertible $F$ we obtain an isomorphic strict monoidal structure $(X,\otimes,i) \cong (S_{\leq \omega},\star,J)$ such that the composite isomorphism of monoidal categories $(1,g,g_0):(S_{\leq \omega},\times,1) \cong (X,\otimes,i) \cong  (S_{\leq \omega},\star,J)$ has underlying functor the identity.  Note that $J=1$ since $S_{\leq \omega}$ is skeletal, and that $g_0:1 \to 1$ is the identity since $1$ is terminal; since the unit is terminal the second monoidal structure is a so-called \emph{semicartesian} monoidal category.  Now it is straightforward to show that a semicartesian monoidal structure is \emph{cartesian} if and only if the projections depicted on the bottom row below form product projections.  

\begin{equation*}
\xymatrix{
a \ar[d]_{1} & a \times 1 \ar[d]|{g_{a,1}} \ar[l]_{\cong} & a\times b \ar[d]|{g_{a,b}} \ar[l]_{a \times !} \ar[r]^{! \times b} & 1 \times b \ar[d]|{g_{1,b}} \ar[r]^{\cong} & b \ar[d]^{1} \\
a & a \star 1 \ar[l]^{=} & a\star b \ar[l]^{a \star !} \ar[r]_{! \star b} & 1 \star b \ar[r]_{=} & b }
\end{equation*}

They do so, in this case, since we have an isomorphism of rows, and the upper row is a product.  It then follows routinely that $(S_{\leq \omega},\star,1)$ is the monoidal structure associated to the terminal object $1$ and choice of products $a \star b$ with projections the above ones.  Now this cartesian monoidal structure on $S_{\leq \omega}$ is strict monoidal.  However, by a clever argument of Isbell (see page 164 of \cite{CWM}) such a cartesian strict monoidal structure cannot exist.

Therefore idempotents do not split in $\SMonCat_p$.  Indeed, since idempotents do split in $\MonCat_p$ and the inclusion $\SMonCat_p \to \MonCat_p$ is surjective up to retracts, what we have shown is that $\MonCat_p$ is the \emph{idempotent completion} of $\SMonCat_p$ --- thus $\SMonCat_p$ provides a \emph{natural example of an inaccessible category whose idempotent completion is accessible}.  

\begin{Remark}
One can similarly prove that the $2$-category $\UMonCat_p$ of unbiased monoidal categories is the idempotent completion of $\SMonCat_p$ --- by the uniqueness of idempotent completions, this gives a simple proof that $\MonCat_p$ and $\UMonCat_p$ are $2$-equivalent.
\end{Remark}

\subsection{Bicategories and homomorphisms}\label{sect:monadic2}
This time we consider $\Bicat_2$, the 2-category of small bicategories, pseudofunctors and icons \cite{Lack2007Icons}.  The modifications required to the previous example are fairly minor.  To begin with, we consider $[P_{1},\Cat]$ --- the 2-category of internal graphs in $\Cat$.  Each bicategory $X$ has an underlying graph $s,t:X_1 \rightrightarrows X_0$ in $\Cat$ in which $X_0$ is its set of objects and $X_1(a,b)=X(a,b)$.  To impose the discreteness condition on $X_0$, we form the pullback below.
\begin{equation*}
\xymatrix{
\CatGph_2 \ar[r] \ar[d] & [P_1,\Cat] \ar[d]^{ev_{0}} \\
\Set \ar[r]_{D} & \Cat}
\end{equation*}
in which $D$ sends a set to the corresponding discrete category.  Since the right leg is an isofibration the pullback is a bipullback.  Since this leg and the bottom leg are finitary right adjoints between locally finitely presentable $2$-categories, it follows from Theorem 6.11 of \cite{Bird1984Limits} that the pullback is locally finitely presentable and that the pullback projections are finitary right adjoints.  In particular $\CatGph_2 \in \Lp$.  Now an internal surjective equivalence $F:X \to Y \in \CatGph_2$ is just a bijective on objects functor which is a local surjective equivalence: that is, each morphism $X(a,b) \to Y(Fa,Fb) \in \Cat$ is a surjective equivalence of categories.  The local surjective equivalences are cofibrantly generated by the set $\atwo_{I}=\{\atwo_{i}:\atwo_{\A} \to \atwo_{B}:i:A \to B \in I\}$ where $I$ is the set of generating cofibrations in $\Cat$ described in Section~\ref{sect:catCof}.  Adjoining to these the two morphisms $\varnothing \to 1$ and $1 +1 \to 1$ of discrete $\Cat$-graphs, we see that the surjective equivalences in $\CatGph_2$ are generated by a set of morphisms between finitely presentable objects.  Therefore $\CatGph_2 \in \Lpm$ by Example~\ref{ex:1}.

Given a $\Cat$-graph $X$ let us form the pullback below.
\begin{equation*}
\xymatrix{
X_1 \times_{X_{0}} X_1 \ar[r]^-{p} \ar[d]_{q} \ar[d] & X_1 \ar[d]^{t}\\
X_1 \ar[r]_{s} & X_0}
\end{equation*}
We note that since $X_0$ is discrete this pullback is equally the pseudo-pullback and so a finite flexible limit.  The data for a bicategorical composition map is a functor $m_X:X_1 \times_{X_{0}} X_1 \to X_1$, and such composition maps are added by forming the pullback below.
\begin{equation*}
\xymatrix{
\Alg_1 \ar[rrr] \ar[d] &&& Ps(P_1,\Cat) \ar[d]^{Ps(J_1,\Cat)} \\
\CatGph_2 \ar[rrr]^{X \mapsto (X_{1} \times_{X_0} X_{1},X_{1})} &&& Ps(P_0,\Cat)}
\hspace{0.5cm}
\begin{xy}
(0,0)*+{X_1 \times_{X_{0}} X_1}="00";(30,0)*+{Y_1 \times_{Y_{0}} Y_1}="10";(0,-15)*+{X_1}="01";(30,-15)*+{Y_1}="11";
{\ar^{f_1 \times _{f_{0}} f_1} "00"; "10"};{\ar_{m_X} "00"; "01"};{\ar^{m_Y} "10"; "11"};{\ar_{f_1} "01"; "11"};
(15,-8)*+{\cong^{\overline{f}}};
\end{xy}
\end{equation*}
Since the bottom leg involves only finite flexible limits, it belongs to $\Lpm$ by Lemma~\ref{lem:2from3}(2).  The right leg is an isofibration in $\Lpm$ by  Theorem~\ref{thm:cellular}, so that by Theorem~\ref{thm:limits2} the pullback belongs to $\Lpm$.  A morphism of $\Alg_1$ then consists of a morphism of $\CatGph_2$ together with a natural isomorphism as above right.  The composition map should satisfy the equations $s_X \circ m_X = s_X \circ p_X$ and $t_X \circ m_X = t_X \circ q_X$ which can be added by forming the pullback
\begin{equation*}
\xymatrix{
\Alg_2 \ar[rrrr] \ar[d] &&&& Ps(D_2,\Cat)^2 \ar[d]^{Ps(J_2,\Cat)^2} \\
\Alg_1 \ar[rrrr]^-{((s_X m_X, s_X  p_X),(t_X m_X, t_X  q_X))} &&&& Ps(P_1,\Cat)^2}
\end{equation*}
which, arguing as before, is easily seen to lie in $\Lpm$.  We add the associativity isomorphism by forming a pullback on the left below
\begin{equation*}
\xymatrix{
\Alg_3 \ar[r] \ar[d] & Ps(I_2,\CatGph) \ar[d]^{Ps(I_1,\CatGph)} \\
\Alg_2 \ar[r]^-{R} & Ps(P_1,\CatGph)}
\hspace{0.5cm}
\xymatrix{
X_1 \times_{X_{0}} \times X_{1} \times _{X_{0}}\times X_{1} \ar[rr]<0.5ex>^-{m_{x} \circ (m_{x} \times 1)} \ar[rr]<-0.5ex>_-{m_{x} \circ (1 \times m_x)} \ar[d]<0.5ex> \ar[d]<-0.5ex> && X_{1} \ar[d]<0.5ex> \ar[d]<-0.5ex> \\
X_0 \ar[rr]<0.5ex>^{1} \ar[rr]<-0.5ex>_{1} && X_0}
\end{equation*}
where $R$ sends $C$ to the parallel pair of $\CatGph$ morphisms above right and finally a further pullback
\begin{equation*}
\xymatrix{
\Alg_4 \ar[rr] \ar[d] && Ps(D_2,\CatGph) \ar[d]^{Ps(J_3,\CatGph)} \\
\Alg_3 \ar[rr]^-{} && Ps(P_2,\CatGph)}
\end{equation*}
adding the pentagon equation.  We leave it to the reader to handle the units.  It then follows that $\Bicat_2$ belongs to $\Lpm$.  

Let $\twocat_2 \hookrightarrow \Bicat_2$ denote the full sub-2-category of small $2$-categories.  Arguing as in Section~\ref{sect:strictmonoidal}, but now using the coherence theorem for bicategories \cite{MacLane1982}, we see that $\twocat_2$ does not have split idempotents and that $\Bicat_2$ is its idempotent completion.  

\subsection{Completeness and cocompleteness}\label{sect:monadic3}

Let $A$ be a small category and $p:A \to 1$ the unique map.   A small category $X$ has $A$-limits just when $p^*:[1,X] \to [A,X]$ has a right adjoint $lim$ which equips it with a choice of $A$-limits.  Moreover a morphism $F:X \to Y$ between two such categories preserves the limits just when the commutative square on the left below forms
\begin{equation*}
\begin{xy}
(0,0)*+{X^{1} }="00";(20,0)*+{X^{A}}="10";(0,-15)*+{Y^{1}}="01";(20,-15)*+{Y^{A}}="11";
{\ar^{p^{*}} "00"; "10"};{\ar_{F_{*}} "00"; "01"};{\ar^{F_{*}} "10"; "11"};{\ar_{p^{*}} "01"; "11"};
\end{xy}
\hspace{2cm}
\begin{xy}
(0,0)*+{X^{A} }="00";(20,0)*+{X^{1}}="10";(0,-15)*+{Y^{A}}="01";(20,-15)*+{Y^{1}}="11";
{\ar^{lim} "00"; "10"};{\ar_{F_{*}} "00"; "01"};{\ar^{F_{*}} "10"; "11"};{\ar_{lim} "01"; "11"};
(10,-8)*+{\cong};
\end{xy}
\end{equation*}
part of a pseudomorphism of adjunctions --- this is equally to say that its \emph{mate}, the natural transformation above right,  is invertible. 

Now let $A = \{A_i:i \in I\}$ be a set of finitely presentable categories --- we let $\ALim$ denote the 2-category of small categories equipped with $A$-limits and functors preserving them, together with all natural transformations between these.  

 We then have a pullback
\begin{equation*}
\xymatrix{
\ALim \ar[d] \ar[r] & Ps(\Adj,\Cat)^{I} \ar[d]^{Ps(f,1)^{I}} \\
\Cat \ar[r]^-{K} & Ps(\atwo,\Cat)^{I}
}
\end{equation*}
where $K$ sends $X$ to the family $(p^*:X^{1} \to X^{A_{i}})_{i \in I}$ and $f:\atwo \to \Adj$ is the cellular morphism selecting the left adjoint $0 \to 1$ in the free adjunction (see Section~\ref{sect:cellular2}).  The right leg is an isofibration of $\Lpm$ by Theorem~\ref{thm:cellular}; since powers by each $A_i$ are finite flexible limits the bottom leg is a morphism of $\Lpm$.  By Theorem~\ref{thm:limits2} the whole pullback diagram belongs to $\Lpm$.  

As a particular instance of this, we see that the 2-category $\Lex$ of small finitely complete categories belongs to $\Lpm$ as does the forgetful 2-functor $U:\Lex \to \Cat$.  A similar argument captures categories equipped with classes of finite colimits or, more generally, categories admitting a set of left or right Kan extensions.  


\subsection{Regular categories and exactness properties}\label{sect:monadic4}
A category is said to be \emph{regular} if it admits finite limits and coequalisers of kernel pairs, and regular epimorphisms are stable under pullback.  A regular functor between regular categories is one preserving finite limits and coequalisers of kernel pairs.  We will now show that the $2$-category $\Reg$ of small regular categories and regular functors belongs to $\Lpm$.  We also describe the modification required to cover Barr-exact categories and these examples can easily be adapted to cover coherent categories, pretopoi and similar structures.

Firstly, let us remind the reader of a few facts about relative adjoints (equivalently, absolute left liftings).  Given a diagram in $\Cat$ on the left below
\begin{equation*}
\begin{xy}
(30,0)*+{B}="10";(0,-20)*+{A}="01";(30,-20)*+{C}="11";
;{\ar^{g} "10"; "11"};{\ar_{j} "01"; "11"};
\end{xy}
\hspace{2cm}
\begin{xy}
(30,0)*+{B}="10";(0,-20)*+{A}="01";(30,-20)*+{C}="11";
{\ar^{f} "01"; "10"};{\ar^{g} "10"; "11"};{\ar_{j} "01"; "11"};
{\ar@{=>}^{\theta}(20,-17)*+{};(20,-10)*+{}};
\end{xy}
\end{equation*}
we say that the $2$-cell $\theta$ exhibits $f$ as left adjoint \emph{relative to $j$} if the induced restriction map $$B(fa,b) \to C(ja,gb): \alpha \mapsto g\alpha \circ \theta_a$$ is a bijection for each $a \in A$ and $b \in B$.  (In $\Cat$, this amounts to providing, for each $a \in A$, an object $fa \in C$ and a morphism $\theta_a:ja \to gfa$ satisfying the universal property captured by the above bijection.)  The case of an ordinary adjunction occurs when $j=Id:C \to C$.  

In a general $2$-category $\C$, we say that a diagram as on the right above in $\C$ exhibits $f$ as left adjoint to $g$ relative to $j$ just when $\C(X,\theta)$ exhibits $\C(X,f)$ as a relative left adjoint to $\C(X,g)$ along $\C(X,j)$ for each $X \in \C$.  This condition can, of course, equally be expressed in elementary $2$-categorical terms, and coincides with the usual notion when $\C= \Cat$.  It equally amounts to the fact that $f$ is the \emph{absolute left lifting} of $g$ along $j$. 

\begin{Example}\label{ex:kernel}
Given $X \in \Lex$ we have the functor $kp:X^{\bullet \to \bullet} \to X^{\bullet \rightrightarrows \bullet}$ sending a morphism to its kernel pair.  It is easy to see that $X$ admits coequalisers of kernel pairs just when the relative left adjoint

\begin{equation}
\begin{xy}
(30,0)*+{X}="10";(0,-20)*+{X^{\bullet \rightarrow \bullet}}="01";(30,-20)*+{X^{\bullet \rightrightarrows \bullet}}="11";
{\ar^{coeq_X} "01"; "10"};{\ar^{\Delta_X} "10"; "11"};{\ar^{kp_X} "01"; "11"};
{\ar@{=>}^{}(20,-15)*+{};(20,-8)*+{}};
\end{xy}
\end{equation}
exists in $\Cat$.  
\end{Example}

The problem with relative left adjoints is that, unlike adjoints, they are not an algebraic notion in a $2$-category --- for instance, they are not preserved by any $2$-functor.  However, as the following lemma shows, they can be defined using the algebraic notion of adjunction in a $2$-category admitting comma objects.  The following lemma can be drawn out of material in Section 3.5 of \cite{Riehl2019Elements}, in particular Theorem 3.5.11.

\begin{Lemma}\label{lem:comma}
Consider a comma object in a $2$-category $\C$
\begin{equation*}
\begin{xy}
(0,0)*+{j/g}="00";(30,0)*+{B}="10";(0,-20)*+{A}="01";(30,-20)*+{C}="11";
{\ar^{p_B} "00"; "10"};{\ar_{p_A} "00"; "01"};{\ar^{g} "10"; "11"};{\ar_{j} "01"; "11"};
{\ar@{=>}^{\theta}(13,-13)*+{};(18,-7)*+{}};
\end{xy}
\end{equation*}
The following are equivalent:
\begin{enumerate}
\item $g$ has a left adjoint relative to $j$;
\item the projection $p_B:j/g \to A$ has a left adjoint with identity unit, and
\item the projection $p_B:j/g \to A$ has a left adjoint with invertible unit.
\end{enumerate}
\end{Lemma}
\begin{proof}
For the equivalence of (1) and (2), firstly observe that there is a bijection between pairs $(f,\theta)$ as on the left below
\begin{equation*}
\begin{xy}
(30,0)*+{B}="10";(0,-20)*+{A}="01";(30,-20)*+{C}="11";
{\ar^{f} "01"; "10"};{\ar^{g} "10"; "11"};{\ar_{j} "01"; "11"};
{\ar@{=>}^{\theta}(20,-17)*+{};(20,-10)*+{}};
\end{xy}
\hspace{2cm}
\begin{xy}
(30,0)*+{j/g}="10";(0,-20)*+{A}="01";(30,-20)*+{A}="11";
{\ar^{r} "01"; "10"};{\ar^{p_A} "10"; "11"};{\ar_{1} "01"; "11"};
{\ar@{=}^{}(20,-17)*+{};(20,-10)*+{}};
\end{xy}
\end{equation*}
and sections $r$ of $p_A$ as on the right above.  From $(f,\theta)$ the corresponding map $r:A \to j/g$ is the unique map to the comma object such that $p_A \circ r = 1$, $p_B \circ r = f$ and $\lambda_{A} \circ r = \theta$.  It remains to show that
\begin{itemize}
\item $\theta$ exhibits $f$ as a relative left adjoint to $g$ along $j$ just when the identity exhibits $r$ as relative left adjoint to $p_A$ along $1_A$.
\end{itemize}  
This amounts to showing that on application of each representable $\C(X,-)$ the left diagram is a relative left adjoint if and only if the right one is.  Since representables also preserve comma objects, it suffices to prove the above claim when $\C = \Cat$.  In this case $r:A \to j/g$ sends $x$ to $(x,\theta_x:jx \to gfx,fx)$ whilst the projection $p_A$ from the comma category sends a triple $(a,\alpha,b)$ to $a$.  To say that the identity exhibits $r$ as left adjoint to $p_A$ along $1_A$ (that is, as a genuine left adjoint) now asserts precisely that given the data $x \in A$, $(a,\alpha,b) \in j/g$ and $u:x \to a \in A$ there exists a unique morphism $u^{\prime}:fx \to b$ making the square
\begin{equation*}
\xymatrix{
jx \ar[d]_{ju} \ar[r]^{\theta_a} & gfx \ar[d]^{gu^{\prime}} \\
ja \ar[r]^{\alpha} & gb}
\end{equation*}
commute.  This clearly implies the condition for a relative left adjoint (on taking $u=1$) and is also implied by it.  This proves the equivalence of (1) and (2).  The equivalence of (2) and (3) follows from the fact that the projection $p_B$ is an isofibration.
\end{proof}


Now let $C$ denote the free cospan.  Then the cospan $(kp_X,\Delta_X)$ of Example~\ref{ex:kernel} is the object component of a 2-functor $K_1:\Lex \to Ps(C,\Cat)$.  Now the components of $K_1$ are finite powers -- and so finite flexible limits -- which commute with filtered colimits in $\Lex$.  Since by Theorem~\ref{thm:cellular} we have that $Ps(C,\Cat) \in \Lpm$ with filtered colimits and flexible limits pointwise, it follows from Lemma~\ref{lem:2from3} that $K_1 \in \Lpm$.  We also have the 2-functor $K_2:Ps(C,\Cat) \to Ps(\atwo,\Cat)$ sending $(j:A \to C,C \leftarrow B:g)$ to the projection $p_A:j/g \to A$ from the comma object which similarly belongs to $\Lpm$.  In particular, composing these $2$-functors we obtain $K:\Lex \to Ps(C,\Cat) \to Ps(\atwo,\Cat) \in \Lpm$.  

Now by Example~\ref{ex:kernel} and Lemma~\ref{lem:comma}, an object of the pullback 
\begin{equation*}
\xymatrix{
\Lex_{ckp} \ar[d] \ar[rr] && Ps(\Coref,\Cat) \ar[d]^{Ps(u,1)} \\
\Lex \ar[rr]^-{K} && Ps(\atwo,\Cat)
}
\end{equation*}
is a finitely complete category with coequalisers of kernel pairs; similarly, a morphism is a finite limit preserving functor preserving coequalisers of kernel pairs. By Theorem~\ref{thm:cellular} the whole diagram belongs to $\Lpm$.

Next, let $X \in \Lex_{ckp}$.  We must capture the exactness property defining regular categories: namely, that coequalisers of kernel pairs are pullback stable.  To that end, consider a cospan consisting of a horizontal and vertical leg.  We have a functor $c^2_Xp^1_X:[C,X] \to [\boxminus,X]$ which takes the cospan, and pulls it back along the horizontal leg to obtain a square, and then factors both vertical legs of the square through the coequalisers of their kernel pairs, and adds the induced map between the coequalisers yielding a composable pair of squares.  On the other hand we have a functor $p^2_Xc^1_X:[C,X] \to [\boxminus,X]$ which firstly factors the vertical leg through the coequaliser of its kernel pair, and then pulls back along the horizontal map.  It is easy to see that we have a natural transformation $\phi_X:c^2_Xp^1_X \Rightarrow p^2_Xc^1_X$ induced by the universal property of the coequaliser, and then coequalisers of kernel pairs are pullback stable --- by construction --- just when $\phi_X$ is invertible. 

Because the operations of pullback, kernel pair and coequaliser thereof, are preserved by any morphism of $\Lex_{ckp}$ the components $\phi_X$ are the object components of a 2-functor $\Lex_{ckp} \to Ps(D_2,\Cat)$ involving only finite powers of objects in $\Lex_{ckp}$, and so belongs to $\Lpm$.  The 2-category of regular categories is then the pullback
\begin{equation*}
\xymatrix{
\Reg \ar[d] \ar[rr] && Ps(I_2,\Cat) \ar[d]^{Ps(J,1)} \\
\Lex_{ckp} \ar[rr]^-{K} && Ps(D_2,\Cat)
}
\end{equation*}
where $J:D_2 \to I_2$ is the cellular quotient map from the free $2$-cell to the free invertible $2$-cell, and so belongs to $\Lpm$ by Theorem~\ref{thm:cellular}.

Let us remark that we have done above shows how to define \emph{regular object} in any $2$-category admitting finite powers and comma objects --- for instance, in a $2$-category with pie limits.  Let us outline how to handle \emph{Barr-exact} categories in a similar way.  To this end, let $ERel(X)$ denote the category of equivalence relations in $X$; then $X$ has coequalisers of kernel pairs precisely when the diagonal $\Delta:X \to ERel(X)$ has a left adjoint; in this case the kernel pair functor $kp:X^{\atwo} \to ERel(X)$ admits a left adjoint and the final condition for a Barr-exact category --- namely, that equivalence relations are effective --- amounts to asking that the unit of this adjunction be invertible.  

In showing that the $2$-category of Barr-exact categories belongs to $\Lpm$ the only real additional subtlety concerns the ``arity" $ERel(X)$.  That is, we should show that the assignment $ERel(-):\Reg \to \Cat$ belongs to $\Lpm$.  To achieve this we will show that $ERel(X)$ can be constructed from $X$ using finite flexible limits.  This is done in three stages.  Firstly, consider the diagram below left in which $\Box$ denotes the generic category containing a commutative square.
\begin{equation*}
\begin{xy}
(0,0)*+{X^{\Box}}="00";(15,-15)*+{X^{C}}="01";(30,0)*+{X^{\Box}}="20";
{\ar^{1} "00"; "20"};{\ar_{u} "00"; "01"};{\ar_{pb} "01"; "20"};
{\ar@{=>}^{\eta}(15,-4)*+{};(15,-11)*+{}};
\end{xy}
\hspace{1.5cm}
\begin{xy}
(-15,0)*+{X^{\atwo}}="-10";
(0,0)*+{X^{\Box}}="00";(15,-15)*+{X^{C}}="01";(30,0)*+{X^{\Box}}="20";
{\ar^{i} "-10"; "00"};{\ar^{1} "00"; "20"};{\ar_{u} "00"; "01"};{\ar_{pb} "01"; "20"};
{\ar@{=>}^{\eta}(15,-4)*+{};(15,-11)*+{}};
\end{xy}
\end{equation*}
Here $u$ takes the underlying cospan of a commutative square and $pb$ takes its pullback; then a pullback square is an object of $X^{\Box}$ at which the component of the canonical natural transformation $\eta:1 \Rightarrow pb \circ u$ is invertible; thus the subobject $Pb(X) \hookrightarrow X^{\Box}$ of pullback squares in $X$ is obtained as the \emph{inverter} of $\eta$.  Let $i:X^{\atwo} \to X^{\Box}$ send $m:a \to b$ to the square
\begin{equation*}
\xymatrix{
a \ar[d]_{1} \ar[r]^{1} & a \ar[d]^{m} \\
a \ar[r]_{m} & b}
\end{equation*}
so that $m$ is mono just when this square is a pullback.  It follows that the full subcategory $Mono(X) \hookrightarrow X^{\atwo}$ of monomorphisms is the inverter of the $2$-cell in the diagram above right.  Let $\Delta_2$ denote the full subcategory of the simplicial category on the objects $0,1$ and $2$. Now an equivalence relation can be described as an object $A$ of $X^{\Delta_{2}^{op}}$ as below left
\begin{equation*}
\xy
(0,0)*+{A_{2}}="c0"; (20,0)*+{A_{1}}="b0";(40,0)*+{A_{0}}="a0";
{\ar@<1.5ex>^{d} "b0"; "a0"}; 
{\ar@<0ex>|{i} "a0"; "b0"}; 
{\ar@<-1.5ex>_{c} "b0"; "a0"}; 
{\ar@<3ex>^{p} "c0"; "b0"}; 
{\ar@<0ex>|{m} "c0"; "b0"}; 
{\ar@<-3ex>_{q} "c0"; "b0"};
{\ar@<1.5ex>|{r} "b0"; "c0"};
{\ar@<-1.5ex>|{l} "b0"; "c0"};
\endxy
\hspace{1cm}
\xy
(0,5)*+{A_{2}}="a0";(15,5)*+{A_{1}}="b0"; (0,-5)*+{A_{1}}="c0"; (15,-5)*+{A_{0}}="d0";
{\ar^{p} "a0";"b0"};
{\ar_{q} "a0";"c0"};
{\ar^{c} "b0";"d0"};
{\ar_{d} "c0";"d0"};
\endxy
\hspace{1cm}
\xy
(0,0)*+{A_{1}}="a0";(15,0)*+{(A_{0})^{2}}="b0"; 
{\ar^-{\langle d,c \rangle} "a0";"b0"};
\endxy
\end{equation*}
for which the central rectangle above is a pullback, and for which the induced map above right is monic.  The subobject of equivalence relations $ERel(X) \hookrightarrow X^{\Delta_{2}^{op}}$ is accordingly a pullback of the form:
\begin{equation*}
\xymatrix{
ERel(X) \ar[d] \ar[r]^{} & X^{\Delta_{2}^{op}} \ar[d]^{} \\
Pb(X) \times Mono(X) \ar[r] & X^{\Box} \times X^{\atwo}
}
\end{equation*}
But is this pullback a finite flexible limit?  Since the lower horizontal leg is an isofibration it is certainly a bipullback; in fact since it is a discrete isofibration the pullback induces an idempotent on the pseudopullback whose splitting is the pullback itself  --- see Proposition~\ref{prop:normal} in the appendix  --- and so indeed finite flexible.  
Alternatively, if one wishes to make do with pie limits, it suffices to simply \emph{define} $ERel(X)$ as the pseudopullback --- since an equivalent object is obtained in this way, it does not alter the resulting notion of Barr-exact object/category.


\section{2-monads, their algebras and pseudoalgebras}\label{sect:2monad}
In the present section, we show that under natural conditions the $2$-categories of lax, pseudo and colax algebras for a $2$-monad $T$, together with their pseudomorphisms, belong to $\Lpm$.  We also obtain a positive result for the strict algebras of a \emph{flexible} $2$-monad --- this covers the examples of monoidal categories, bicategories and categories with finite limits but not those of structures such as regular categories.  Furthermore, we show that in certain circumstances accessibility is equivalent to a weak form of flexibility --- namely, that each pseudoalgebra is isomorphic to a strict one.

Let $T$ be a $2$-monad on a $2$-category $\C$.  A lax $T$-algebra $(X,x,x_0,\overline{x})$ is specified by a morphism $x:TX \to X$ together with a pair of structural $2$-cells as below

\begin{equation}\label{eq:laxalg}
\begin{xy}
(0,0)*+{T^{2}X}="00";(20,0)*+{TX}="10";(0,-15)*+{X}="01";(20,-15)*+{TX}="11";
{\ar^{Tx} "00"; "10"};{\ar_{\mu_X} "00"; "01"};{\ar^{x} "10"; "11"};{\ar_{x} "01"; "11"};
{\ar@{=>}^{\overline{x}}(13,-5)*+{};(7,-10)*+{}};
\end{xy}
\hspace{1.5cm}
\begin{xy}
(0,0)*+{X}="00";(0,-15)*+{TX}="10";(20,-15)*+{X}="11";
{\ar_{\eta_{X}} "00"; "10"};{\ar_{x} "10"; "11"};{\ar^{1} "00"; "11"};
{\ar@{=>}^{x_0}(9,-7)*+{};(4,-12)*+{}};
\end{xy}
\end{equation}
satisfying three equations described in \cite{Lack2002Codescent}.  A \emph{pseudoalgebra} is a lax algebra for which the above structural $2$-cells are invertible.  
A pseudomorphism of lax algebras is specified by an invertible $2$-cell
\begin{equation*}\label{eq:laxmap}
\begin{xy}
(0,0)*+{TX}="00";(20,0)*+{TY}="10";(0,-15)*+{X}="01";(20,-15)*+{Y}="11";
{\ar^{Tf} "00"; "10"};{\ar_{x} "00"; "01"};{\ar^{y} "10"; "11"};{\ar_{f} "01"; "11"};
(10,-8)*+{\cong^{\overline{f}}};
\end{xy}
\end{equation*}
satisfying two equations also described in \cite{Lack2002Codescent}.  Together with the appropriate notion of $2$-cell these form the $2$-category $\LaxTAlg$ of lax $T$-algebras in $\C$, whose full sub $2$-category of pseudoalgebras is denoted by $\PsTAlg$.  There is also the notion of a colax $T$-algebra, in which the orientations of the structural $2$-cells \eqref{eq:laxalg} are reversed.  With the attendant notions of pseudomorphism and $2$-cells, these form the $2$-category $\ColaxTAlg$.

\begin{Prop}\label{prop:2monad}
Let $T$ be a filtered colimit preserving $2$-monad on $\C \in \Lp_{\mathcal M}$.  Then the 2-categories $\LaxTAlg,\PsTAlg$ and $\ColaxTAlg$ belongs to $\Lp_{\mathcal M}$ as does each of the corresponding forgetful $2$-functors to $\C$.
\end{Prop}
\begin{proof}
Let us focus on $\LaxTAlg$ --- the other two cases are minor variants.  
Consider the pullback below left.
\begin{equation*}
\xymatrix{\TAlg_{1} \ar[d] \ar[rr] && Ps(D_1,\C) \ar[d]^{Ps(J_{1},\C)} \\
\C \ar[rr]_{R_{1}:C \mapsto (TC,C)} && Ps(P_0,\C)}
\hspace{1.5cm}
\begin{xy}
(0,0)*+{TX}="00";(20,0)*+{TY}="10";(0,-15)*+{X}="01";(20,-15)*+{Y}="11";
{\ar^{Tf} "00"; "10"};{\ar_{x} "00"; "01"};{\ar^{y} "10"; "11"};{\ar_{f} "01"; "11"};
(10,-8)*+{\cong^{\overline{f}}};
\end{xy}
\end{equation*}
An object of the pullback $\TAlg_1$ consists of an action $(X,x:TX \to X)$ whilst a morphism $(f,\overline{f}):(X,x) \to (Y,y)$ consists of $f:X \to Y$ and an invertible $2$-cell as above right.
The $2$-functor $R_1$ does not belong to $\Lpm$ but does belong to $ \Acc_{\omega,\mathcal{M}}$ since $T$ preserves filtered colimits.  The right leg is an isofibration in $\Lp_{\mathcal M}$ and so in $ \Acc_{\omega,\mathcal{M}}$.  Hence, by Theorem~\ref{thm:limits2}, the whole pullback diagram belongs to $\Acc_{\omega,\mathcal{M}}$.

Now the remaining data of a lax algebra involves structural $2$-cells $\overline{x}$ and $x_0$ as on the right below.  In order to add these, we form a pullback
\begin{equation*}
\xymatrix{\TAlg_{2} \ar[d] \ar[r] & Ps(D_2,\C)^2 \ar[d]^{Ps(J_{2},\C)^2} \\
\TAlg_1 \ar[r]_{R_2} & Ps(P_1,\C)^2}
\hspace{0.7cm}
\begin{xy}
(0,0)*+{T^{2}X}="00";(20,0)*+{TX}="10";(0,-15)*+{X}="01";(20,-15)*+{TX}="11";
{\ar^{Tx} "00"; "10"};{\ar_{\mu_X} "00"; "01"};{\ar^{x} "10"; "11"};{\ar_{x} "01"; "11"};
{\ar@{=>}^{\overline{x}}(13,-5)*+{};(7,-10)*+{}};
\end{xy}
\hspace{1cm}
\begin{xy}
(0,0)*+{X}="00";(0,-15)*+{TX}="10";(20,-15)*+{X}="11";
{\ar_{\eta_{X}} "00"; "10"};{\ar_{x} "10"; "11"};{\ar^{1} "00"; "11"};
{\ar@{=>}^{x_0}(9,-7)*+{};(4,-12)*+{}};
\end{xy}
\end{equation*}
in which $R_2$ sends $(X,x)$ to the boundaries of the two $2$-cells depicted above right --- hence an object of the pullback $\TAlg_{2}$ is an object $(X,x)$ equipped with a pair of such $2$-cells, and a morphism $(f,\overline{f})$ now commutes with the additional structure in the sense of a pseudomorphism of lax $T$-algebras.  Since both $T$ and $T^2$ preserve filtered colimits, so does $R_2$ --- it follows as before that the whole diagram belongs to $ \Acc_{\omega,\mathcal{M}}$.


Finally, we form a further pullback

\begin{equation*}
\xymatrix{\TAlg_{3} \ar[d] \ar[rr] && Ps(D_2,\C)^3 \ar[d]^{Ps(J_{3},\C)^3} \\
\TAlg_2 \ar[rr]_{R_3} && Ps(P_2,\C)^3}
\end{equation*}

encoding the three equations for a lax $T$-algebra.  This time the $2$-functor $R_3$ again involves only powers of $T$ (up to $T^3$) and so preserves filtered colimits.  Arguing as before, we see that the entire diagram belongs to $ \Acc_{\omega,\mathcal{M}}$.  In particular $\LaxTAlg = \TAlg_{3}$ does so, as does the composite forgetful $2$-functor $\LaxTAlg \to \C$ obtained by composing the left legs of the various pullbacks above.  

Finally, we must show the requisite limits lift along $U:\LaxTAlg \to \C$ --- since $\LaxTAlg$ has split idempotents, it suffices to check that products, inserters and equifiers lift along $U$.  These three cases can be verified directly, by adapting the corresponding argument for strict algebras and pseudomorphisms given in Section 2 of \cite{Blackwell1989Two-dimensional}.  Since $U$ preserves filtered colimits and flexible limits and is conservative,  the requisite limit-colimit commutativity properties then follow from those of $\C$.  Thus $\C \in \Lpm$.
\end{proof}

\begin{Theorem}\label{thm:pseudoalgebra}
Let $T$ be a filtered colimit preserving $2$-monad on $\C \in \Lp_{\mathcal M}$ such that each pseudo-$T$-algebra is isomorphic in $\PsTAlg$ to a strict $T$-algebra.  Then $U:\TAlg \to \C \in \Lp_{\mathcal M}$.
\end{Theorem}
\begin{proof}
Since the inclusion $\TAlg \to \PsTAlg$ is a $2$-equivalence, this follows from the preceding result.
\end{proof}

An important case in which each pseudo-$T$-algebra is isomorphic to a strict one is when $T$ is a \emph{flexible} $2$-monad \cite{Kelly1972Doctrinal}.  Let us summarise a few key facts about flexible $2$-monads from \cite{Lack2007Homotopy-theoretic}.  For consistency with \emph{loc.cit.} we restrict ourselves to the case that $T$ is a finitary $2$-monad on a locally finitely presentable $2$-category $\C$.  In that case the strict monoidal $2$-category $\mathbf{End}_f(\C)$ of finitary endo-2-functors of $\C$ has the same property.  As with any complete and cocomplete $2$-category, $\mathbf{End}_f(\C)$ admits a natural model structure in which the weak equivalences and fibrations are the equivalences and isofibrations respectly.  Let $\mathbf{Mnd}_f(\C)$ denote the $2$-category of finitary $2$-monads on $\C$.  This is again locally finitely presentable whilst the forgetful $2$-functor $U:\mathbf{Mnd}_f(\C) \to \mathbf{End}_f(\C)$ is a finitary right adjoint.  The right adjoint induces a model structure on $\mathbf{Mnd}_f(\C)$ in which $f$ is a weak equivalence/fibration just when $Uf$ is one in $\mathbf{End}_f(\C)$.  

Now a $2$-monad is \emph{flexible} just when it is a cofibrant object with respect to the lifted model structure.  The flexible $2$-monads include all flexible colimits of free $2$-monads on finitary endo-2-functors of $\C$ --- it follows easily that they capture the $2$-monads for monoidal categories, and other finitary structures whose definition involves \emph{no equations between objects}.  This is explored in detail in Section 6 of \cite{Bourke2013On}.  Since for flexible $T$ each pseudo-$T$-algebra is isomorphic to a strict one we have:

\begin{Cor}
Let $T$ be a finitary flexible 2-monad on a locally finitely presentable 2-category $\C$ belonging to $\Lp_{\mathcal M}$.  Then $U:\TAlg \to \C \in \Lp_{\mathcal M}$.
\end{Cor}

A precursor to the above result is given in Remark 7.2 of \cite{Bird1989Flexible}, where it is observed that if $T$ is flexible then, under suitable conditions on the base, the $2$-category $\TAlg$ has flexible limits, and in particular split idempotents.  

The above result subsumes the examples such as monoidal categories, bicategories and categories with finite limits of Sections~\ref{sect:monadic1},~\ref{sect:monadic2} and ~\ref{sect:monadic3} but not those structures in Section~\ref{sect:monadic4} involving exactness properties --- for the appropriate framework for this, we refer the reader to Section~\ref{sect:limittheories} below.

In fact for many $2$-monads Theorem~\ref{thm:pseudoalgebra} has a converse --- for such 2-monads $T$ the \emph{accessibility} of $\TAlg$ is equivalent to the property that each \emph{pseudoalgebra is isomorphic to a strict one}.  For brevity, we avoid maximum generality in the following result, which abstracts the situation described for strict monoidal categories in Section~\ref{sect:strictmonoidal}.

\begin{Theorem}\label{thm:limitations}
Consider a $2$-monad $T$ on $\Cat^{X}$ for $X$ a set.  Suppose that $T$ preserves pointwise bijections on objects and that each unit component $\eta_{A}:A \to TA$ is pointwise injective on objects.  
\begin{enumerate}
\item Then the inclusion $\TAlg \to \PsTAlg$ exhibits $\PsTAlg$ as the idempotent completion of $\TAlg$.
\item Furthermore, if $T$ is finitary the following are equivalent.
\begin{enumerate}
\item Each pseudoalgebra is isomorphic to a strict one.
\item Idempotents split in $\TAlg$.
\item The category $\TAlg_0$ is accessible.
\item $\TAlg \in \Lpm$.
\end{enumerate}
\end{enumerate}  
\end{Theorem}
\begin{proof}
The pointwise bijections on objects and pointwise fully faithful morphisms on $\Cat^X$ form an \emph{enhanced} factorisation system on $\Cat^{X}$ in the sense of \cite{Lack2002Codescent}.  Since $T$ preserves the left class by assumption, Theorem 4.10 of \emph{loc.cit.} ensures that the inclusion $j:\TAlg_s \to \PsTAlg$ has a left adjoint with each unit component $q_A:\mathbf A \to \mathbf{A}^{\prime}$ an equivalence.  Let us write $\mathbf A= (A,a,a_0,\overline{a})$ for the pseudoalgebra and $k:\mathbf{A}^{\prime} \to \mathbf A \to \mathbf A$ for the equivalence inverse.  Examining the proof in \emph{loc.cit} we see that the isomorphism $k \circ q_A \cong 1_{\mathbf A}$ is given by the structural isomorphism $a_0:a \circ \eta_{A} \cong 1_{A}$.  Therefore if $a_0 = 1$ --- that is, if $\mathbf A$ is a \emph{normal pseudoalgebra} --- then $\mathbf A$ is, in fact, a \emph{retract} of the strict algebra $\mathbf{A}^{\prime}$.  Therefore, if we can show that each pseudoalgebra is isomorphic to a normal one, it will follow that each pseudoalgebra is a retract of a strict one --- in particular proving (1) above.  

For this, we use the assumption that each unit component $\eta_{A}:A \to TA$ is pointwise injective on objects.  The value of this is that, by Theorem 2 of \cite{Joyal1993Pullbacks}, the pointwise injections on objects are precisely the isocofibrations in $\Cat^{X}$ --- thus, there exists $a^{\prime}:TA \to A$ with $\theta:a^{\prime} \cong a$ such that $a \circ \eta_{A} = 1$ and $\theta \circ \eta_{A} = a_0$.  Transporting the pseudoalgebra structure along the isomorphism $\theta$ yields a unique \emph{normal} pseudoalgebra $(A,a^{\prime},1,\overline{a}^{\prime})$ such that $(1_A,\theta):(A,a^{\prime},1,\overline{a}^{\prime}) \to (A,a,a_0,\overline{a})$ is an isomorphism of pseudoalgebras.  In particular, each pseudoalgebra is isomorphic to a normal pseudoalgebra, as required.

For Part (2) suppose (a) holds.  By Theorem~\ref{thm:pseudoalgebra} this implies (d) whilst trivially $(d) \implies (c) \implies (b)$.  Supposing $(b)$ then we have that each pseudoalgebra $\mathbf A$ is a retract of a strict algebra $\mathbf B$; thus splitting the corresponding idempotent on $\mathbf B$ produces an object $\mathbf C \in \TAlg$.  Since the inclusion $\TAlg \to \PsTAlg$ preserves idempotent splittings, therefore $\mathbf B \cong \mathbf A$, proving (a).
\end{proof}

\begin{Example}
Let us describe another example to which Theorem~\ref{thm:limitations} applies.  Given a small $2$-category $\A$ let $ob\A$ denote its underlying set of objects.  Restriction and left Kan extension along the inclusion $ob\A \to \A$ induces a $2$-monad $T$ on $\Cat^{ob\A}$ with $\TAlg \cong Ps(\A,\Cat)$ and $\PsTAlg \cong Hom(\A,\Cat)$ the $2$-category of pseudofunctors, pseudonatural transformations and modifications -- see 6.6 of \cite{Blackwell1989Two-dimensional} and Remark 7.2 of \cite{Bird1989Flexible}.  This $2$-monad has the simple formula $(TX)(a) = \Sigma_{b \in \A}\A(b,a) \times X(b)$ with unit $\eta_a:X(a) \to TX(a)$ sending $x$ to the pair $(id_a,x)$.  It follows from this description that $T$ satisfies the two hypotheses of Theorem~\ref{thm:limitations}.  Therefore $Hom(\A,\Cat) \in \Lpm$ is the idempotent completion of $Ps(\A,\Cat)$, with the latter $2$-category having split idempotents just when each pseudofunctor from $\A$ to $\Cat$ is isomorphic to a $2$-functor.  Example 6.2 of \cite{Bird1989Flexible} gives an $\A$ for which this condition is not met --- namely, when $\A$ is the free $2$-category containing a retraction, or equivalently, the free split idempotent.  Of course, the free split idempotent is not cellular.
\end{Example}

\section{The accessibility of isofibrations and fibrations in a 2-category}\label{sect:isofibrations}

The three classes of morphisms in a $2$-category $\C$ --- isofibrations, equivalences and surjective equivalences --- are closely connected.  Indeed in a sufficiently complete and cocomplete $2$-category, they form the fibrations, weak equivalences and trivial fibrations of a model structure \cite{Lack2007Homotopy-theoretic}.  From this perspective, it may seem somewhat obscure that Property $\mathcal M$ concerns only the accessibility of the surjective equivalences and says nothing of the other two classes.  The following result addresses this concern.

\begin{Prop}\label{prop:isofibrations}
Let $\C \in \Lp$ and let $IsoFib(\C)$ and $Equiv(\C)$ denote the full sub-$2$-categories of $\C^{\atwo}$ containing the isofibrations and equivalences respectively.  Then the following are equivalent.
\begin{enumerate}
\item $RE(\C) \hookrightarrow \C$ belongs to $\Lp$ --- that is $\C$ satisfies Property $\mathcal M$.
\item Both $IsoFib(\C) \hookrightarrow \C^{\atwo}$ and $Equiv(\C) \hookrightarrow \C^{\atwo}$ belong to $\Lp$.
\end{enumerate}
\end{Prop}
\begin{proof}
Since a morphism is a surjective equivalence just when it is a fibration and an equivalence, we have the pullback $RE(\C) = IsoFib(\C) \cap Equiv(\C)$ of the two morphisms in (2).  If both belong to $\Lp$ then, as the pullback of a pair of isofibrations in $\Lp$, the pullback square belongs to $\Lp$ by Proposition~\ref{prop:limits}.  Thus $(2 \implies 1)$.  

For the converse, we use that $f:A \to B$ is an equivalence if and only if $q_f:Pf \to B$ is a surjective equivalence.  Indeed, we have $f=q_f \circ r_f$ with $r_f$ an equivalence and $q_f$ an isofibration by Proposition~\ref{prop:properties}.  Thus by the two from three property for equivalences, $f$ is an equivalence just when $q_f$ is both an equivalence and an isofibration --- that is, a surjective equivalence --- as required.  Thus we have a pullback square as on the left below.
\begin{equation*}
\xymatrix{
Equiv(\C) \ar[d] \ar[rrr] &&& RE(\C) \ar[d]^{} \\
\C^{\atwo} \ar[rrr]_{(f:A \to B) \mapsto (q_f:Pf \to B)} &&& \C^{\atwo}}
\hspace{1cm}
\xymatrix{
Isofib(\C) \ar[d] \ar[rrr] &&& RE(\C) \ar[d]^{} \\
\C^{\atwo} \ar[rrr]_{(f:A \to B) \mapsto (t_f:A^{I} \to Pf)} &&& \C^{\atwo}}
\end{equation*}
Secondly, letting $I$ denote the free isomorphism it is simple to see $f$ is an isofibration if and only if the induced map $t_f:A^{I} \to Pf$ is a surjective equivalence.  We thus have a pullback square as on the right above.  Since finite powers and pseudolimits of arrows commute with flexible limits and filtered colimits in $\C$, the bottom leg of each pullback square belongs to $\Lp$.  Assuming (1) the right vertical morphism does too.  Since it is also an isofibration, both pullback squares belong to $\Lp$ by Proposition~\ref{prop:limits}, thus proving (2).
\end{proof}

We now use the above to investigate fibrations in a $2$-category $\C$ --- see \cite{Street1972Fibrations, Weber, Loregian2019Categorical}.  We presume that the reader is familiar with fibrations in $\Cat$ --- briefly, these are functors $f:A \to B$ in which each morphism $a \to fb \in \B$ admits a lift to a \emph{$j$-cartesian morphism} in $A$.   A morphism of fibrations is a commutative square $(r,s):j \to k$ such that $r$ takes $j$-cartesian morphisms to $k$-cartesian morphisms.

In a general $2$-category $\C$ a morphism $j:A \to B$ is said to be a \emph{fibration} if it satisfies two conditions.  Firstly,
\begin{itemize}
\item[(1)] For each $X \in \C$ the functor $\C(X,j):\C(X,A) \to \C(X,B)$ is a fibration in $\Cat$.
\end{itemize}
In this context a $2$-cell $\lambda:u \Rightarrow v \in \C(X,A)$ is said to be $j$-cartesian if it is a cartesian morphism for the fibration $\C(X,j):\C(X,A) \to \C(X,B)$.  This is used to formulate the second condition.
\begin{itemize}
\item[(2)] Cartesian $2$-cells $\lambda:u \Rightarrow v \in \C(X,A)$ are preserved by whiskering by morphisms $Y \to X$ --- that is, for all $Y \to X \in \C$ the induced square $\C(X,j) \to \C(Y,j)$ is a morphism of fibrations.  
\end{itemize}
Now $Fib(\C) \hookrightarrow \C^{\atwo}$ is defined to be the locally full sub-$2$-category whose objects are fibrations and whose morphisms are the commutative squares taking cartesian $2$-cells to cartesian $2$-cells.   
\begin{Prop}
Let $\C \in \Lpm$.  Then $Fib(\C) \hookrightarrow \C^{\atwo} \in \Lp$.
\end{Prop}
\begin{proof}
Let $\C \in \Lpm$.  Given $j:A \to B \in \C$ we can form the oplax limit $B/j$ of $j$.  This is the universal diagram of shape
$$\xy
(0,0)*+{B/j}="a0"; (-15,-15)*+{A}="b0";(15,-15)*+{B}="c0";
{\ar_{p_{j}} "a0"; "b0"}; 
{\ar^{q_{j}} "a0"; "c0"}; 
{\ar_{j} "b0"; "c0"}; 
{\ar@{=>}^{\lambda_j}(2,-6)*+{};(-2,-11)*+{}};
\endxy$$
and equally the comma object $1_B/j$.  From Theorem 3.1.3 of \cite{Loregian2019Categorical}, the map $j$ is a fibration just when the canonical map $A^{\atwo} \to B/j$ has a right adjoint with identity counit.  (This is a slight variant of a result of Street \cite{Street1972Fibrations}.)  

Now each fibration is an isofibration, and if $j$ is an isofibration then the map $t_j:A^{\atwo} \to B/j$ induced by the identity $2$-cell $1_j:j \Rightarrow j.1_A$ is easily seen to be an isofibration too.  Now if an isofibration has a right adjoint with \emph{invertible} counit, then using its isomorphism lifting property we can find a right adjoint with \emph{identity counit}.  Combining the above, we obtain that
\begin{itemize}
\item the morphism $j$ is a fibration just when it is an isofibration and the canonical map $t_j:A^{\atwo} \to B/j$ has a right adjoint with invertible counit.
\end{itemize}
Therefore we have a pullback square as below
\begin{equation*}
\xymatrix{
Fib(\C) \ar[d] \ar[rrrrrr] &&&&&& Ps(\Ref,\C) \ar[d]^{Ps(f,\C)} \\
Isofib(\C) \ar@{{(}->}[r] & \C^{\atwo} \ar[rrrr]^{(j:A \to B) \mapsto (t_j:A^{\atwo} \to B/j)} &&&& \C^{\atwo} \ar@{{(}->}[r] & Ps(\atwo,\C)}
\end{equation*}
in which the right vertical leg selects the left adjoint of a reflection.  This leg is an isofibration belonging to $\Lpm$ by Theorem~\ref{thm:cellular}.  Turning to the lower horizontal leg we have that $Isofib(\C) \hookrightarrow \C^{\atwo}$ belongs to $\Lp$ by Proposition~\ref{prop:isofibrations}.  The second component of the bottom leg belongs to $\Lp$ since powers by $\atwo$ and comma objects are finite flexible limits and thus commute with flexible limits and filtered colimits in $\C$.  Finally the inclusion $\C^{\atwo} \hookrightarrow Ps(\atwo,\C)$ preserves flexible limits and filtered colimits since they are pointwise in both of these $2$-categories.  Therefore by Proposition~\ref{prop:limits} the pullback belongs to $\Lp$.

%
%
\end{proof}

\section{Finite limit 2-theories}\label{sect:limittheories}

In this final section, we briefly consider two-dimensional limit theories, which are a notion of theory general enough to capture all of the structures considered in the paper thus far.  Whilst the models and strict morphisms of limit-theories can be defined over any base of enrichment \cite{Kelly1982Structures}, the notions of pseudo and lax morphism are somewhat more subtle.  Building on work of Lack and Power \cite{Lack2007Lawvere} on two-dimensional Lawvere theories, we give definitions of pseudo and lax morphisms in the more general context of a pie limit theory.  Using this, we state a result on the accessibility of the $2$-category of models and pseudomaps, to be proven in a followup paper on two-dimensional limit theories.

\subsection{Models and their homomorphisms}
Let $\Phi$ be a small class of weights.  By a $\Phi$-limit theory we simply mean a small $2$-category with chosen $\Phi$-limits.  Let $\mathbb T$ be a $\Phi$-limit theory and $\C$ a $2$-category admitting $\Phi$-limits.  A model of $\mathbb T$ in $\C$ is a $2$-functor $\mathbb T \to \C$ preserving $\Phi$-limits, whilst a \emph{strict morphism} of models is a $2$-natural transformation.  

\begin{Example}
The case of the weaker kinds of morphisms is more subtle and we will not define them for fully general $\Phi$, but begin here by recalling how it goes for Lawvere $2$-theories and the subtlety involved.  A \emph{Lawvere $2$-theory} is an identity on objects finite product preserving $2$-functor $J:\mathbb F^{op} \to \mathbb T$ where $\mathbb F$ is a skeleton of finite sets --- we take its objects to be the natural numbers.

To see the issue concerning weak morphisms, let us consider the Lawvere theory $J:\mathbb F^{op} \to \mathbb T$ for monoids.  In addition to the two product projections $\Pi_1,\Pi_2:2 \rightrightarrows 1 $ coming from $\mathbb F^{op}$,  $\mathbb T$ also contains a morphism $m:2 \to 1$.  Then a finite product preserving functor $X:\mathbb T \to \Cat$ gives a strict monoidal category $X(1)$ with multiplication $X(m):X(2) = X(1)^{2} \to X(1)$.  Now a lax morphism $f:X \to Y$ of models should amount to a \emph{lax monoidal functor}, which has one of its structure constraints given by a natural transformation as on the left below.
\begin{equation*}
\begin{xy}
(0,0)*+{X(1)^{2}}="00";(20,0)*+{Y(1)^{2}}="10";(0,-15)*+{X(1)}="01";(20,-15)*+{Y(1)}="11";
{\ar^{f(1)^{2}} "00"; "10"};{\ar_{X(m)} "00"; "01"};{\ar^{Y(m)} "10"; "11"};{\ar_{f(1)} "01"; "11"};
{\ar@{=>}^{}(13,-5)*+{};(7,-10)*+{}};
\end{xy}
\hspace{0.5cm}
\begin{xy}
(00,0)*+{X(1)^{2}}="00";(20,0)*+{Y(1)^{2}}="10";(00,-15)*+{X(1)}="01";(20,-15)*+{Y(1)}="11";
{\ar^{f(2)} "00"; "10"};{\ar_{X(m)} "00"; "01"};{\ar^{Y(m)} "10"; "11"};{\ar_{f} "01"; "11"};
{\ar@{=>}^{\overline{f}_{m}}(13,-5)*+{};(7,-10)*+{}};
\end{xy}
\hspace{0.5cm}
\begin{xy}
(00,0)*+{X(1)^{2}}="00";(25,0)*+{Y(1)^{2}}="10";(00,-15)*+{X(1)}="01";(25,-15)*+{Y(1)}="11";
{\ar^{f(2)=f(1)^{2}} "00"; "10"};{\ar@<-1ex>_{X(\Pi_{1})} "00"; "01"};{\ar@<1ex>^{X(\Pi_2)} "00"; "01"};{\ar@<-1ex>_{Y(\Pi_1)} "10"; "11"};{\ar@<1ex>^{Y(\Pi_2)} "10"; "11"};{\ar_{f} "01"; "11"};
\end{xy}
\end{equation*}
This suggests defining a lax morphism to be a \emph{lax natural transformation} $f:X \to Y$.  However, this is \emph{not} quite right, since a lax transformation merely gives a $2$-cell as in the middle above but does not force $f(2) = f(1)^{2}$.  For that equation to hold, one requires the two naturality squares above right to commute on the nose.  In fact, the correct notion of \emph{lax morphism} $f:X \to Y$ is a lax natural transformation which restricts along $J:\mathbb F^{op} \to \mathbb T$ to a $2$-natural transformation \cite{Lack2007Lawvere}; equivalently, a lax natural transformation whose $2$-cell components are identities at product projections.  
\end{Example}

Now let us suppose that $\Phi$ is a small class of \emph{pie weights}.  We will define pseudo and lax morphisms for $\Phi$-limit theories.  The key to our approach is the characterisation of pie weights of Power and Robinson \cite{Power1991A-characterization}, which we now recall.  Let $W:\J \to \Cat$ be a weight.  This induces a functor
\begin{equation*}
\xymatrix{
\J_0 \ar[r]^{W_0} & \Cat_0 \ar[r]^{ob} & \Set .
}
\end{equation*}
whose category of elements $el(ob \circ W_0)$ has objects given by pairs $(j \in \J, x \in Wj)$ and morphisms $f:(j,x) \to (k,y)$ given by morphisms $f:j \to k \in \J$ such that $Wf(x) = y$.
\begin{Theorem}[Power and Robinson]
$W$ is a pie weight if and only if each connected component of $el(ob \circ W_0)$ has an initial object.
\end{Theorem}
Henceforth, by a pie weight we will mean a pair $(W,\Omega_W)$ where $W$ is a pie weight and $\Omega \subseteq el(ob \circ W_0)$ is a witnessing set of such initial objects --- exactly one from each connected component.  Our definition of pseudo and lax morphisms below depends on the choice of such a set.

Suppose that $\Phi$ is a small class of pie weights and consider a $\Phi$-limit theory $\mathbb T \in \ff_s$.  Then given a weight $W \in \Phi(J)$ and diagram $D:J \to \mathbb T$ we have the specified limit $L=\{W,D\}$ and limit projections $p_{(x,j)}:L \to Dj$.   We call these \emph{initial limit projections} if $(x,j) \in \Omega_W$.

\begin{Def}
Let $\mathbb T$ be a $\Phi$-limit theory and $\C$ be a $2$-category with $\Phi$-limits.  Given models $X,Y:\mathbb T \rightrightarrows \C$ a \emph{pseudo/lax} morphism of models $\theta:X \to Y$ is a pseudonatural/lax natural transformation with the property that 
\begin{itemize}
\item if $f:L \to Dj$ is an \emph{initial limit projection} then the naturality $2$-cell $\theta_{f}:Yf.\theta_{X} \Rightarrow \theta_{Y}.Xf$ is the \emph{identity}.
\end{itemize}
\end{Def}
In each case, the $2$-cells between morphisms of models are modifications.  We write $\mathbf{Mod}_{s}(\mathbb T,\C), \mathbf{Mod}_{p}(\mathbb T,\C)$ and $\mathbf{Mod}_{l}(\mathbb T,\C)$ for the $2$-categories of models together with the strict, pseudo and lax morphisms of models respectively.

Henceforth, we will suppose further that $\Phi$ is a class of \emph{finite} pie weights.  We write $\ff_s$ for the $2$-category containing the $\Phi$-limit theories, whose morphisms are $2$-functors preserving $\Phi$-limits strictly, and whose morphisms are $2$-natural transformations.  

By Theorem 6.1 of \cite{Kelly2000On}, the forgetful $2$-functor $U:\ff_{s} \to \twocat$ has a left adjoint $F$. and is monadic.  Since each weight $W \in \Phi$ is finitely presentable, $U$ also preserves filtered colimits --- therefore, by Theorem 3.8 of \cite{Blackwell1989Two-dimensional} the $2$-category $\ff_s$ is cocomplete.  
Let $J$ be the set of generating cofibrations in $\twocat$, whereby we can consider the resulting set $FJ \in \ff_{s}$.  By a \emph{cellular theory}, by we mean an $FJ$-cellular object in $\ff_{s}$.  

Our main theorem in this context, to be proved in a followup paper on two-dimensional limit theories, is the following one.  Ultimately, after taking care of various subtleties, this will follow naturally from Theorem~\ref{thm:cellular}.
\begin{Theorem}\label{thm:limittheories}
Let $\Phi$ be a small class of finite pie weights with the additional property that the source of each weight $W \in \Phi$ is a cellular $2$-category.  Let $\mathbb T$ be a cellular $\Phi$-limit theory and $\C \in \Lpm$.  Then $\mathbf{Mod}_{p}(\mathbb T,\C) \in \Lpm$ with flexible limits and filtered colimits pointwise as in $\C$.
\end{Theorem}

The definitions of the weaker kinds of morphism in this context are somewhat subtle --- the situation becomes more transparent if we view a pie weight $(W,\Omega_{W})$ as an $\F$-weight in the sense of \cite{Lack2012Enhanced} and pass from $\Cat$-enriched limit theories to $\F$-enriched ones.  

Recall that an $\F$-category is specified by a $2$-category equipped with a subcollection of \emph{tight} morphisms closed under composition and containing the identities.  General morphisms are called loose.  The point is that $\F$-categories allow one to keep track of which morphisms are to be treated strictly.  Even for simple structures such as double categories --- that is, internal categories in $\Cat$ --- it appears that $\F$-theories are required to capture the correct pseudo and lax maps.  Indeed, the correct notion of lax morphism in this context is that of a lax double functor --- this involves a graph map as below left
\begin{equation*}
\begin{xy}
(0,0)*+{X_1}="00";(25,0)*+{Y_1}="10";(00,-15)*+{X_0}="01";(25,-15)*+{Y_0}="11";
{\ar^{f_1} "00"; "10"};{\ar@<-1ex>_{s_X} "00"; "01"};{\ar@<1ex>^{t_X} "00"; "01"};{\ar@<-1ex>_{s_Y} "10"; "11"};{\ar@<1ex>^{t_Y} "10"; "11"};{\ar_{f_0} "01"; "11"};
\end{xy}
\hspace{1cm}
\begin{xy}
(0,0)*+{X_0}="00";(20,0)*+{Y_0}="10";(0,-15)*+{X_1}="01";(20,-15)*+{Y_1}="11";
{\ar^{f_0} "00"; "10"};{\ar_{i_X} "00"; "01"};{\ar^{i_Y} "10"; "11"};{\ar_{f_1} "01"; "11"};
{\ar@{=>}^{f_i}(13,-5)*+{};(7,-10)*+{}};
\end{xy}
\end{equation*}
but a natural transformation as depicted on the right above, whose components become identities on composing with $t_Y$ and $s_Y$ respectively.  There is also a natural transformation comparing the composition functors for the internal categories.  The point is that, in this case, one needs to keep track of the fact that the source and target maps $s,t:1 \rightrightarrows 0$ in the corresponding \emph{theory} are tight, whilst $i:0 \rightsquigarrow 1$ is loose.  Whilst double categories and their pseudo and lax maps cannot, as far as we can tell, be captured using $\Cat$-enriched theories they can be captured using $\F$-theories.  We will discuss these issues further in the sequel.


\section{Appendix}\label{appendix}

Recall that $f:A \to B$ is an isofibration if given $g:X \to A$, $h:X \to B$ and $\alpha:f \circ g \cong h$ there exists $\alpha^{\prime}:g \cong h^{\prime}$ such that $f \circ \alpha^{\prime} = \alpha$.  A cleavage for an isofibration consists of a choice, for each $(g,\alpha,h)$, of a pair $(\alpha^{\prime},h^{\prime})$ as above such that these choices are \emph{natural in $X$}.  The cleavage is said to be normal if whenever $\alpha$ is an identity $2$-cell so is $\alpha^{\prime}$ and a normal isofibration \cite{Garner2009Two-dimensional} is, by definition, a morphism supporting a normal cleavage.  Clearly each discrete isofibration is normal.

Consider the pseudolimit $Pf$ of $f$ as below left.  
\begin{equation*}\label{eq:normal}
\xy
(0,0)*+{Pf}="a0"; (-15,-15)*+{A}="b0";(15,-15)*+{B}="c0";
{\ar_{p_{f}} "a0"; "b0"}; 
{\ar^{q_{f}} "a0"; "c0"}; 
{\ar_{f} "b0"; "c0"}; 
(1,-9)*{{\cong}_{\lambda_{f}}};
\endxy
\hspace{1cm}
\xy
(0,0)*+{A}="a0"; (15,0)*+{Pf}="b0";(30,0)*+{A}="c0";(15,-15)*+{B}="d0";
{\ar^{r_{f}} "a0"; "b0"}; 
{\ar^{x} "b0"; "c0"}; 
{\ar_{f} "a0"; "d0"}; 
{\ar_{p_{f}} "b0"; "d0"}; 
{\ar^{f} "c0"; "d0"}; 
\endxy
\end{equation*}
It is an easy exercise to see that $f$ admits the structure of a normal isofibration just when $r_{f}:A \to Pf$ admits a section $x$ in the slice over $B$, as depicted above right.

\begin{Prop}\label{prop:normal}
For a $2$-category $\C$ the following are equivalent.
\begin{enumerate}
\item $\C$ has flexible limits.
\item $\C$ has products, powers, splittings of idempotents and pullbacks of discrete isofibrations.
\item $\C$ has products, powers, splittings of idempotents and pullbacks of normal isofibrations.
\end{enumerate}
\end{Prop}
\begin{proof}
Certainly $(3 \implies 2)$.  Let us show that $(2 \implies 1)$.  For this, it suffices to show that $\C$ admits inserters and equifiers.  The inserter of $f,g:A \rightrightarrows B$ is the pullback below left
\begin{equation*}
\xymatrix{
Ins(f,g)\ar[d] \ar[r] & B^{\atwo}\ar[d]^{B^{j}} \\
A \ar[r]_{\langle f,g\rangle} & B^{2}
}
\hspace{1cm}
\xymatrix{
Eq(\alpha,\beta) \ar[d] \ar[r] & B^{D_1}\ar[d]^{B^{\nabla}} \\
A \ar[r]_{\langle \alpha,\beta \rangle} & B^{P_1}
}
\end{equation*}
in which $j:2 \to \atwo$ is the boundary inclusion.  Since this is bijective on objects it is a discrete isocofibration, whence the right vertical is a discrete isofibration.  The equifier of $\alpha,\beta \in \C(A,B)(f,g)$ is the pullback above right in which $\nabla:P_1 \to D_1$ is the codiagonal identifying the parallel pair --- since this is bijective on objects the same reasoning applies.

To prove $(1 \implies 3)$ we must show that a $2$-category with flexible limits admits pullbacks of normal isofibrations.  Now let $f:B \to C$ be a normal isofibration and consider $g:A \to C$ arbitrary.  To form the pullback of $f$ and $g$ we first form the pseudopullback of $f$ and $g$,  which is equally the \emph{pullback} below.
\begin{equation*}
\begin{xy}
(0,0)*+{Ps(f,g)}="00";(20,0)*+{Pf}="10";(0,-15)*+{C}="01";(20,-15)*+{B}="11";
{\ar^-{p} "00"; "10"};{\ar_{q} "00"; "01"};{\ar^{p_f} "10"; "11"};{\ar_{g} "01"; "11"};
\end{xy}
\end{equation*}
Since $f$ is a normal isofibration, with $x:Pf \to A$ as in \eqref{eq:normal}, we obtain an idempotent $t=x \circ r_f$ on $p_f:Pf \to B$ in $\C/B$.  Since the idempotent lives in $\C/B$ it pulls back to an idempotent $t^{\prime}$ on $p:Ps(f,g) \to C$.  Since pullback along $f$, in so far as it exists, preserve split idempotents, the idempotent splitting of $t^{\prime}$ is equally the pullback of $f$ along $g$.
\end{proof}

%


\begin{thebibliography}{10}

\bibitem{Adamek1994Locally}
{\sc Ad{\'a}mek, J. and Rosick{\'y}, J.}
\newblock {\em Locally Presentable and Accessible Categories}, vol.~189 of {\em
  London Mathematical Society Lecture Note Series}.
\newblock Cambridge University Press, 1994.

\bibitem{Adamek1994Weakly}
{\sc Ad{\'a}mek, J. and Rosick{\'y}, J.}
\newblock Weakly locally presentable categories.
\newblock {\em Cahiers de Topologie et Geom{\'e}trie Diff{\'e}rentielle
  Cat{\'e}goriques 35}, 3 (1994), 179--186.
  
\bibitem{Beke2000Sheafifiable}
{\sc Beke, T.}
\newblock{ Sheafifiable homotopy model categories.}
\newblock{ \em Math. Proc. Cambridge Philos. Soc. 129} (2000), no. 3, 447--475.

\bibitem{Bird1984Limits}
{\sc Bird, G.}
\newblock{\em Limits in 2-categories of locally-presented categories.}
\newblock PhD thesis, University of Sydney, 1984.
  
\bibitem{Bird1989Flexible}
{\sc Bird, G., Kelly, G.~M., Power, A.~J. and Street, R.}
\newblock Flexible limits for $2$-categories.
\newblock{\em Journal of Pure and Applied Algebra 61}, 1 (1989), 1--27.

\bibitem{Blackwell1989Two-dimensional}
{\sc Blackwell, R., Kelly, G.~M., and Power, A.~J.}
\newblock Two-dimensional monad theory.
\newblock {\em Journal of Pure and Applied Algebra 59}, 1 (1989), 1--41.

\bibitem{BorceuxQuinteiro}
{\sc Borceux, F. and Quinteriro, C.}
\newblock Enriched accessible categories.
\newblock {\em Bull. Austral. Math. Soc.}, 54(3):489--501, 1996.

\bibitem{Borceux2007Purity}
{\sc Borceux, F. and Rosick{\'y}, J.}
\newblock Purity in algebra.
\newblock {\em J. Algebra univers. 56}, 1 (2007), 17--35.

\bibitem{Bourke2013On}
{\sc Bourke, J. and Garner, R.}
\newblock On semiflexible, flexible and pie algebras.
\newblock{ \em Journal of Pure and Applied Algebra 217}, 2 (2013), 293--321.

\bibitem{Bourke2017Equipping}
{\sc Bourke, John.}
\newblock Equipping weak equivalences with algebraic structure.
\newblock{\em Mathematische Zeitschrift 294(3)} (2020), 995-1019.

\bibitem{Bourke2018Iterated}
{\sc Bourke, J.}
\newblock Iterated algebraic injectivity and the faithfulness conjecture. (2018)
\newblock Arxiv Preprint: \url{https://arxiv.org/abs/1811.09532}.

\bibitem{Bourke2020Adjoint}
{\sc Bourke, J., Lack, S. and  Vok\v{r}\'{i}nek, L.}
\newblock Adjoint functor theorems for homotopically enriched categories.
\newblock \emph{In preparation.}

\bibitem{Gabriel1971Lokal}
{\sc Gabriel, P., and Ulmer, F.}
\newblock {\em Lokal pr\"asentierbare {K}ategorien}, vol.~221 of {\em Lecture
  Notes in Mathematics}.
\newblock Springer-Verlag, 1971.

\bibitem{Garner2009Two-dimensional}
{\sc Garner, R}
\newblock Two-dimensional models of type theory.
\newblock{\em Mathematical Structures in Computer Science 19} 4 (2009), 687--736.

\bibitem{Garner2011Understanding}
{\sc Garner, R.}
\newblock Understanding the small object argument.
\newblock {\em Applied Categorical Structures 17}, 3 (2009), 247--285.
  
\bibitem{Gray1974Formal}
{\sc Gray, J.~W.}
\newblock Formal category theory: adjointness for 2-categories.
\newblock {vol.~391 of Lecture Notes in Mathematics  Springer, 1974.}

\bibitem{Hovey1999Model}
Hovey, M.
\newblock {\em Model categories}, volume~63 of {\em Mathematical Surveys and
  Monographs}.
\newblock American Mathematical Society, Providence, RI, 1999.
  
\bibitem{Joyal1993Pullbacks}
{\sc Joyal, A., and Street, R.}
\newblock Pullbacks equivalent to pseudopullbacks.
\newblock {\em Cahiers de Topologie et Geom{\'e}trie Diff{\'e}rentielle
  Cat{\'e}goriques 34}, 2 (1993), 153--156.
  
\bibitem{Kelly1972Doctrinal}
{\sc Kelly, G.~M.}
\newblock Doctrinal adjunction.
\newblock{In {\em Category Seminar, Sydney}, {\em Lecture Notes in Math. vol. 420},
Springer (1974) 257--280.}
  
  \bibitem{Kelly1982Basic}
{\sc Kelly, G.~M.}
\newblock {\em Basic concepts of enriched category theory}, vol.~64 of {\em
  London Mathematical Society Lecture Note Series}.
\newblock Cambridge University Press, 1982.
\newblock Republished as: \textit{Reprints in Theory and Applications of
  Categories 10} (2005).
  
\bibitem{Kelly1982Structures}
{\sc Kelly, G.~M.}
\newblock Structures defined by finite limits in the enriched context. {I}.
\newblock {\em Cahiers Topologie G\'eom. Diff\'erentielle 23},1 (1982) 3--42.
  
 \bibitem{Kelly2000On}
{\sc Kelly, G.~M., and Lack, S.}
\newblock On the monadicity of categories with chosen colimits.
\newblock{\em Theory and Applications of Categories} Vol. 7, 2000, No. 7, 148-170.

\bibitem{Kelly2005Notes}
{\sc Kelly, G.~M. and Schmitt, V.}
\newblock Notes on enriched categories with colimits of some class.
\newblock{\em Theory and Applications of Categories 14}, 17 (2005) 399--423.

\bibitem{Kelly1972Review}
{\sc Kelly, G.~M. and Street, R.}
\newblock Review of the elements of $2$-categories.
\newblock{In {\em Category Seminar, Sydney}, {\em Lecture Notes in Math. vol. 420},
Springer (1974) 75--103.}

\bibitem{Lack2002Codescent}
{\sc Lack, S.}
\newblock Codescent objects and coherence.
\newblock {\em Journal of Pure and Applied Algebra 175} (2002) 223--241.

\bibitem{Lack2002A-quillen}
{\sc Lack, S.}
\newblock A Quillen model structure for 2-categories.
\newblock {\em K-Theory}, 26 (2002),
171--205.

\bibitem{Lack2007Lawvere}
{\sc Lack, S., and Power, J.}
\newblock Lawvere 2-theories.
\newblock Presented at CT2007, Coimbra, 2007.

\bibitem{Lack2007Homotopy-theoretic}
{\sc Lack, S.}
\newblock Homotopy-theoretic aspects of 2-monads.
\newblock {\em Journal of Homotopy and Related Structures 7}, 2 (2007),
  229--260.
  
 \bibitem{Lack2007Icons}
{\sc Lack, S.}
\newblock Icons.
\newblock {\em Applied Categorical Structures 18}, 3 (2010), 289--307.

\bibitem{Lack2012Enriched}
{\sc Lack, S. and Rosick{\'y}, J.}
\newblock Enriched weakness.
\newblock{\em Journal of Pure and Applied Algebra 216}, 8-9 (2012), 1807--1822.

\bibitem{Lack2012Enhanced}
{\sc Lack, S. and Shulman, M.}
\newblock Enhanced 2-categories and limits for lax morphisms.
\newblock{\em Advances in Mathematics 229}, 1 (2012), 294--356.
  
\bibitem{Loregian2019Categorical}
{\sc Loregian, F. and Riehl, E.}
\newblock Categorical notions of fibration.
\newblock{\em Expositiones Mathematicae}, 2019, \url{https://doi.org/10.1016/j.exmath.2019.02.004}.

\bibitem{Lurie}
{\sc Lurie, J.}
\newblock {\em Higher topos theory}, volume 170 of {\em Annals of Mathematics
  Studies}.
\newblock Princeton University Press, Princeton, NJ, 2009.
  
 \bibitem{CWM}
{\sc MacLane, S.}
\newblock{\em Categories for the working mathematician, Second Edition}, vol. 5 of {\em Graduate Texts in Mathematics}.
\newblock Springer-Verlag, New York, 1998.

 \bibitem{MacLane1963}
{\sc MacLane, S.}
\newblock Natural associativity and commutativity.
\newblock{ \em Rice Univ. Stud. 49}, (1963) 28--46.

 \bibitem{MacLane1982}
{\sc MacLane, S. and Par{\'e}, R.}
\newblock Coherence for bicategories and indexed categories.
\newblock{\em Journal of Pure and Applied Algebra 37}, (1985), 59--80.
  
\bibitem{Makkai1997Sketches1}
{\sc Makkai, M.}
\newblock Generalized sketches as a framework for completeness theorems. Part I.
\newblock{\em Journal of Pure and Applied Algebra 115}, 1 (1997), 49--79.

\bibitem{Makkai1997Sketches2}
{\sc Makkai, M.}
\newblock Generalized sketches as a framework for completeness theorems. Part II.
\newblock{\em Journal of Pure and Applied Algebra 115}, 2 (1997), 197--212.
  
\bibitem{Makkai1989Accessible}
{\sc Makkai, M., and Par{\'e}, R.}
\newblock {\em Accessible categories: the foundations of categorical model
  theory}, vol.~104 of {\em Contemporary Mathematics}.
\newblock American Mathematical Society, 1989.

\bibitem{Nikolaus2011Algebraic}
{\sc Nikolaus, T.}
\newblock{Algebraic models for higher categories}
\newblock{\em  Indag. Math. (N.S.)} 21 (2011), no. 1--2, 52--75.

\bibitem{Power1991A-characterization}
{\sc Power, A. and Robinson, E.}
\newblock { A characterization of pie limits}, 
\newblock{ \em Math. Proc. Cam. Phil. Soc. 100}, (1991), 33--47. 

\bibitem{Malts2010}
{\sc Maltsiniotis, G.}
\newblock{Grothendieck {$\omega$}-groupoids, and still another definition of {$\omega$}-categories.}
\newblock {Arxiv Preprint 2010, \url{https://arxiv.org/abs/1009.2331}.}

\bibitem{Riehl2015}
{\sc Riehl, E. and Verity, D.}
\newblock The 2-category theory of quasi-categories.
\newblock {\em Advances in Mathematics 280}, (2015), 549--642.

\bibitem{Riehl2017Fibrations}
{\sc Riehl, E. and Verity, D.}
\newblock Fibrations and Yoneda's lemma in an $\infty$-cosmos.
\newblock {\em Journal of Pure and Applied Algebra 221}, 3 (2017) 499--564.

\bibitem{Riehl2019Elements}
{\sc Riehl, E. and Verity, D.}
\newblock {\em Elements of $\infty$-category theory.}
\newblock Book draft downloaded 31 January 2020 from
\url{http://www.math.jhu.edu/~eriehl/elements.pdf}. 

\bibitem{Rosicky2007On}
{\sc Rosick{\'y}, J.}
\newblock On combinatorial model categories.
\newblock {\em Applied Categorical Structures 17}, (2009), 303--316.

\bibitem{Street1972Fibrations}
{\sc Street, R.}
\newblock Fibrations and Yoneda's lemma in a $2$-category, 
\newblock{In {\em Category Seminar, Sydney}, {\em Lecture Notes in Math. vol. 420},
Springer (1974) 104--133.}

\bibitem{Schanuel1986Thefree}
{\sc Schanuel, S. and Street, R.}
\newblock The free adjunction.
\newblock {\em Cahiers de Topologie et Geom{\'e}trie Diff{\'e}rentielle
  Cat{\'e}goriques 27}, 1 (1986), 81--83.

\bibitem{Weber}
{\sc Weber, M.}
\newblock Yoneda structures from $2$-toposes.
\newblock{\em Applied Categorical Structures 15}, (2007), 259--323.

\end{thebibliography}
\end{document}